\documentclass[a4paper,10pt,oneside]{article}

\usepackage{graphicx}
\usepackage[all]{xy}
\usepackage[centertags]{amsmath}
\usepackage{latexsym}
\usepackage{xcolor}
\usepackage{amsfonts}
\usepackage{tikz-cd}
\usepackage{pst-node}
\usepackage{amssymb,amsthm}
\usepackage[english]{babel}

\usepackage{newlfont}

\usepackage{indentfirst}
\usepackage[english]{babel}
\usepackage[latin1]{inputenc}
\usepackage{newlfont}

\usepackage{indentfirst}
\usepackage[english]{babel}
\usepackage[latin1]{inputenc}
\usepackage{eufrak}
\usepackage{hyperref}

\newtheorem{remark}{Remark}[section]

\theoremstyle{plain}
\newtheorem{lemma}{Lemma}[section]

\newtheorem{proposition}{Proposition}[section]
\newtheorem{theorem}{Theorem}
\newtheorem{definition}{Definition}[section]
\newtheorem{corollary}{Corollary}[section]

\newtheorem{example}{Example}[section]

\newcommand{\G}{\mathcal{G}}

\newcommand{\sspan}{\hbox{\rm span}}

\newcommand{\beqn}{\begin{eqnarray}}
\newcommand{\eeqn}{\end{eqnarray}}
\newcommand{\beq}{\begin{eqnarray}}
\newcommand{\eeq}{\end{eqnarray}}

\newcommand{\bpro}{\begin{proposition}}

\newcommand{\epro}{\end{proposition}}
\newcommand{\blem}{\begin{lemma}}
\newcommand{\elem}{\end{lemma}}
\newcommand{\bdfn}{\begin{definition}}
\newcommand{\edfn}{\end{definition}}
\newcommand{\bcor}{\begin{corollary}}
\newcommand{\ecor}{\end{corollary}}
\newcommand{\bthm}{\begin{theorem}}
\newcommand{\ethm}{\end{theorem}}
\newcommand{\bex}{\begin{example}}
\newcommand{\eex}{\end{example}}
\newcommand{\brmq}{\begin{remark}}
\newcommand{\ermq}{\end{remark}}
\newcommand{\benum}{\begin{enumerate}}
\newcommand{\eenum}{\end{enumerate}}
\newcommand{\bitem}{\begin{itemize}}
\newcommand{\eitem}{\end{itemize}}

\theoremstyle{plain}

\usepackage[latin1]{inputenc}

\title
      {On the classification of 2-solvable Frobenius Lie algebras}

\author{ Andr\'e Diatta$^{( 1)}$\footnote{ \footnotesize \noindent (1) Aix-Marseille Univ, CNRS, Centrale Marseille, Institut Fresnel, 13013 Marseille, France.
\newline $^*$ Corresponding author: andre.diatta@fresnel.fr; andrediatta@gmail.com.}
;  Bakary Manga$^{( 2)}$
  and  
Ameth Mbaye$^{( 2)}$\footnote{\footnotesize \noindent(2) D\'epartement de Math\'ematiques et Informatique,
Universit\'e Cheikh Anta Diop de Dakar,
BP 5005 Dakar-Fann, Dakar, S\'en\'egal. Email: bakary.manga@ucad.edu.sn;  ameth3.mbaye@ucad.edu.sn
} 
\thanks{AD and BM acknowledge financial support from the Simons Foundation through the NLAGA project.}
  }

\begin{document}
\maketitle



\date{}

\maketitle

\begin{abstract}
We discuss the classification of $2$-solvable Frobenius Lie algebras. We prove that every $2$-solvable Frobenius Lie algebra splits as a semidirect sum of an $n$-dimensional vector space V and an $n$-dimensional maximal Abelian subalgebra (MASA) of  the full space of endomorphisms of $V.$
We supply a complete classification of $2$-solvable Frobenius Lie algebras corresponding to nonderogatory endomorphisms,
 as well as those given by
maximal Abelian nilpotent subalgebras (MANS)  
of class 2, hence of Kravchuk signature $(n-1,0,1)$. In low dimensions, we classify all  2-solvable Frobenius Lie algebras in general up to  dimension $8$. 
We correct and  complete the classification  list of MASAs of $\mathfrak{sl}(4,\mathbb R)$ by Winternitz and 
Zassenhaus. As a biproduct, we give a simple proof  that every nonderogatory endormorphism of a real vector space admits a Jordan form and also provide a new characterization of Cartan subalgebras of $\mathfrak{sl}(n,\mathbb R)$.
\end{abstract}


\section{Introduction}
Throughout this work, $V$ stands for a real vector space of dimension $n$. 
 We are mainly interested in the classification of 2-solvable Frobenius Lie algebras. It turns out that such a classification is closely related to that of $n$-dimensional maximal Abelian subalgebras (MASAs) of the space $\mathfrak{gl} (V)$ of linear operators (endomorphisms) on $V.$
The sudy of MASAs traces back to Frobenius and Schur, it has been a vibrant  and vivid subject these last decades, in relation with several subjects in mathematics and physics such as the classification of Lie algebras and the study of dynamical systems  
(\cite{de-Olmo-Rodriguez-Wintemitz-Zassenhaus, dixmier,   harish-chandra, jacobson-schur,  kostant, schur,  suprunenko, sugiura, winternitz-zassenhaus, winternitz}). Our interest in $2$-solvable Frobenius Lie algebras is mainly related to their applications in symplectic (\cite{bordemann-medina,  diatta-manga,  diatta-medina, lichnerowicz-medina})   and information geometry (\cite{mbaye-these}), as well as to the investigation of Lie algebras in general. Frobenius Lie algebras have gained popularity mainly since the work of Ooms (\cite{ooms, ooms-primitve-ideals}).   Among other results, a characterization of 2-solvable Frobenius Lie algebras  is given in \cite{ooms} (Theorem 4.1, Section 4, p. 43). On the other hand, the family of 2-solvable Frobenius Lie algebras with an Abelian nilradical is studied in \cite{alvarez-abelian-nilradical}. The main results of the present paper are the following.

\noindent
$\bullet$ We show that, if   $\mathfrak{B}\subset \mathfrak{gl} (V)$ is Abelian  and the semidirect sum $\mathfrak{B}\ltimes V$ is a Frobenius Lie algebra, then $\mathfrak{B}$ is a MASA of  $\mathfrak{gl}(V)$ (Theorem  \ref{thm:structure-2-solvable-Frobenius}). Moreover, every 2-solvable Frobenius Lie algebra is of that form. Two 2-solvable Frobenius Lie algebras $\mathfrak B_1\ltimes V$  and  $\mathfrak B_2\ltimes V$  are isomorphic if and only if   there exists  some  $\phi$ in the group GL$(V)$ of  invertible linear operators of $V$, such that  $\mathfrak B_2=\phi \mathfrak B_1 \phi^{-1} $ (Proposition  \ref{lemma:conjugaison-MASAs}).  

\noindent
$\bullet$ In  Theorem \ref{thm:class2-kravchuk-(n-1)-0-1}, we classify  2-solvable Frobenius Lie algebras $\mathfrak{B}\ltimes V$, where  $\mathfrak{B}\cap \mathfrak{sl}(V) $ is a maximal Abelian nilpotent subalgebra (MANS)  of class 2 and hence of  Kravchuk signature $(n-1,0,1)$. Here $\mathfrak{sl}(V) \subset \mathfrak{gl}(V) $ stands for the subspace of traceless linear operators. 

\noindent
$\bullet$ 
We  fully classify, in Theorem \ref{thm:classification}, the family of $2$-solvable
 Frobenius Lie algebras  $\mathfrak{B}\ltimes V$, where $\mathfrak B$ is the algebra $\mathbb K[\phi]$ of the polynomials in some nonderogatory  $\phi\in\mathfrak{gl}(V)$. We will simply write  $\mathcal G_\phi =\mathbb K[\phi] \ltimes V$ instead of $\mathfrak{B} \ltimes V$.

\noindent
$\bullet$ We supply a full classification list of all 2-solvable Frobenius Lie algebras in general, up to and including dimension $8$ (Theorem \ref{thm:classification-D6}).

\noindent
$\bullet$ Theorems \ref{thm:cartansubalgebras}, \ref{prop:abeliannilrad} give a new characterization of Cartan subalgebras of $\mathfrak{sl}(V).$  

\noindent
$\bullet$ We correct the classification  list of MASAs of $\mathfrak{sl}(4,\mathbb R)$ of  \cite{winternitz-zassenhaus} and further complete it with a missing item. 
 
\noindent $\bullet$  We prove that every nonderogatory real matrix admits a Jordan form (Theorem \ref{thm:Jordanization}).  

The paper is organized as follows. Section \ref{sect:preliminaries},  is devoted to some preliminaries and notations.   In Section \ref{sect:2-step-solvable-Frobenius}, we  discuss the classification of $2$-solvable Frobenius Lie algebras, in general.  General examples of MASAs and 2-solvable Frobenius Lie algebras are discussed in details and a new characterization of Cartan subalgebras of $\mathfrak{sl}(n,\mathbb R)$ is given.
We supply a full classification of $2$-solvable Frobenius Lie algebras  given by nonderogatory maps in Section \ref{chap:classification-nonderogatory}.
 We  fully classify all $2$-solvable Frobenius Lie algebras up to dimension $8$ in Section \ref{sect:6D} and  furthermore, we correct  and complete the classification  list of MASAs of $\mathfrak{sl}(4,\mathbb R)$ of  \cite{winternitz-zassenhaus} with a missing item. 
 The paper ends by 
a simple direct proof of the Jordanization of nonderogatory real  matrices in Appendix (Section \ref{sect:proof-of-Jordanization}). 

\section{Preliminaries}\label{sect:preliminaries}
Throughout this work, if $\mathfrak{F}$ is a vector space, $\mathfrak{F}^*$ will stand for its  (linear) dual, $\mathfrak{gl} (\mathfrak{F})$ the space of linear maps $\psi: \mathfrak{F}\to \mathfrak{F}$ and GL($\mathfrak{F}$) the subspace of  $\mathfrak{gl} (\mathfrak{F})$ consisting of invertible maps. The symmetric bilinear form $\langle,\rangle$ stands for the 
duality pairing between vectors and linear forms: 
 $\langle u,f \rangle: = f(u),$ for $u\in\mathfrak{F},$ $f\in\mathfrak{F}^*$. If $(e_j)$ is a basis of $\mathfrak{F}$, we denote by $(e_j^*)$ its dual basis.
Let $\mathfrak{B}$ be a Lie subalgebra of $\mathfrak{gl}(V)$, where $V$ is a vector space. We consider the natural action $ \mathfrak{B} \times  V \to  V,$ \; $(a,x)\mapsto \rho(a)x:= ax$.  For any $(a,f)\in  \mathfrak{B} \times V^*$, let  $\rho^*(a)f$ be the element of $V^*$  defined on elements $x$ of $V$ by
 $\langle\rho^*(a)f, x\rangle:= -f(ax).$ We recall that $\rho^*$ is called the contragredient representation (or just action) of $\rho$. 
\begin{definition}  
By abuse of language, we say that the orbit of $\alpha\in V^*$, under $\rho^*$, is open, if $V^*=\{\rho^*(a)\alpha, \;\; a\in \mathfrak{B}\}$. 
We say that $\alpha$ has a trivial isotropy  if the equality $\rho^*(a)\alpha=0$ implies $a=0.$  
\end{definition}
The vector space $\mathfrak{B}\oplus V$ endowed with the Lie bracket defined, for every $a,b\in\mathfrak{B}$ and $x,y\in V$, as 
$[a,b]=[a,b]_{\mathfrak{B}}\;,\; [x,y] =0,  \text{ and }  [a,x]=\rho(a)x=ax,$
will be termed the semidirect sum of  $\mathfrak{B}$  and  $V$  via $\rho$ and denoted by $\mathfrak{B}\ltimes_\rho V$, or simply $\mathfrak{B}\ltimes  V$ if no confusion is to be made. In the present work, $\mathfrak{B}$ is Abelian,  $[,]_\mathfrak{B}=0.$

A Lie algebra  $\mathcal G$ is called a Frobenius Lie algebra if there exists $\alpha\in\mathcal G^*$, called a Frobenius functional, such that  $\partial\alpha$ is nondegenerate, 
where, for any $u,v\in\mathcal G,$   
\begin{eqnarray}\label{eq:Frobenius}\partial\alpha(u,v):=-\langle \alpha, [u,v]\rangle.\end{eqnarray}
  Every Frobenius Lie algebra is a codimension 1 subalgebra of some contact Lie algebra and could be used to construct the latter.  The converse is true under some conditions (\cite{Diatta-Contact}). 
Any Lie group $G$ with Lie algebra $\mathcal G$, has a left invariant  symplectic form $\omega^+$ with value $\omega_\epsilon^+=\partial\alpha$ at the unit (neutral) element $\epsilon$ of $G.$ Here $\mathcal G$ is identified with the tangent space to $G$ at $\epsilon.$ 
See for example \cite{diatta-manga}, \cite{diatta-medina}, for more details. 

\begin{definition}A Lie algebra $\mathcal G$  is said to be $2$-step solvable (2-solvable, for short) if its derived ideal $[\mathcal
G,\mathcal G]$ is Abelian. A Lie algebra is said to be indecomposable if it cannot be written 
as the direct sum of two of its ideals. 
\end{definition}
\begin{remark}\label{rmq:indecomposable}
Note that two direct sums  $\mathcal G_1\oplus\dots\oplus \mathcal G_k$ and $\mathcal G_1'\oplus\dots\oplus \mathcal G_k'$  of $2$-solvable Frobenius Lie algebras, are isomorphic if and only if, up to rearranging the order of the indices, each $\mathcal G_j$ is isomorphic to each $\mathcal G_j'.$ So the classification of $2$-solvable Frobenius Lie algebras boils down to that of the indecomposable ones.
\end{remark}
\begin{definition} Let $\mathcal G$ be a Lie algebra, $\mathcal A$ a subalgebra of $\mathcal G.$ The normalizer of $\mathcal A$ in $\mathcal G$ is  $\mathcal N_{\mathcal G}(\mathcal A):=\{a\in \mathcal{G}, [a,\mathcal A]\subset \mathcal A\}.$
 The centralizer of $\mathcal A$ in $\mathcal G$ is cent$(\mathcal A,\mathcal G):=\{a\in \mathcal{G}, [a, \mathcal A]=0\}.$  
We say that $\mathcal A$ is a maximal Abelian subalgebra (MASA) of $\mathcal{G}$, if $\mathcal A=$cent($\mathcal A,\mathcal G).$ Equivalently, $\mathcal A$ is Abelian and is contained in no Abelian subalgebra of higher dimension of $\mathcal{G}.$
 For $\phi\in\mathfrak{gl}(V),$ we use the convention $\phi^0=\mathbb I_{ V}=$ identity map of $V$ and $\phi^{p+1} (x)=\phi( \phi^{p}(x)),$ for every $x\in V$ and $p$ an integer. We let $\mathbb K[\phi]$ stand for the space of polynomials in $\phi$ with coefficients in $\mathbb K$ and denote by $cent(\phi,\mathfrak{gl}(V))$
the set of endomorphisms commuting with $\phi$.
Recall that $\phi$ (resp. an $n\times n$ matrix $M$) is said to be nonderogatory (or cyclic), if its characteristic polynomial $\chi_\phi$ and its minimal polynomial $\chi_{\min,\phi}$ coincide. \end{definition}

\begin{lemma}[see e.g. \cite{bordemann-medina}] \label{nonderogatorymatrices}
Let $V$ be a vector space of dimension $n$ over a field $\mathbb K$ of 
       characteristic zero  and $\phi\in\mathfrak{gl}(V)$.  
 The following assertions are equivalent:
  \item[1)]   there exists $\bar \alpha\in V^*$ such that $(\bar\alpha,\bar\alpha\circ \phi,...,
\bar\alpha\circ \phi^{n-1})$ is a basis of $V^*$,
\item[2)]  there exists $\bar x\in V$ such that $(\bar x,\phi(\bar x),..., \phi^{n-1}(\bar x))$ is a basis of
$V$,
\item[3)] $\dim(cent(\phi,\mathfrak{gl}(V)))=n$ and $cent(\phi,\mathfrak{gl}(V))$ is commutative,
\item[4)] $cent(\phi,\mathfrak{gl}(V))=\mathbb K[\phi]$,
\item[5)] the characteristic and the minimal polynomials of $\phi$
are the same,
\item[6)] In every extension $\bar{ \mathbb K}$  of $\mathbb K$
where $\phi$ admits a Jordan form, this latter has only one Jordan bloc for
each eigenvalue.
\end{lemma}
In Theorem \ref{thm:Jordanization},  we go beyond Lemma \ref{nonderogatorymatrices} to prove that, for  every $n$-dimensional real vector space $V,$ every nonderogatory $\phi\in\mathfrak{gl}(V)$ admits a Jordan form.
 Throughout this work, $E_{i,j}$ stands for the $n\times n$ matrix with zero in all entries except the $(i,j)$ entry which is equal to $1$. 
 
\section{On the structure of $2$-solvable Frobenius Lie algebras and $n$-dimensional MASAs}\label{sect:2-step-solvable-Frobenius}

\subsection{MASAs of $\mathfrak{gl}(V)$ and maximal Abelian subgroups of $GL(V)$ }\label{chap:correspondence-MASA-maximalAbeliansubgroup}

In this section, we briefly discuss the 1-1 correspondence between
MASAs of $\mathfrak{gl}(V)$ and maximal Abelian subgroups of $GL(V)$. 
Let $\mathfrak{B}$ be a MASA of $\mathfrak{gl}(V).$ Then  $\mathfrak{B}$ is a commutative  subalgebra of the associative algebra $\mathfrak{gl}(V)$,  and it contains the unit $\mathbb I_V$. Hence its (open) subset $U(\mathfrak{B})$  consisting of invertible elements is an analytic Lie group and $\dim U(\mathfrak{B})= \dim \mathfrak{B}$, (see e.g. J. P. Serre \cite{serre}, p. 103, for the general case)
 which, by definition,  is also an Abelian Lie subgroup of GL(V). Furthermore, since  $\mathfrak{B}$ is a MASA of $\mathfrak{gl}(V)$, it contains the exponential of all its elements. Thus $\exp \mathfrak{B}$ coincides with the connected component of  the unit of $U(\mathfrak{B})$.   
In fact,  $U(\mathfrak{B})$ is a maximal Abelian subgroup of $GL(V)$. Conversely, every maximal Abelian subgroup of $GL(V)$ is of that form $U(\mathfrak{B})$, for some MASA $\mathfrak{B}$ of $\mathfrak{gl}(V)$. Two such subgroups are conjugate in $GL(V)$ if and only if the corresponding MASAs are conjugate in $\mathfrak{gl}(V)$ (see \cite{suprunenko}).  So we deduce that a simple way to construct a 2-solvable Lie group with Lie algebra $\mathfrak{B}\ltimes V$ is to simply take 
$U(\mathfrak{B})\ltimes V.$

\subsection{Main result on MASAs and 2-solvable Frobenius Lie algebras}

Let  $V$ be a vector space of dimension $n$ and $\mathfrak{B}$ an  $n$-dimensional Abelian Lie subalgebra of $\mathfrak{gl}(V).$ Then  $\mathfrak{B} \ltimes V$ is a Frobenius Lie algebra if and only if   $\mathfrak{B}$ has an open orbit in $V^*$ and every 2-solvable Frobenius Lie algebra is of that form (\cite{diatta-manga-mbaye-gerstenhaber}). 
Equivalently,  the Lie group $U(\mathfrak{B})$ of invertible elements of $\mathfrak{B}$, has  some open or dense orbit in $V^*$. 
Here we prove the following.

\begin{theorem} 
\label{thm:structure-2-solvable-Frobenius}
 Let  $V$ be a vector space of dimension $n$ and $\mathfrak{B}$ an  $n$-dimensional Abelian Lie subalgebra of $\mathfrak{gl}(V).$ Suppose  $\mathfrak{B} \ltimes V$ is a Frobenius Lie algebra. Then  $\mathfrak{B}$ is a  
MASA 
of $\mathfrak{gl}(V)$. The Lie group $U(\mathfrak{B})$ of invertible elements of $\mathfrak{B}$, is a  maximal Abelian subgroup of $\text{GL}(V).$   
Conversely, every 2-solvable Frobenius Lie algebra is isomorphic to one as above.
\end{theorem}
\begin{proof} Without loss of generality, set $V=\mathbb K^n.$ 
The  Lie algebra  $\mathfrak{B} \ltimes \mathbb K^n$ is Frobenius if and only if the action $ \mathfrak{B}\times (\mathbb K^n)^* \to  (\mathbb K^n)^*$,
 $(a,f) \mapsto - f\circ a$, has an open orbit (\cite{diatta-manga-mbaye-gerstenhaber}).   Consider some $\alpha\in  (\mathbb K^n)^*$ with an open orbit. Thus, any basis $(a_1,\dots,a_n)$ of $ \mathfrak{B}$ gives rise to a basis  $( \rho^*(a_1)\alpha, \dots, \rho^*(a_n)\alpha)$ of $(\mathbb K^n)^*$ and the (linear) orbital map $Q: \mathfrak{B} \to  (\mathbb K^n)^*$, $Q(a) = \rho^* (a)\alpha$ is an isomorphism between the vector spaces  $\mathfrak{B}$ and $(\mathbb K^n)^*$. Just for convenience, we will let $(\hat e_1,\dots,\hat e_n)$ stand for the basis of $\mathbb K^n$ whose dual basis is $(\hat e_1^*:=\rho^*(a_1)\alpha, \dots, \hat e_n^*:=\rho^*(a_n)\alpha).$
Suppose $\tilde a\in\mathfrak{gl}(n,\mathbb K)$ is such that $[\tilde a, a] =0$ for any $a\in \mathfrak{B}$ and $\tilde a\neq 0$. Assume $\tilde a$ is not an element of 
 $\mathfrak{B}$, then  $ \mathfrak{\tilde B}:= \mathbb K\tilde a\oplus \mathfrak{B}$ is an $(n+1)$-dimensional Abelian subalgebra of $\mathfrak{gl}(n,\mathbb K)$. Because $\dim \mathfrak{\tilde B} = \dim   (\mathbb K^n)^*+1$, the orbital map $\tilde Q:\mathfrak{\tilde B} \to (\mathbb K^n)^*$,  also given by $\tilde Q(a) = \tilde \rho^* (a)\alpha:= -\alpha\circ a$, must have a $1$-dimensional kernel. So, there exists some 
$\tilde b=k\tilde a+a_0\neq 0$, with $k\in\mathbb K$ and $a_0\in \mathfrak{B}$ such that $\tilde \rho^* (\tilde b)\alpha= -\alpha\circ \tilde b =0.$  Since $\tilde \rho^* (a_0)\alpha= \rho^* (a_0)\alpha=Q(a_0) \neq 0$ if $a_0\neq 0$, we must have $k\neq 0$. 
We then must have $ \tilde b x=0$, for any $x\in\mathbb K^n,$ or equivalently $\tilde b =0.$ Indeed, for any $x\in\mathbb K^n$, the expression of $\hat b x$ in the above basis is $\hat b x=\sum\limits_{j=1}^n  \langle \hat e_j^*,\tilde b x\rangle \hat e_j.$ But the components $\langle \hat e_j^*,\tilde b x\rangle $ are all egal to zero for any $j=1,\dots,n$, since the following holds true
\begin{eqnarray}\langle \hat e_j^*,\tilde b x\rangle= \langle\rho^* (a_j)\alpha, \tilde b x\rangle =
- \langle\alpha, a_j \tilde b x\rangle  = - \langle\alpha, \tilde b  a_j x\rangle =  \langle\tilde \rho^* (\tilde b) \alpha,   a_j x\rangle =0. \nonumber\end{eqnarray}
Now, the equality $\tilde b=k\tilde a+a_0 =0$,  contradicts the assumption that $\tilde a$ is not in $\mathfrak{B}$. So $\tilde a$ must necessary be in $\mathfrak{B}$. This proves that $\mathfrak{B}$ is a MASA of $\mathfrak{gl}(n,\mathbb K)$. 
 Due to the correspondence between MASAs of $\mathfrak{gl}(n,\mathbb K)$ and maximal Abelian subgroups of GL($n,\mathbb K$), see \cite{suprunenko} and Section \ref{chap:correspondence-MASA-maximalAbeliansubgroup}, $\mathfrak{B}$ is a MASA if and only if $U(\mathfrak{B})$ is a maximal Abelian subgroup of GL$(n,\mathbb K).$
\end{proof}
\begin{remark}\label{rmq:nonderogatory}
From Lemma \ref{nonderogatorymatrices}, if $\phi\in  \mathfrak{gl}(V)$ is nonderogatory, then $\mathbb K[\phi]$ is an $n$-dimensional MASA of $\mathfrak{gl}(V),$ there is $\bar \alpha\in V^*$ with an open orbit for the action of $\mathbb K[\phi]$. The 2-solvable Lie algebra $\mathcal G_\phi:=\mathbb K[\phi]\ltimes V$ is a Frobenius Lie algebra (Theorem \ref{thm:structure-2-solvable-Frobenius}).  Frobenius  Lie algebras of type $\mathcal G_\phi$ are studied in details and fully classified in Section \ref{chap:classification-nonderogatory}. Obviously, if  $\phi,\psi\in \mathfrak{gl}(V)$  satisfy $\phi=P\psi P^{-1}$, for some $P\in GL(V),$ then $\mathbb K[\phi]$ and $\mathbb K[\psi]$ are conjugate,  $\mathcal G_\phi$ and $\mathcal G_\psi$  
are isomorphic  via the map  $\xi:\mathcal G_\psi \to \mathcal G_\phi$,  $\xi(a)=PaP^{-1}$, $ \xi(x)=Px$, for any $a\in \mathbb K[\psi]$ and 
$x \in V.$
\end{remark}

\begin{proposition} 
\label{lemma:conjugaison-MASAs} Let $V$ be a vector space and $\mathfrak B_1,$ $\mathfrak B_2$ two  MASAs of  $\mathfrak{gl}(V),$ such that  $\dim V=\dim\mathfrak B_1=\dim\mathfrak B_2.$ Then
$\mathfrak B_1\ltimes V$  and $\mathfrak B_2\ltimes V$  are isomorphic if and only if   $\mathfrak B_1$ and $\mathfrak B_2$ are  conjugate.
A linear map $\psi:\mathfrak B_1\ltimes V\to \mathfrak B_2\ltimes V$ is an isomorphism if and only if there exists $(\phi_\psi, x_\psi)\in GL(V)\times V$ 
such that $ \phi_\psi\circ \mathfrak{B}_1 \circ \phi_\psi^{-1}=\mathfrak{B}_2$ and for any $(a,x)\in \mathfrak B_1\ltimes V,$
$\psi (a,x) =  (\phi_\psi\circ a \circ \phi_\psi^{-1} \;,\; \phi_\psi\circ a \circ \phi_\psi^{-1}(x_\psi) +\phi_\psi(x) ).$    
 \end{proposition}

\begin{proof}Suppose $\mathcal G_1:=\mathfrak B_1\ltimes V$ and $\mathcal G_2:=\mathfrak B_2\ltimes V$ are isomorphic and $\psi:\mathcal G_1\to\mathcal G_2$ is an isomorphism.
As the derived ideal $[\mathcal G_1,\mathcal G_1]=V$ must be mapped to $[\mathcal G_2,\mathcal G_2]=V$, 
there are linear maps 
$\psi_{1,1}:\mathfrak B_1\to\mathfrak B_2$,  $\psi_{1,2}:\mathfrak B_1\to V$,  $\phi_\psi:V \to V$, with $\psi_{1,1}$ and $\phi_\psi$ invertible, such that 
 $\psi(a)=\psi_{1,1}(a) +\psi_{1,2} (a)$ and $\psi(x)=\phi_\psi(x)$ for any $a\in\mathfrak B_1$, $x\in V$. 
We deduce the equality $\psi_{1,1}(a)=\phi_\psi \circ a \circ \phi_\psi^{-1}$ from the following 
\beqn 
\phi_\psi(ax)=\psi([a,x])= [\psi_{1,1} (a)+\psi_{1,2}(a),\phi_\psi(x)] = \Big(\psi_{1,1} (a)\circ\phi_\psi\Big)(x). \nonumber
\eeqn 

In particular $\phi_\psi \circ a \circ \phi_\psi^{-1}\in\mathfrak B_2$, for any $a\in\mathfrak B_1$, equivalently $\phi_\psi \circ \mathfrak B_1 \circ \phi_\psi^{-1}= \mathfrak B_2$. 
Now set  $x_\psi:=\psi_{1,2}(b)$, with $b=\mathbb I_{V} \in \mathfrak B_1$. We have  $\psi_{1,1} (\mathbb I_{V}) =\mathbb I_{V}$, since $\psi$ is an isomorphism. The equalities
\beqn 
0=\psi([a,b])= [\psi (a),\psi(b))] 
=\psi_{1,1} (a)\psi_{1,2}(b)- \psi_{1,1} (b) \psi_{1,2}(a)\nonumber
\eeqn
 yield $\psi_{1,2} (a)=\phi_\psi\circ a\circ\phi_\psi^{-1}x_\psi,$
 for any $a\in\mathfrak B_1.$ 
Conversely, suppose there exists some $\phi\in GL(V)$ such that $\mathfrak B_2=\phi\mathfrak B_1\phi^{-1}$. Then  the map $\psi:\mathcal G_1\to\mathcal G_2$  given for any $a\in\mathfrak{B}_1$ and $x\in V$ by  $\psi (a+x) =  \phi\circ a \circ \phi^{-1}+ \phi\circ a \circ \phi^{-1}(x_0) +\phi(x)  $ is an isomorphism, 
for any $ x_0\in V.$
\end{proof}

\subsection{Different types of MASAs} 

A MASA of $\mathfrak{gl}(V)$ is said to be decomposable, if $V$ can be written as the direct sum of subspaces which are all preserved by the MASA. Otherwise, the MASA is said to be indecomposable.
Every MASA  of $\mathfrak{gl}(V)$ is a direct sum of $\mathbb K\mathbb I_{V}$ and a  MASA of $\mathfrak{sl}(V)$. Without loss of generality, we will sometimes simply take $V=\mathbb K^n.$ 
A MASA of $\mathfrak{sl}(n,\mathbb R)$ is decomposable into a direct sum of indecomposable 
MASAs. The latter can either be absolutely indecomposable
(AID), or indecomposable but not absolutely indecomposable (ID $\&$ NAID). 
The NAID ones  become decomposable after complexification. See e.g. \cite{ndogmo-winternitz}.
Let us recall that a set of matrices is
decomposable if and only if  it commutes with some idempotent $Z\neq \mathbb I_{\mathbb K^n}=:\mathbb I_n$.
Further recall that, an idempotent of a $\mathbb K$-algebra, is an element $Z$ satisfying $Z^2=Z.$
The NAID MASAs of $\mathfrak{sl}(n, \mathbb R)$ are those MASAs all of whose elements commute with $M_s,$ given in (\ref{Ms-Mn-M01}). Upon complexification, $M_s$ becomes $M_{\mathbb C}:=$diag($i\mathbb I_{\frac{n}{2}}$, $-i\mathbb I_{\frac{n}{2}}$), so that $Z:=\frac{1}{2}(\mathbb I_{n}+ iM_{\mathbb C})$ and $Z':=\frac{1}{2}(\mathbb I_{n}- iM_{\mathbb C})$ are idempotents that commute with the elements of the NAID MASAs. 
They
exist only for $n$ even and have the form 
$\mathbb RM_s\oplus$MASA($\mathfrak{sl}(\frac{n}{2}, \mathbb C)$),  and  in our case, we only need 
MASAs of $\mathfrak{sl}(\frac{n}{2}, \mathbb C)$ of real dimension $n-2.$ An indecomposable MASA of $\mathfrak{sl}(n, \mathbb C)$ is always a maximal Abelian nilpotent subalgebra (MANS). 
The AID MASAs are MANS and they remain indecomposable after field extension.  
A MANS is represented by nilpotent matrices in any finite-dimensional
representation.  Following \cite{de-Olmo-Rodriguez-Wintemitz-Zassenhaus}, a MANS $\mathcal A$ of   $\mathfrak{sl}(n, \mathbb K),$ $\mathbb K=\mathbb R$ or $\mathbb C,$ is characterized by a Kravchuk signature $(\nu,m,\mu),$ where $1\leq \nu=n-\dim_{\mathbb K}\mathcal A\mathbb K^n$,  
$1\leq \mu=\dim\ker\mathcal{A}$, where $\ker\mathcal{A}:= \{x\in\mathbb K^n, ax=0, \; \forall a\in\mathcal A\}$  and  $m=n-\mu-\nu.$

\begin{definition}We say that $\mathcal A$ is of class $p\ge 2$ if $a_1a_2\cdots a_p=0$ for any  p elements $a_1,\dots,a_p\in\mathcal A$ and there exists some  $a\in\mathcal A$ such that $a^{p-1}\neq 0.$ We simply write $\mathcal A^{p-1}\neq 0$ and  $\mathcal A^p=0.$ Set Im($\mathcal A^k$):$=\mathcal A^k\mathbb K^n,$ $k\ge 1.$\end{definition}

\subsection{Classification of 2-solvable Frobenius Lie algebras corresponding  to class 2 MANS}

In this section, we treat one of the extreme cases, namely the case of MANS of class 2. The other extreme case, the class $n$, will be treated together with the more general case of  MASAs given by nonderogatory maps, in Section \ref{chap:classification-nonderogatory}.
\begin{lemma}\label{2-nilp-MASA}
Let $\mathcal A$ be an $(n-1)$-dimensional MASA of $\mathfrak{sl}(n,\mathbb R)$ such that $\mathcal A^{p+1}=0$ and $\mathcal A^{p}\neq 0,$ for some $p\ge 1.$ If $\mathcal G:=(\mathbb R \mathbb I_{\mathbb R^n}\oplus \mathcal A)\ltimes \mathbb R^n$ is a Frobenius Lie algebra, then $\dim$\;Im($\mathcal A^{p}$)$=1.$
In particular if $\mathcal A^2=0$  and $(\mathbb R \mathbb I_{\mathbb R^n}\oplus \mathcal A)\ltimes \mathbb R^n$ is Frobenius, then $\nu=n-1$ and $\dim Im(\mathcal A)=1$.
\end{lemma}
\begin{proof}
  If $\dim$\;Im($\mathcal A^{p}$)$\ge 2$, then for any $\alpha\in \mathcal G^*$, we  would  have $Im(\mathcal A^{p})\cap \ker\alpha\neq 0.$ Choose $x\in Im(\mathcal A^{p})\cap \ker\alpha$ such that $x\neq 0.$  As any $v\in \mathcal G$   is  of the form $v=k\mathbb I_{\mathbb R^n} + a + y,$ where $k\in\mathbb R$, $a\in\mathcal A$, $y\in\mathbb R^n$, we would have $\partial \alpha(x,v)=\langle \alpha, kx+ax \rangle =0,$  $\forall v\in\mathcal G.$ This  would mean that $\partial \alpha$ is degenerate, for any $\alpha\in \mathcal G^*.$ 
  Thus $\mathcal G$ would not be Frobenius.
In particular if $p=1$ and $(\mathbb R\mathbb I_{\mathbb R^n}\oplus \mathcal A)\ltimes \mathbb R^n$ is Frobenius, then $\dim Im(\mathcal A)=1$ and hence $\nu=n-1.$
\end{proof}
Now we are ready to prove the following classification theorem. 
\begin{theorem}\label{thm:class2-kravchuk-(n-1)-0-1} Up to conjugation,  $\mathcal A:=  \mathcal A_{n,1}$  as in Example  \ref{example-general-nD}, is  the unique MANS of class $2$, of $\mathfrak{sl}(n,\mathbb R)$,  such that $(\mathbb R\mathbb I_{\mathbb R^n}\oplus \mathcal A)\ltimes \mathbb R^n$ is  Frobenius. Equivalently, it is the unique MANS  of $\mathfrak{sl}(n,\mathbb R)$ with  Kravchuk signature $(n-1,0,1)$  such that   $(\mathbb R\mathbb I_{\mathbb R^n}\oplus \mathcal A)\ltimes \mathbb R^n$ is  Frobenius. In other words, 
 $ \mathcal G_{n,1}$  as in Example  \ref{example-general-nD}, is the unique 2-solvable Frobenius Lie algebra  given by a MANS of class 2.

\end{theorem}
\begin{proof}  First, note that if $\mathcal A\subset\mathfrak{sl}(n,\mathbb R)$ is an $(n-1)$-dimensional MASA such that  $\mathcal A^2=0$ and  $(\mathbb R\mathbb I_{\mathbb R^n}\oplus\mathcal A)\ltimes \mathbb R^n$ is a Frobenius Lie algebra, then from Lemma \ref{2-nilp-MASA}, $\dim(Im(\mathcal A))=1$. Hence $\nu=n-1$ and $\mu=1.$ Conversely, if $\nu=n-1$, since $\mathcal A$ is nilpotent, we necessarily have $\mathcal A^2=0.$ 
Now, under the assumption that   $\mathcal A^2=0$ and  $(\mathbb R\mathbb I_{\mathbb R^n}\oplus\mathcal A)\ltimes \mathbb R^n$ is a Frobenius Lie algebra, let $\hat e_1\in \mathbb R^n$  such that $\mathcal A\mathbb R^n=\mathbb R\hat e_1.$ We have $\mathcal A\hat e_1=0$ and  
$\ker \mathcal A=\mathbb R\hat e_1.$
Thus, in any basis of $\mathbb R^n$ whose first vector is $\hat e_1,$  the elements of $\mathcal A$ are linear combinations of $E_{1,j},$ $j=2,\dots,n$ and so
 $\mathcal A$ assumes the same form as the MASA $\mathcal A_{n,1}$  of $\mathfrak{sl}(n,\mathbb R)$ of Example  \ref{example-general-nD}. 
Consequently,  $\mathcal A$ is conjugate to  $\mathcal A_{n,1}$ in $\mathfrak{gl}(n,\mathbb R)$ and   $(\mathbb R\mathbb I_{\mathbb R^n}\oplus \mathcal A)\ltimes\mathbb R^n$ is isomorphic to the Lie algebra $\mathcal G_{n,1}$, see  Example  \ref{example-general-nD}). 
 \end{proof}

\subsection{Some important examples}

We propose several families of pairwise non-isomorphic (resp. non-conjugate) $2$-solvable Frobenius Lie algebras (resp. MASAs). 
\begin{example}
\label{example-general-nD}
 For any integer $n\ge 2$, we construct below a family of ($n-1$) 2-solvable Frobenius Lie algebras $\mathcal G_{n,p}:=\mathfrak{B}_{n,p}\ltimes \mathbb R^n$,  $p=1,\dots,n-1$. The $\mathfrak{B}_{n,p}$'s are pairwise non-conjugate commutative algebras of polynomials in $n-1, n-2,\dots,2, 1$ $n\times n$ matrices, respectively.
 For a given $p$, with $1\leq p\leq n-1,$ set $M_{n,p}:=\sum\limits_{l=1}^{p}E_{l,l+1}$, so that $M_{n,p}^{j-1}=\sum\limits_{i=1}^{p-j+2} E_{i,i+j-1}$,  for  $2\leq j\leq p+1$, and  $\; M_{n,p}^{p+1}=0$. 
Let $\mathfrak{B}_{n,p}$  be the Abelian subalgebra of $\mathfrak{gl}(n,\mathbb R)$ spanned by  $e_j,$ $j=1,\dots,n$, where
$e_j:=M_{n,p}^{j-1}$, $j=1,\dots,p+1$,  and for $j=p+2,\dots,n$, $e_{j}:=E_{1,j}$. In 
 the canonical basis $(\tilde e_1, \dots,\tilde e_n)$ of $\mathbb R^n$, with dual basis $(\tilde e_1^*,\dots,\tilde e_n^*)$, we have
\begin{eqnarray}M_{n,p}\tilde e_1&=&M_{n,p}\tilde e_{q}=0,\;\; q\ge p+2, \nonumber\\
M_{n,p}^{k}\tilde e_q&=&M_{n,p}^{k-1}\tilde e_{q-1}=M_{n,p}^{k-j}\tilde e_{q-j}, \;\;  1\leq j\leq\text{min}(k,q-1), \forall q, \; 2\leq q\leq p+1.\nonumber
\end{eqnarray}
So if $k\ge q$, then $M_{n,p}^{k}\tilde e_q=0$,  and  if $k\leq q-1\leq p$, then $M_{n,p}^{k}\tilde e_q=\tilde e_{q-k}$.
As one can see, we have 
$\tilde e_1^*\circ e_j= \tilde e_1^*\circ M_p^{j-1} = \tilde e_j^*$, $j=1,\dots,p+1$ and $\tilde e_1^*\circ e_j=\tilde e_1^*\circ E_{1,j}=\tilde e_j^*$, for $j=p+2,\dots,n$.
 This shows that $\tilde e_1^*$ has an open orbit for the contragredient action of $\mathfrak{B}_{n,p}$ on $(\mathbb R^n)^*$.
For any $M:=x_1e_1+\dots+x_ne_n\in \mathfrak{B}_{n,p},$  we have $\chi_M(\lambda) =(x_1-\lambda)^n$. But  the equality $(x_1-M)^{p+1}=0$ implies that $\chi_{\min,M}(\lambda)$ divides $(x_1-\lambda)^{p+1}$. So  $\mathfrak{B}_{n,p}$ is not the space of polynomials of a nonderogatory matrix, unless  $p=n-1$. In the basis 
$(e_j, e_{n+j}=\tilde e_j,\; j=1,\dots,n)$ the Lie bracket of the  2-solvable Frobenius Lie algebra  
 $\mathcal G_{n,p}:=\mathfrak{B}_{n,p}\ltimes \mathbb R^n$,  is given by the following table, $j,k,q\in\{1,2,\dots,n\},$  
\begin{eqnarray}
\;\; \;\;\;\;\;\; 
[e_1,e_{n+j}] &=& e_{n+j} \;, \; j=1,\dots, n,\nonumber\\ 
\;
 [e_p,e_{n+q}] &=&0, \text{ if } p\ge q+1, \nonumber\\  
\; [e_k,e_{n+q}]&=&e_{n+q-k+1} \text{ if } 1\leq k \leq q\leq p+1,
\nonumber\\ 
\; [e_q,e_{n+q}] &=& e_{n+1},  \text{ for }  q=p+2,\dots,n.\label{BracketGnp} 
\end{eqnarray}
The form  $e_{n+1}^*$ is a Frobenius functional on $\mathcal G_{n,p}$ and $\partial e_{n+1}^*=-\sum\limits_{i=1}^{n}e_j^*\wedge e_{n+j}^*$.  Note that $\mathcal G_{n,p}$ and $\mathcal G_{n,q}$ are isomorphic if and only if $p=q$. Indeed, although each $\mathcal{G}_{n,p}$ has a codimension $1$ nilradical $\mathcal N_{n,p}=\sspan(e_2,\dots,e_{2n})$, the derived ideal $[\mathcal N_{n,p},\mathcal N_{n,p}]=\sspan(e_{n+1},\dots, e_{n+p})$ is $p$-dimensional
and $\mathcal{N}_{n,p}$ is $(p+1)$-nilpotent. So $\mathcal N_{n,p}$ and $\mathcal N_{n,q}$ are not isomorphic whenever $p\neq q$ and hence, neither are the Lie algebras $\mathcal G_{n,p}$ and $\mathcal G_{n,q}$. 
Thus   $\mathfrak{B}_{n,p}$ and  $\mathfrak{B}_{n,q}$ are not conjugate.
The family ($\mathcal G_{n,p}$)$_{1\leq p\leq n-1}$, has two special cases. (1) The case $\mathfrak{B}_{n,n-1}=\mathbb R[M_0],$ where $M_0=\sum\limits_{j=1}^{n-1} E_{j,j+1}$ is nonderogatory.   See 
 Section \ref{chap:classification-nonderogatory} for a 
 full classification of  2-solvable Frobenius Lie algebras $\mathbb R[M]\ltimes V$,  for nonderogatory $M\in\mathfrak{gl}(V)$. To keep the same notations as in Section \ref{chap:classification-nonderogatory}, set $\mathfrak{D}_0^n=\mathcal G_{n,n-1}.$ 
(2) As regards the case $p=1$, the  
nilradical $\mathcal N_{n,1}$ of $\mathcal G_{n,1}$ is the $(2n-1)$-dimensional Heisenberg Lie algebra $\mathcal H_{2n-1}:=$span($e_{j}:=E_{1,j}, j=2,\dots,n, e_{n+k}, \; k=1,\dots,n$),  with Lie brackets $[e_j,e_{n+j}] =e_{n+1}$, $j=2,\dots,n.$ Note also that the $(n-1)$ spaces $\mathcal A_{n,p}=$span($e_2,e_3,\dots,e_n$)  are all $(n-1)$-dimensional Abelian subalgebras of $\mathfrak{sl}(n,\mathbb R)$, which according to Theorem \ref{thm:structure-2-solvable-Frobenius} and Proposition \ref{lemma:conjugaison-MASAs}, are pairwise non-conjugate MASAs of $\mathfrak{sl}(n,\mathbb R).$ 
In particular,  $\mathcal A_{n,1}$ is spanned by $E_{1,j},$ $j=2,\dots,n$ and so $\mathcal A_{n,1}^2=0.$
\end{example}

\begin{example}
\label{example2a} For $n\ge 3,$ let  $\mathcal P_{n,p}$ be the Abelian subalgebra of $\mathfrak{sl}(n,\mathbb R)$ defined as  $\mathcal P_{n,p}=\{\sum\limits_{j=2}^{p}m_j(E_{1,j}+E_{j,n}) + \sum\limits_{j=p+1}^{n}m_jE_{1,j}
 ,\; m_2,\dots,m_n\in\mathbb R\}$,  $p=2,3,\dots,n-1$. 
On  $\mathcal C_{n,p}:=\mathcal{P}_{n,p}\oplus \mathbb R \mathbb I_{\mathbb R^n}$,  
 set $e_1:=\mathbb I_{\mathbb R^n}$ and $e_j:=E_{1,j}+E_{j,n}$, $j=2,\dots,p$, 
$e_{p+i}=E_{1,p+i}$, $i=1,\dots,n-p$. Let $(\tilde e_j)$ be the canonical basis of $\mathbb R^n$. 
For any  $M=m_1e_1+\dots+m_ne_n$, we have $\chi_M(X)=(m_1-X)^n$, 
 $ (m_1-M)^{2}=(m_2^2+\dots+m_{p}^2)E_{1,n}$ and $ (m_1-M)^{3}=0$. So, up to a scaling,  $\chi_{\min,M}=(m_1-X)^3$.  Thus for any $n\ge 4,$ the algebra $\mathcal C_{n,p}$ contains no nonderogatory matrix. 
We have $\tilde e_1^*\circ e_1= \tilde e_1^*$, $\tilde e_1^*\circ e_j= \tilde e_1^*\circ (E_{1,j}+E_{j,n})= \tilde e_1^*\circ E_{1,j}=\tilde e_j^*$, 
$\tilde e_1^*\circ e_{p+i}= \tilde e_1^*\circ E_{1,p+i}= \tilde e_{p+i}^*$, $i=1,\dots,n-p,$ $j=2,\dots,p.$ So $\tilde e_1^*$ has an open orbit. 
Hence $\mathcal C_{n,p}$ and $\mathcal P_{n,p}$ are MASAs of $\mathfrak{gl}(n,\mathbb R)$ and $\mathfrak{sl}(n,\mathbb R)$, respectively (Theorem \ref{thm:structure-2-solvable-Frobenius}).
The Lie bracket of $\mathfrak h_{n,p}:=\mathcal{C}_{n,p}\ltimes \mathbb R^n$,  in  the basis $(e_j, e_{n+j}:=\tilde e_j, j=1,\dots,n),$  is 
\begin{eqnarray}
\; [e_1,e_{n+j}] &=&e_{n+j}\;, \; j=1,\dots,n,\nonumber\\
 \; [e_j,e_{n+j}] &=&e_{n+1}, \;j=2,\dots,n, \; \nonumber \\
\; [e_k,e_{2n}] &=&e_{n+k}\;, \;\; k=2,\dots,p,\nonumber
\end{eqnarray}
so  $\partial e_{n+1}^*=-\sum\limits_{j=1}^ne_j^*\wedge e_{n+j}^*$ is nondegenerate, $\mathfrak h_{n,p}$ is a Frobenius Lie algebra. The  nilradical $\mathfrak{n}_{n,p}:=$span$(e_2,\dots,e_{2n})$ of $ \mathfrak h_{n,p}$ is of
codimension $1.$ It is the semidirect sum $\mathfrak{n}_{n,p}=\mathbb R e_{2n}\ltimes\Big( \mathcal H_{2n-3}\oplus\mathbb R e_{n}\Big)$ of the line $\mathbb R e_{2n}$ and the so-called Abbena Lie algebra $ \mathcal H_{2n-3}\oplus\mathbb R e_{n}$ where the former acts on the latter by nilpotent derivations.
Here, $ \mathcal H_{2n-3}$ is the  ($2n-3$)-dimensional Heisenberg Lie algebra  $ \mathcal H_{2n-3}=\text{span}(e_{2},\dots,e_{n-1},e_{n+1},\dots,e_{2n-1})$. The derived ideal $[\mathfrak{n}_{n,p},\mathfrak{n}_{n,p}]$ is $p$-dimensional and spanned by $(e_{n+1},\dots,e_{n+p}).$
 So if $p\neq q$, then $\mathfrak{n}_{n,p}$ and $\mathfrak{n}_{n,q}$ are not isomorphic, hence $\mathcal{P}_{n,p}$ and $\mathcal{P}_{n,q}$ are not conjugate, $\mathfrak h_{n,p}$ and  $\mathfrak h_{n,q}$ are not isomorphic.
Note that $\mathfrak{n}_{n,p}$ is $3$-step nilpotent. 
So, for any $(n,p)$ with $n\ge 4$ and $3\le p\le n-1$,  $\mathfrak{n}_{n,p}$  is isomorphic to none of the nilradicals $\mathcal N_{n,q}$, $3\le q\le n-1$, of Example  \ref{example-general-nD}, 
as they are all $(q+1)$-step nilpotent. When $p=2,$ we note that the linear map $\psi:\mathcal G_{n,2}\to \mathfrak h_{n,2}$,
\beqn \label{isom:Gn2-hn2} \psi(e_m)&=&e_m'\; , \;\; 1\le m\le 2n\;,\;\; m\notin \{3,n, n+3,2n\}\;,\nonumber \\
\psi(e_3)&=&e_{n}'\; , \; \psi(e_n)=e_{3}'\; , \; \psi(e_{n+3})=e_{2n}'\; , \; \psi(e_{2n})=e_{n+3}'\;, 
\eeqn
is an isomorphism of Lie algebras, for any $n\ge 3,$ where $\mathcal G_{n,2}$ is as in Example \ref{example-general-nD} and the above basis of 
$ \mathfrak h_{n,2}$ has been renamed $(e_1',\dots,e_{2n}').$
 Hence, altogether, the Lie algebras  $\mathfrak{h}_{n,p}$ are not isomorphic to any of the Lie algebras $\mathcal G_{n,q}$ of Example  \ref{example-general-nD}, unless $p= q=2.$  
Thus none of the MASAs $\mathfrak{B}_{n,q}$  of Example  \ref{example-general-nD} is conjugate to $\mathcal{C}_{n,p},$ unless $p= q=2.$   
\end{example}

 \begin{example}\label{example23} Let $L_{n,n}'$ be the following  Abelian subspace of $\mathfrak{sl}(n,\mathbb R)$,  
\begin{eqnarray}L_{n,n}'=\{M:=\sum\limits_{j=2}^{n}m_j(E_{1,j}+E_{n-j+1,n}) \;,\; m_j\in\mathbb R, j=2,\dots,n\}.\end{eqnarray}
On $\mathfrak B_{n,n}':=\mathbb R \mathbb I_{\mathbb R^n}\oplus  L_{n,n}'$, 
 we set $e_1:=\mathbb I_{\mathbb R^n},$ $e_j:=E_{1,j}+E_{n-j+1,n},$ $j=2,\dots,n$. 
For any $M=m_1e_1+\dots+m_ne_n\in \mathfrak{B},$ we have $\chi(X)=(m_1-X)^n$. We also have $\chi_{\min,M}=(m_1-X)^3$, since $ (m_1-M)^{3}=0$ and $ (m_1-M)^{2}\neq 0$. Thus, for any $n\ge 4,$ 
 $\mathfrak B_{n,n}'$ contains 
no nonderogatory matrix.
We have $\tilde e_1^*\circ e_j= \tilde e_j^*$, $j=1,\dots,n-1,$ $\tilde e_1^*\circ e_n= 2\tilde e_n^*.$ So $\tilde e_1$ has an open orbit. According to Theorem \ref{thm:structure-2-solvable-Frobenius}, $\mathfrak B_{n,n}'$ and $L_{n,n}'$ are MASAs of $\mathfrak{gl}(n,\mathbb R)$ and $\mathfrak{sl}(n,\mathbb R)$, respectively.
 In the basis $(e_j, e_{n+j}:=\tilde e_j, j=1,\dots,n),$  the Lie bracket of $\mathcal G_{n,n}':=\mathfrak{B}_{n,n}'\ltimes \mathbb R^n$, is 
\begin{eqnarray}
\; [e_1,e_{n+j}] &=&e_{n+j}\;, \; j=1,\dots,n,\nonumber\\
\; [e_j,e_{n+j}] &=&e_{n+1}, \; [e_j,e_{2n}] =e_{2n-j+1}, \; j=2,\dots,n-1, \;\; [e_n,e_{2n}]=2e_{n+1}\;.\nonumber
\end{eqnarray}
As one  sees, $\partial e_{n+1}^*=-\sum\limits_{j=1}^{n-1}e_j^*\wedge e_{n+j}^*-2e_n^*\wedge e_{2n}^*$ is nondegenerate, so $\mathcal G_{n,n}'$ is a Frobenius Lie algebra. It also has a codimension $1$ nilradical $\mathcal N_{n,n}'=span(e_2,\dots,e_{2n})$, which is the semidirect sum  
of $\mathbb R e_{2n}$ and $ \mathcal H_{2n-3}\oplus\mathbb R e_{n},$ where the former acts on the latter by nilpotent derivations. However,
$[\mathcal N_{n,n}',\mathcal N_{n,n}']=\sspan(e_{n+1},\dots,e_{2n-1})$ is $(n-1)$-dimensional.
 So $\mathcal G_{n,n}'$ is not isomorphic to any of the Lie algebras $\mathcal G_{n,p}$ of Example  \ref{example-general-nD} and
 $\mathfrak h_{n,p}$ of Example \ref{example2a}. From Theorem \ref{thm:structure-2-solvable-Frobenius}, none of the MASAs $\mathfrak{B}_{n,p}$  in Example  \ref{example-general-nD} or $\mathcal C_{n,p}$ in Example \ref{example2a}, is conjugate to  $\mathfrak B_{n,n}'.$
\end{example}
\begin{remark}\label{rem:masa-no-frobenius}From  Theorem \ref{thm:structure-2-solvable-Frobenius}, the classification of $2$-solvable Frobenius Lie algebras is equivalent to that of $n$-dimensional MASAs of $\mathfrak{gl}(n,\mathbb K)$ acting on $(\mathbb K^n)^*$ with an open an orbit. However, not all $n$-dimensional MASAs have open orbit in $(\mathbb K^n)^*$. 
Indeed, for $n\ge 3$, the algebra   $\mathfrak{B}_{n}:= \mathbb R \mathbb I_{\mathbb R^n} \oplus L_{n}$, is a MASA of $\mathfrak{gl}(n,\mathbb R)$, where
 $L_{n}:=\Big\{\sum\limits_{i=1}^{n-1}k_{i,n}E_{i,n},\; k _{i,n}\in\mathbb R, i=1,\dots,n-1\Big\}$ is a MASA of $\mathfrak{sl}(n,\mathbb R).$  To see that,
consider $b=\sum\limits_{p,q=1}^nt_{p,q}E_{p,q} \in\mathfrak{gl}(n,\mathbb R)$, with $t_{p,q}\in\mathbb R$, $p,q=1\dots,n.$  We have $[E_{i,n},b]=\sum\limits_{l=1}^{n-1}t_{n,l}E_{i,l} +(t_{n,n}-t_{i,i})E_{i,n} 
-\sum\limits_{1\leq k\leq n, \; k\neq i} t_{k,i}E_{k,n} $, for any $i=1,\dots,n-1.$ 
So the relation $[b,a] =0$, for any $ a\in\mathfrak B_n$, is equivalent to the following, valid for any $i$ with $1\leq i\leq n-1$: $t_{i,i}=t_{n,n}$ and for any $k$ with $1\leq k\leq n$ and $k\neq i$, one has 
  $t_{k,i}=0$. 
In other words $b=t_{n,n}\mathbb I_{\mathbb R^n} + t_{1,n}E_{1,n}+\dots + t_{n-1,n}E_{n-1,n}$, or equivalently, $b$ is an element of $\mathfrak{B}_n$. This simply means that $\mathfrak{B}_n$ is a MASA of $\mathfrak{gl}(n,\mathbb R)$.
The orbit  $\{\alpha\circ a, a\in \mathfrak{B}_{n}\}$  of any $\alpha\in (\mathbb R^n)^*$ is  at most 2-dimensional and spanned by $\alpha$ and $\tilde e_n^*$. More precisely, let $(\tilde e_1,\dots,\tilde e_n)$ be the canonical basis of $\mathbb R^n$ and let $\alpha = s_1\tilde e_1^*+\dots+s_n\tilde e_n^*\in (\mathbb R^n)^*$ , 
where $s_1,\dots,s_n\in\mathbb R,$ then for any $k_1,k_{1,n},\dots,k_{n-1,n}\in \mathbb R$ and $a=k_1\mathbb I_{\mathbb R^3}+k_{1,n}E_{1,n}+\dots+k_{n-1,n}E_{n-1,n}\in \mathfrak{B}_{n}$, one has $\tilde e_i^*\circ a =k_1\tilde e_i^*+k_{i,n}\tilde e_n^*$ ,\;  $i=1,\dots,n-1$ and  $\tilde e_n^*\circ a =k_1\tilde e_n^*$, 
so that
 \begin{eqnarray}\alpha\circ a
&=&
k_1\alpha+(k_{1,n}s_1
+k_{2,n}s_2
+\dots+
 k_{n-1,n} s_{n-1})e_n^*.\nonumber
\end{eqnarray}
For $n=3$, $L_3$ coincides with $L_{2,4}$ in the list of MASA of $\mathfrak{sl}(n,\mathbb R)$ supplied in \cite{winternitz}. 
\end{remark}

\subsection{Cartan subalgebras of $\mathfrak{sl}(n,\mathbb R)$}\label{sect:Cartan-subalgebras-sl(n,R)}

Recall that a Cartan subalgebra $\mathfrak{h}$ of a Lie algebra $\mathcal G$ is a nilpotent subalgebra which is equal to its own normalizer $\mathcal N_{\mathcal G} (\mathfrak{h})$ 
 in $\mathcal G.$
Cartan subalgebras of semisimple Lie algebras must necessarily be Abelian, more precisely they are MASAs  which contain only semisimple elements.
They have been extensively studied by several authors amongst which E. Cartan, Harish-Chandra (\cite{harish-chandra}), B. Kostant (\cite{kostant}), M. Sugiura (\cite{sugiura}), etc. 
 Further recall that a Cartan subalgebra $\mathfrak{h} $ of a semisimple Lie algebra $\mathcal G$, splits into a direct sum $\mathfrak{h} =\mathfrak{h^+} \oplus \mathfrak{h} ^-$ of two subalgebras $\mathfrak{h}^+ $ and $\mathfrak{h}^- $, respectively 
called its toroidal and its vector parts, such that $\mathfrak{h}^+ $ is only made of elements of $\mathfrak{h} $ whose adjoint operator (as a linear operator of $\mathcal G$) only has purely imaginary eigenvalues and adjoint operators of elements of $\mathfrak{h}^- $ have only real eigenvalues. See e.g. \cite{sugiura}. We propose the following characterization of Cartan subalgebras of $\mathfrak{sl}(n,\mathbb R).$
\begin{theorem}\label{thm:cartansubalgebras} Let  $\mathfrak{h}\subset \mathfrak{sl}(n,\mathbb R)$ be a subalgebra. Set $\mathfrak{B}_{\mathfrak{h}}:=\mathbb R\mathbb I_{\mathbb R^n}\oplus \mathfrak{h}$. The following are equivalent.
(a) $\mathfrak{h}$ is a Cartan subalgebra   of $\mathfrak{sl}(n,\mathbb R)$, with toroidal and vector parts of  respective  dimensions $p$ and $q.$
(b) $\mathfrak{B}_\mathfrak{h}=\mathbb R[M]$, for   
some nonderogatory $M\in\mathfrak{gl}(n,\mathbb R)$ with distinct $2p$ complex and q real eigenvalues, where $2p+q=n.$
(c) $\mathfrak{B}_\mathfrak{h}\ltimes  \mathbb R^n$ is the direct sum 
 of $p$ copies of
$\mathfrak{aff}(\mathbb C)$  and $q$ copies of $\mathfrak{aff}(\mathbb R)$, where $\mathfrak{aff}(\mathbb R)$ and $\mathfrak{aff}(\mathbb C)$  are as in Examples  \ref{ex:D0n}, \ref{ex:D01n}.  
\end{theorem}
Theorem \ref{thm:cartansubalgebras} is proved in Section \ref{proofofCartan-subalgebras-sl(n,R)}. Theorem \ref{prop:abeliannilrad} gives a complementary characterization.

\section{Classification of 2-solvable Frobenius Lie algebras given by nonderogatory linear maps}\label{chap:classification-nonderogatory}

\subsection{Some key examples}\label{sect:examples-nonderogatory}

\subsubsection{The Lie algebras $\mathfrak D_0^n$}\label{ex:D0n}
(1) Consider the simplest nonderogatory linear map $\psi := \mathbb I_{\mathbb R}.$ One gets the Lie algebra  $\mathcal G_\psi=\mathfrak{aff}(\mathbb R)$  of the group of affine motions of $\mathbb R.$
In the basis  $(e_1,e_2)$, with  $e_1=\psi^0$ and $e_2\in\mathbb R $,  the Lie bracket is $[e_1,e_2]=e_2.$ The 2-form $\partial e_2^*=-e_1^*\wedge e_2^*$ is nondegenerate.
(2) In the canonical basis $(\tilde e_{j})$   of $\mathbb R^n$, let $M_0\in\mathfrak{gl}(n,\mathbb R)$ be the  principal nilpotent matrix
\begin{eqnarray}\label{principalnilpotent} 
M_0= \sum\limits_{i=1}^{n-1}E_{i,i+1}.\end{eqnarray}
 It is nonderogatory, as $\chi_{\min,M_0}(X)=\chi_{M_0}(X)=X^n.$ 
We use the notation  $ e_{j}:=M_0^{j-1},$ $j=1,\dots,n.$
Set $e_{n+j}:=M_0^{j-1}\bar x=\tilde e_{n-j+1}$, where 
 $\bar x:=\tilde e_n$. Then ($e_{n+j}$) 
is a basis of $\mathbb R^n$. In the basis   $(e_{1}, e_{2}, \dots, e_{2n})$,  the Lie bracket of $\mathcal G_{M_{0}}=\mathbb R[M_0]\ltimes \mathbb R^n,$ is
\begin{eqnarray}\label{LieBracketsD0n}
 [e_{i},e_{n+j}]=e_{n+j+i-1}, ~i,j=1,\dots,n, \end{eqnarray}
where we use the convention $e_{2n+k}=0,$ whenever $k\ge 1.$
 In particular, 
$[e_{1},e_{n+j}]=e_{n+j},$ $j=1,\dots,n,$  and $[e_{j},e_{2n}]=0,$ $j=2,\dots,n.$
We write  ${\mathfrak D}_{0}^{n}$ instead of   $\mathcal G_{M_{0}}.$ The following  2-form  on  ${\mathfrak D}_{0}^{n}$, is non-degenerate,  
\begin{eqnarray}\partial e_{2n}^*=-\sum\limits_{j=1}^ne_j^*\wedge e_{2n-j+1}^*=- \sum\limits_{j=1}^n e_{n-j+1}^*\wedge e_{n+j}^*\;.\end{eqnarray} 
 Note that ${\mathfrak D}_{0}^{n}$ has an $n$-nilpotent codimension $1$ nilradical  $\mathcal N=$span$(e_2, \dots,  e_{2n}).$

\subsubsection{The Lie algebra $\mathfrak{D}_{0,1}^n$} \label{ex:D01n}
(A)  Let  $\psi\in\mathfrak{gl}(2,\mathbb R)$ with $\psi(\tilde e_1)=\tilde e_2$, $\psi(\tilde e_2)=-\tilde e_1$. One has  $\chi_{\min,\psi}(X)=\chi_\psi(X)
=1+X^2$. So $\psi$ has the  $2$ complex eigenvalues $ i$, $-i$.  Of course, $\mathbb R[\psi] =\mathbb R\mathbb I_{\mathbb R^2}\oplus \mathbb R \psi$ is a MASA of $\mathfrak{gl}(2,\mathbb R)$ and  $\tilde e_1^*\circ \psi^0 =\tilde e_1^*$,   $\tilde e_1^*\circ  \psi = -\tilde e_2^*$. 
 In the basis $(e_1,\dots,e_4)$ of $\mathcal G_{\psi}:=\mathbb R[\psi]\ltimes \mathbb R^2$, with $ e_1 =  \psi^0$, $e_2=  \psi,$ $e_3:= \tilde e_1$, $e_4:=\tilde e_2$, the Lie bracket  reads
 $[e_1, e_3] = e_3, [e_1, e_4] = e_4, [e_2, e_3] =e_4, [e_2, e_4] = - e_3.$
The 2-form  $\partial e_3^*= - e_1^* \wedge e_3^*+ e_2^* \wedge e_4^*$ 
 is  nondegenerate on $\mathcal G_{\psi}.$ 
 Note that $\mathcal G_{\psi}$ is the Lie algebra 
$\mathfrak{aff}(\mathbb C)$=$\Big\{
N(z_1,z_2):=\begin{pmatrix} z_1& z_2\\ 0 & 0\end{pmatrix}, z_j\in\mathbb C, \;j=1,2.\Big\}$ of the group of affine motions of the complex line $\mathbb C,$ looked at as a real Lie algebra, e.g. with the identifications 
 $e_1=N(1,0)$,  $e_2=N(i,0),$ 
$e_3=N(0,1)$, $e_4=N(0,i).$ We set $\mathfrak{D}_{0,1}^2:=\mathfrak{aff}(\mathbb C).$
(B) Let  $M_{0,1}=E_{2,1}-E_{1,2}+E_{4,3}-E_{3,4}+E_{1,3}+E_{2,4}\in\mathfrak{gl}(4,\mathbb R)$. 
One has  $\chi_{\min,M_{0,1}}(X)=\chi_{M_{0,1}}(X)
=(X^2+1)^2$. So $i$ and  $-i$  are the only eigenvalues of $M_{0,1}$ and $\mathbb R[M_{0,1}]$
 is a MASA of $\mathfrak{gl}(4,\mathbb R)$.
In the basis  $(e_1,e_2,\dots,e_8)$ of $\mathcal{G}_{M_{0,1}}$, with $e_1:=\mathbb I_{\mathbb R^4}$,  $e_2:=M_{0,1}$, 
$e_3:=M_{0,1}^2$, $e_4:=M_{0,1}^3$, 
 $e_{4+j}=\tilde e_j$ , $j=1,\dots,4$, 
the Lie bracket reads
\begin{eqnarray}\label{eq:forgottenLiealgebra} && [e_1,e_l]=e_l\;, \;\; l=5,6,7,8, \nonumber\\
&&[e_2,e_5]=e_6\;,\;\;\;\;[e_2,e_6]=-e_5\;,\; [e_2,e_7]=e_5+e_8\; ,\;\;\;\;\; [e_2,e_8]=e_6-e_7\;,
\nonumber\\
&&[e_3,e_5]=-e_5\;, \; [e_3,e_6]=-e_6\;,\; [e_3,e_7]=2e_6-e_7\;,\;\; \;\; [e_3,e_8]=-2e_5 -e_8\;,
\nonumber\\
&&[e_4,e_5]=-e_6\;,\; [e_4,e_6]=e_5\;,\;\; \;\;[e_4,e_7]=-3e_5-e_8\;,\; [e_4,e_8]=-3e_6+e_7\;.
\nonumber
\end{eqnarray}
Note that both $(e_7,M_{0,1}e_7 = e_5+e_8,  
M_{0,1}^2e_7 = 2e_6-e_7, M_{0,1}^3e_7 = -3e_5-e_8)$ and 

\noindent
$(e_8,M_{0,1}e_8 = e_6-e_7, \; 
 M_{0,1}^2e_8 =- 2e_5-e_8, M_{0,1}^3e_8 = -3e_5+e_7)$ are bases of $\mathbb R^4.$ 
The 2-form $\partial e_5^* =- e_1^*\wedge e_5^* + e_2^*\wedge( e_6^* - e_7^*) +  e_3^*\wedge (e_5^* + 2  e_8^*)   -  e_4^*\wedge (e_6^* -3 e_7^*) $ is nondegenerate. 
 (C) This generalizes (see Section \ref{sect:nondiagonalizableinC}) to a $2n$-dimensional 2-solvable Frobenius Lie algebra $\mathfrak{D}_{0,1}^n:=\mathcal G_{M_{0,1}}=\mathbb R[M_{0,1}]\ltimes \mathbb R^n,$  for any even $n\ge 4$,  where the nonderogatory   $M_{0,1}\in\mathfrak{gl}(n,\mathbb R)$ is defined, in the canonical basis $(\tilde e_1,\dots,\tilde e_n)$ of $\mathbb R^n$, as
 \begin{eqnarray} \label{Ms-Mn-M01} M_{0,1}= M_s + M_n,  \text{ where } M_s=-\sum\limits_{j=0}^{\frac{n}{2}-1}(E_{2j+1,2j+2}-E_{2j+2,2j+1}),\; M_n=\sum\limits_{j=1}^{n-2}E_{j,j+2}.
\end{eqnarray} 
Hence, $\chi_{\min,M_{0,1}}(X)=\chi_{M_{0,1}}(X)=(X^2+1)^{\frac{n}{2}}$, so  $i$ and  $-i$  are the only (${\frac{n}{2}}$ times repeated complex conjugate) eigenvalues of $M_{0,1}$ and $\mathbb R[M_{0,1}]$ is a MASA of $\mathfrak{gl}(n,\mathbb R)$.
In the basis  $(e_1,e_2,\dots,e_{2n})$ of $\mathfrak{D}_{0,1}^n$, with  $e_j:=(M_{0,1})^{j-1}$, $e_{n+j}=\tilde e_{j}$, $j=1,\dots, n$, 
the Lie bracket reads, for any $ j, l=1,\dots,n,$
\begin{eqnarray}\label{eq:forgottenLiealgebra} 
&&[e_j,e_{n+l}]=(M_{0,1})^{j-1}\tilde e_l  \;.
\end{eqnarray}

\subsection{The classification Theorem}

In Theorem \ref{thm:classification}, we completely classify all $2$-solvable Frobenius Lie algebras  of the form 
$\mathcal G_\phi:=\mathbb R[\phi]\ltimes V$, where $\phi\in\mathfrak{gl}(V)$ is  nonderogatory and $V$ a real vector space of dimension $n$. In particular, we show that  the Lie algebras $\mathfrak{D}_0^p$,  $\mathfrak{D}_{0,1}^{2p}$, where $p\ge 1$ is an integer, discussed in Section \ref{sect:examples-nonderogatory}, are the  building blocks that make up, in a trivial way (direct sums), the Lie algebras  $\mathcal G_\phi$. 
  As in Examples \ref{ex:D0n},  \ref{ex:D01n}, the notations $\mathfrak{D}_0^1:=\mathfrak{aff}(\mathbb R)$ and  $\mathfrak{D}_{0,1}^2:=\mathfrak{aff}(\mathbb C)$, are implicitly adopted.
It is obvious that, if $z,\bar z$ are two complex conjugate eigenvalues of $\phi$, then $(X-Re(z))^2+Im(z)^2$ divides $\chi_\phi(X)$ (Lemma \ref{lemma:factorization}). 
\begin{definition}We say that  $2$ complex conjugate eigenvalues $z,\bar z$ of $\phi\in\mathfrak{gl}(V)$ are of multiplicity  $m$, if $m$ is the largest integer such that  $\Big( (X-Re(z))^2+Im(z)^2\Big)^m$ is a factor of  $\chi_\phi(X)$.
\end{definition}
We call the following, the Factorization Lemma.
\begin{lemma}[Factorization Lemma]\label{lemma:factorization} Let $V$ be a real  vector space,   $\dim V=n$. Suppose a nonderogatory  $\phi\in\mathfrak{gl}(V)$ has $p$ distinct real eigenvalues $\lambda_j,$  with multiplicity $k_j$, $j=1,\dots,p$ and
 $2q$ complex eigenvalues $\lambda_{p+l},\bar\lambda_{p+l}$, with multiplicity $k_{p+{l}}$, $l=1,\dots,q,$ where the equality $n=k_1+\dots+ k_p+2(k_{p+1}+\dots+k_{p+q})$ holds. Then the characteristic polynomial  of $\phi$ factorizes as 
 \beqn\label{eq:factorization}
  \chi_\phi(X)=\displaystyle\prod_{j=1}^{p}(X-\lambda_j)^{k_j}  \prod_{l=1}^{q} \Big((X-Re(\lambda_{p+l}))^2+Im(\lambda_{p+l})^2\Big)^{k_{p+l}}.
\eeqn
 In particular, if $\phi$  has only $2$ eigenvalues which are both complex 
 $\lambda,$ $\bar\lambda$, then 
\beqn \chi_\phi(X)=\Big((Re(\lambda)-X)^2+Im(\lambda)^2\Big)^{\frac{n}{2}}.\eeqn
 \end{lemma}
\begin{proof} 
Suppose a nonderogatory $\phi\in\mathfrak{gl}(V)$ has only $2$ eigenvalues which are both complex $\lambda,$ $\bar\lambda$.  Since $(X-\lambda)(X-\bar\lambda)=(X-Re(\lambda))^2+Im(\lambda)^2$ divides $\chi_\phi(X),$ we therefore write $\chi_\phi(X)=\Big((X-Re(\lambda))^2+Im(\lambda)^2\Big)P_1(X)$ 
where $P_1(X)$ is  a polynomial of degree $n-2,$ with real coefficients. As $\mathbb C$ is a closed field, $P_1(X)$ admits some complex zeros, bound to be 
 $\lambda,$ $\bar\lambda$ since they are the only zeros of $\chi_\phi(X)$, by hypothesis. We now re-write  $\chi_\phi(X) $ as $\chi_\phi(X) =\Big((X-Re(\lambda))^2+Im(\lambda)^2\Big)^2P_2(X)$ 
where $P_2(X)$ is  a polynomial of degree $n-4$, with real coefficients. The result follows by inductively reapplying the same process to $P_2$. The proof of the general case immediately follows by applying the Primary Decomposition Theorem to $\chi_\phi(X)$ to reduce the problem to the cases where $\phi$ admits a unique eigenvalue or only two eigenvalues which are both complex and conjugate, as above.
\end{proof}
Note that  Lemma \ref{lemma:factorization} is still valid even if $\phi$ is not nonderogatory.
In Theorem \ref{thm:classification} below, $V$ is a real vector space of dimension $n$.
 \begin{theorem}\label{thm:classification}Let   $\phi\in\mathfrak{gl}(V)$ be nonderogatory with $p$ real distinct eigenvalues $\lambda_j$ with multiplicities $k_j,$ $j=1,\dots,p$ and $2q$ distinct complex  eigenvalues $z_l,\bar z_l,$ with multiplicities $m_l,$ $l=1,\dots,q$ where $n=k_1+\dots+k_p+2(m_1+\dots+m_q)$.

Then the Lie algebra 
 $\mathcal G_\phi:=\mathbb R[\phi]\ltimes V$  is isomorphic to the direct sum 
 of the Lie algebras  
${\mathfrak D}_{0}^{k_1}, $ $\dots,$ ${\mathfrak D}_{0}^{k_p}$, $ {\mathfrak D}_{0,1}^{2m_1},$ $ \dots,$ ${\mathfrak D}_{0,1}^{2m_q}.$
In particular, 
\begin{itemize}
\item (a) if $p=n,$ then $\mathcal G_\phi$  is isomorphic to the direct sum 
 of $n$ copies of $\mathfrak{aff}(\mathbb R) $,

\item (b) if $p=1$ and $q=0,$ then  $\mathcal G_\phi$ is isomorphic to ${\mathfrak D}_{0}^{n}$,

\item (c)  if $2q=n,$ then 
 $\mathcal G_\phi$  is  isomorphic to 
the direct sum 
of   $\frac{n}{2}$
 copies of  $\mathfrak{aff}(\mathbb C) $,

\item (d)  if $p=0$ and $q=1,$ then  $\mathcal G_\phi$ is isomorphic to ${\mathfrak D}_{0,1}^{n}$.
\end{itemize}
\end{theorem}
The following is a direct corollary of Theorem \ref{thm:classification}.
\begin{corollary}\label{corollary1} Let $\phi\in\mathfrak{gl}(V)$ be nonderogatory. 
The  Lie algebra $\mathcal G_\phi$ is indecomposable if and only if one of the following holds true: 

(a) $\phi$ has a unique eigenvalue, in which case $\mathcal G_\phi$ is completely solvable, or

 (b)  $\phi$ has only 2 eigenvalues which are both complex.
\end{corollary}

The rest of this section and Sections \ref{Sect:uniqueeigenvalue}, \ref{Sect:complexeigenvalues} are mainly concerned with discussions and the proof of Theorem \ref{thm:classification}.
Lemma \ref{lem:splitting} allows us to  split  the proof into two main cases discussed in Propositions \ref{prop:classification-multiple-eigenvalues}, \ref{prop:classification-multiple-complexeigenvalues}, \ref{prop:isomorphismcomplexeigenvaluesnonsemiimple}, \ref{prop:all-complex-eigenvalues-nondiagonalizable}.
Part (a) of  Theorem \ref{thm:classification} is obtained by taking, in the general case, $p=n$, $q=0$, with the identification $\mathfrak{D}_0^1=\mathfrak{aff}(\mathbb R)$. A direct proof (for nonzero eigenvalues) is also presented in  Proposition \ref{ex:onlyrealeigenvalues}.
Part  (b) can be directly found in Lemma \ref{lem:uniquerealeigenvalue}, whereas  parts  (c)  and (d) are directly proved in
Propositions \ref{prop:classification-multiple-complexeigenvalues} and \ref{prop:isomorphismcomplexeigenvaluesnonsemiimple}, respectively.

\begin{lemma}\label{lem:splitting}
Suppose $V= \mathcal{E}_1\oplus\mathcal{E}_2$ and $\phi\mathcal{E}_j\subset \mathcal E_j,$ $j=1,2,$ where $\phi\in\mathfrak{gl}(V)$ is nonderogatory. 
Let  $\mathcal G_{\phi_j}:=\mathbb R[\phi_j]\ltimes \mathcal{E}_j$, where $\phi_j$ is the restriction of $\phi$ to $\mathcal{E}_j$.
Then both $\mathcal G_{\phi_1}$, $\mathcal G_{\phi_2}$ are ideals of $\mathcal G_{\phi}$ and, in fact, $\mathcal G_{\phi}=\mathcal G_{\phi_1}\oplus \mathcal G_{\phi_2}$.
\end{lemma}

\begin{proof} Let us extend $\phi_{j}$ to the linear map $\tilde {\phi}_{j}$ on 
$V$ such that $\tilde{\phi}_{j}(\mathcal{E}_{p})=0$ if $p\neq j.$ So  $[\tilde {\phi}_{1} , \phi] =[\tilde {\phi}_{2} , \phi] =0$ and hence
 $\tilde{\phi}_{1},$ $\tilde{\phi}_{2}\in \mathbb R[\phi]$. More precisely $\mathbb R[\tilde \phi_j]$ is a subalgebra of $\mathbb R[\phi]$. We thus see that $\mathcal  G_{\phi_{j}}$, identified with   
$\mathcal G_{{\tilde \phi}_{j}} =\mathbb R[\tilde{\phi}_{j}]\ltimes \mathcal{E}_{j}$, is a Lie subalgebra of $\mathcal G_\phi:=\mathbb R[\phi]\ltimes V.$ As a matter of fact,  each $\mathcal  G_{\phi_j}$ is an ideal of $\mathcal G_\phi$.  This directly follows from the combination of the following properties $\phi= \tilde \phi_1+\tilde \phi_2,$ $\mathbb R^n=\mathcal E_1\oplus\mathcal E_2,$ 
$\phi\mathcal{E}_j\subset \mathcal E_j,$ $j=1,2$  and 
$\tilde{\phi}_{j}(\mathcal{E}_{p})=0$ if $p\neq j.$
Thus, as the  ideals $\mathcal{G}_{\phi_j}$ form a direct sum (they only meet at $\{0\}$, unless they are identical), we get $\mathcal{G}:=\mathcal G_{{\tilde \phi}_{1}}\oplus \mathcal G_{{\tilde \phi}_{2}}$.
\end{proof}

Consider the general case where  $\phi\in\mathfrak{gl}(V)$ admits $p$ real and $2q$ complex eigenvalues.  We write $\chi_\phi(X)$ as the product 
 $\chi_\phi(X) = Q_1(X)Q_2(X)$ where $Q_1(X)$ has only complex (nonreal) zeros, whereas  the zeros of  $Q_2(X)$ are all real.
Of course as $Q_1(X)$ and $Q_2(X)$ must be relatively prime, the 
Primary Decomposition Theorem combined with Cayley-Hamilton theorem imply the following
\begin{eqnarray}V=\ker(\chi_\phi(\phi)) = \ker(Q_1(\phi))\oplus \ker(Q_2(\phi)). \end{eqnarray} 
As $\ker Q_1(\phi)$ and $\ker Q_2(\phi)$  are both stable by $\phi$, Lemma \ref{lem:splitting} reduces the proof of Theorem \ref{thm:classification} to two cases: the case where all the eigenvalues of $\phi$ are real and the case where all the eigenvalues of $\phi$ are complex. 

\subsection{Nonderogatory $\phi\in\mathfrak{gl}(V)$ with only real eigenvalues}\label{Sect:uniqueeigenvalue}

If a nonderogatory $\phi\in\mathfrak{gl}(V)$ has only real eigenvalues $\lambda_j$
of respective multiplicity $k_j$, $j=1,\dots,p$, then  its restriction $\phi_j$ to each subspace $\mathcal E_j:=\ker(\phi-\lambda_j)^{k_j}$ 
is again nonderogatory with 
a unique eigenvalue $\lambda_j$ of multiplicity $k_j$ and $V $ splits as $V = \mathcal E_1\oplus\cdots\oplus \mathcal E_p.$

\begin{definition}[Notation]\label{def:D0n}
When a nonderogatory $\phi\in\mathfrak{gl}(V)$ has a unique real eigenvalue $\lambda$ of multiplicity $n$, we set $\mathfrak{D}_\lambda^n:=\mathcal G_\phi=\mathbb R[\phi]\ltimes V.$
\end{definition}
\begin{lemma}\label{lem:multiplerealeigenvalues}Suppose a nonderogatory $\phi\in\mathfrak{gl}(V)$ has $p$ distinct eigenvalues $\lambda_1, \dots,\lambda_p,$  all of which are real and of respective multiplicity $k_1,\dots, k_p$, where
$k_1+\cdots+ k_p=n.$ Then $\mathcal G_\phi$ is isomorphic to the direct sum 
$ \mathfrak{D}_{\lambda_1}^{k_1}\oplus \cdots \oplus \mathfrak{D}_{\lambda_p}^{k_p}$.
\end{lemma}
\begin{proof} Applying Lemma \ref{lem:splitting} to 
Equality (\ref{eq:split}), where $\ker\chi_j(\phi)=\ker(\phi-\lambda_j \mathbb I_{V})^{k_j}$ $ =:\mathcal E_j,$ $j=1,\dots,p,$  leads to $\mathcal G_{\phi} = \mathcal G_{\phi_1}\oplus\dots\oplus\mathcal G_{\phi_p}$, where $\phi_j$ is the restriction of $\phi$ to $\mathcal E_j.$
 Each $\phi_j$   is a nonderogatory element of $\mathfrak{gl}(\mathcal E_j)$, with 
 a unique eigenvalue $\lambda_j$ of multiplicity $k_j$, so $\mathcal G_{\phi_j}=\mathfrak{D}_{\lambda_j}^{k_j},$ by Definition \ref{def:D0n}. Thus
$\mathcal G_\phi= \mathfrak{D}_{\lambda_1}^{k_1}\oplus \dots \oplus \mathfrak{D}_{\lambda_p}^{k_p}$.
\end{proof}
\begin{lemma}\label{lem:uniquerealeigenvalue}
If a nonderogatory $\phi\in\mathfrak{gl}(V)$ has a unique real eigenvalue $\lambda$, then the Lie algebra $\mathcal G_\phi =:\mathfrak{D}_{\lambda}^{n} $ is isomorphic to ${\mathfrak D}_{0}^{n}$ as in Example \ref{ex:D0n}.
\end{lemma}
\begin{proof}
Following Theorem \ref{thm:Jordanization}, consider a basis in which the matrix of $\phi$ has the form
$M_\lambda=\lambda\mathbb I_{V}+M_0,$ with $M_0$ as in (\ref{principalnilpotent}). 
 Both $M_0$ and $M_\lambda$  being nonderogatory and  $M_\lambda$ being a polynomial in  $M_0,$ imply the following:  
cent$(M_0,\mathfrak{gl}(V))=$ cent$(M_\lambda,\mathfrak{gl}(V))=$ $\mathbb K[M_0]$.  Thus we have  ${\mathfrak D}_{0}^{n} = \mathcal G_{M_0} = \mathcal G_{M_\lambda} =\mathfrak{D}_{\lambda}^{n}  .$
\end{proof}
Lemmas \ref{lem:multiplerealeigenvalues} and \ref{lem:uniquerealeigenvalue} prove the following.
\begin{proposition}\label{prop:classification-multiple-eigenvalues}
Suppose a nonderogatory $\phi\in\mathfrak{gl}(V)$ has $p$ distinct eigenvalues $\lambda_1, \dots,\lambda_p,$  all of which are real and of respective multiplicity $k_1,\dots, k_p$, where
$k_1+\cdots+ k_p=n.$ Then $\mathcal G_\phi$ is isomorphic to the direct sum $ \mathfrak{D}_{0}^{k_1}\oplus \cdots \oplus \mathfrak{D}_{0}^{k_p}$.
\end{proposition}
\begin{proof}Indeed, if a  nonderogatory $\phi\in\mathfrak{gl}(V)$ has $p$ real eigenvalues $\lambda_1, \dots,\lambda_p,$ 
 of respective multiplicity $k_1,\dots, k_p$, where
$k_1+\cdots+ k_p=n,$ then Lemma \ref{lem:multiplerealeigenvalues} ensures that $\mathcal G_\phi$ is isomorphic to the direct sum $ \mathfrak{D}_{\lambda_1}^{k_1}\oplus \cdots \oplus \mathfrak{D}_{\lambda_p}^{k_p}$ and Lemma  \ref{lem:uniquerealeigenvalue} further proves that each $\mathfrak{D}_{\lambda_j}^{k_j}$ is isomorphic to $\mathfrak{D}_{0}^{k_j}$, for  $j=1,\dots,p.$ 
\end{proof}
This concludes the proof of Theorem \ref{thm:classification} in the case  of nonderogatory $\phi\in\mathfrak{gl}(V)$  with only real eigenvalues. Proposition \ref{ex:onlyrealeigenvalues} supplies another rather direct proof (valid only when all the eigenvalues are nonzero) of Theorem \ref{thm:classification} (a).

\begin{proposition}\label{ex:onlyrealeigenvalues}  
Let $\mathcal G $ be  the direct sum  of $n$ copies of the Lie algebra $
\mathfrak{aff}(\mathbb R) $ and $\lambda_1,\dots,\lambda_n,$  $n$ distinct nonzero real numbers. Suppose $\phi\in\mathfrak{gl}(\mathbb R^n)$ is nonderogatory with eigenvalues $\lambda_1,\dots,\lambda_n.$ Then $\mathcal G$ and $\mathcal G_\phi$ are isomorphic.
\end{proposition}
\begin{proof}
 Choose a basis
$(a_i,b_i)$ of the ith copy of $
\mathfrak{aff}(\mathbb R) $, so that $[a_i,b_i]=b_i$  and set
$\mathfrak{B}=\hbox{\rm span}(a_1,\dots,a_n)$. Let $V:=[\mathcal G,\mathcal G]=\hbox{\rm span}(b_1, \dots,b_n) $.
 For $a\in\mathfrak{B}$, denote by $\rho(a)$ the restriction to $V$ of the adjoint  operator $[a,\cdot ]$ of $a.$ 
So one has  
 $\mathcal G=\mathfrak{B}\ltimes_\rho V.$
Further choose $a_0 =
\sum\limits_{i=1}^n\lambda_ia_i$, 
with $\lambda_i$ as in the hypotheses, so that the $\lambda_i$'s are the eigenvalues 
of   $\rho(a_0).$ By hypothesis, the characteristic
and minimal polynomials of $\rho(a_0)$ coincide. Hence $\mathbb R[\rho(a_0)]$ and $\mathbb R[\phi]$ are conjugate (with the identification $V=\mathbb R^n$).
Consider the linear map
 $\psi : \mathcal G \to\mathbb R[\rho(a_0)]\ltimes V
=\mathcal G_{\phi}
$
defined by $\psi (b)=b$, for every $b\in V$ and
 $\psi(a)=\sum\limits_{s=1}^nv_s
(\rho(a_0))^{s}$ for any $a=\sum\limits_{i=1}^nk_ia_i$, where $v=(v_1,\dots,v_n)$ is the solution of the equation
 $ Nv^T =
K^T
$, with $K=(k_1,\dots,k_n)$  
and $N$ is the $n\times n$ matrix with coefficients $N_{ij}=\lambda_i^j$.
One sees that, for any $b=\sum\limits_{j}^nt_jb_j,$ 
 $t_j\in\mathbb R,$ $\psi[a,b]=[a,b]=\sum\limits_{j=1}^nt_jk_jb_j = 
\sum\limits_{j=1}^nt_j  \sum\limits_{s=1}^nN_{js}v_sb_j = \sum\limits_{j=1}^nt_j  \sum\limits_{s=1}^n\lambda_j^sv_sb_j 
$ and $[\psi(a),\psi(b)] = \sum\limits_{j=1}^n t_j \sum\limits_{s=1}^nv_s
(\rho(a_0))^{s}b_j =   \sum\limits_{j=1}^n t_j \sum\limits_{s=1}^nv_s
\lambda_j^{s}b_j  =\psi[a,b]$
 and 
$[\psi(a),\psi(a')]=\psi[a,a']= [\psi(b),\psi(b')]=\psi[b,b']=0,$ for any $a,a'\in\mathfrak{B}$, $b,b'\in V$. We further have
$$\det(N)=\left(\prod\limits_{i=1}^n\lambda_i\right)\left(\prod\limits_{1\le i<j\le
n}(\lambda_j-\lambda_i)\right)\neq 0.$$
Thus $\psi$ is a Lie algebra isomorphism.
\end{proof}

\begin{example}
Let $\mathcal G$ be the  $6$-dimensional rank-three K\"ahler-Einstein solvable Lie algebra  in \cite{g2-structures-fino}, with  Lie bracket  
$[e_4,e_1]=se_1$, $[e_4,e_2]=se_2$, $[e_4,e_3]=se_3$, 
$[e_5,e_1]=-s\frac{\sqrt{6}}{2}  e_1$, $[e_5,e_2]=s\frac{\sqrt{6}}{2}  e_2$,
 $[e_6,e_1]=\frac{\sqrt{2}}{2}  e_1$, $[e_6,e_2]=s\frac{\sqrt{2}}{2}  e_2$, 
$[e_6,e_3]=-s\sqrt{2}  e_3,$  in a basis $(e_1,\dots,e_6),$ with $s\in\mathbb R.$ 
Consider  $\mathcal B=\mathfrak{aff}(\mathbb R)\oplus \mathfrak{aff}(\mathbb R)\oplus\mathfrak{aff}(\mathbb R)$ and a basis $(a_1,b_1,a_2,b_2,a_3,b_3)$ 
in which its Lie bracket reads
$[ a_j,  b_j]= b_j,$ $j=1,2,3$.
 The linear map $\psi: \mathcal G\to \mathcal B$,   $\psi(e_1)= b_1$, $\psi(e_2)= b_2$,  $\psi(e_3)= b_3$, 
$\psi(e_4)= s(a_1+ a_2+  a_3)$, $\psi(e_5)=s\frac{\sqrt{6}}{2}(a_2-a_1)$, 
$\psi(e_6)=\frac{\sqrt{2}}{2} (a_1+ s a_2 -2s a_3)$, is a homomorphism and $\det(\psi)=-\frac{\sqrt{3}}{2} s^2(5s+1)$. So, $\psi$ is an isomorphism between the Lie algebras $\mathcal G$ and $\mathcal B$,  except when $s\in\{0,-\frac{1}5\}$. 
Thus, any $\phi \in\mathfrak{gl}(3,\mathbb R)$ with $3$ distinct  real eigenvalues  gives rise to $\mathcal G_\phi=\mathcal G,$ 
when $s\notin \{0,-\frac{1}5\}$. Note that $\G$ is Frobenius if and only if $s\notin\{ 0,-\frac{1}5\}$.
\end{example}

\subsection{Nonderogatory $\phi\in\mathfrak{gl}(V)$ all of whose eigenvalues are complex}\label{Sect:complexeigenvalues}
\subsubsection{Nonderogatory $\phi\in\mathfrak{gl}(V)$ diagonalizable in $\mathbb C$}

Let $\phi\in\mathfrak{gl}(V)$ be nonderogatory. Suppose all the eigenvalues $\lambda_1,\bar \lambda_1,\dots,\lambda_{\frac{n}{2}}, \bar \lambda_{\frac{n}{2}}$ of $\phi$ are complex (nonreal) and distinct,  in which case  $n$ is even and $\phi$ is diagonalizable in $\mathbb C$. 
From Theorem \ref{thm:Jordanization} (B),  there is a basis of $V$ in which the matrix $\phi$ is of the form diag($J_1,\dots,J_{\frac{n}{2}}$), where each block $J_j$ is of the form   $\begin{pmatrix}Re(\lambda_j)&-Im(\lambda_j)\\ Im(\lambda_j)&Re(\lambda_j)\end{pmatrix},$ $Im(\lambda_j)\neq 0.$
Note that  each  $J_j$ is nonderogatory with characteristic polynomial $\chi_{J_j}(X)=(X-Re(\lambda_j))^2+Im(\lambda_j)^2.$ 
 Let $v_j$ be an  eigenvector of $\phi$ with corresponding eigenvalue $\lambda_j$.
Both $Re(v_j)$ and $Im(v_j)$ are in $\ker\chi_{J_j}(\phi) =:\mathcal E_{j}.$  Obviously, one has $\phi\mathcal E_{j}\subset  \mathcal E_j,$ $j=1,\dots,\frac{n}{2}.$
Set $\phi_j:=\phi_{\vert_{\mathcal E_j}}.$
Of course, $\mathbb R  [\phi_j]= \mathbb R \mathbb I_{\mathcal E_j}\oplus \mathbb R  \phi_j$.  In the basis 
$e_1=\mathbb I_{\mathcal E_j},$ $e_2=\phi_j,$ $e_3=Re(v_j),$ $e_4=Im(v_j)$, 
the Lie bracket of  $\mathcal  G_{{ \phi}_{j}}:=\mathbb R  [\phi_j]\ltimes \mathcal E_j$, is
$[e_1,e_3] = e_3,$ $[e_1,e_4] = e_4,$ $[e_2,e_3] = Re(\lambda_j ) e_3 +  Im(\lambda_j) e_4$, 
 $[e_2,e_4] = -  Im(\lambda_j) e_3 +Re(\lambda_j) e_4.$ 
 Lemma \ref{lem:splitting} applied to the direct sum  $V=\mathcal E_{1}\oplus\dots\oplus \mathcal E_{{\frac{n}{2}}}$, leads to $\mathcal G_\phi = \mathcal  G_{{\phi}_{1}}\oplus\cdots\oplus \mathcal  G_{{\phi}_{\frac{n}{2}}}.$ 
We identify $\mathcal E_{j}$ with $\mathbb R^2$ and show below that each  Lie ideal
 $\mathcal G_{{\phi}_{j}}$ is isomorphic to $\mathfrak{aff}(\mathbb C).$  
 \begin{lemma}\label{lem:isomorphism4D} Let $\phi\in\mathfrak{gl}(2,\mathbb R)$ with complex eigenvalues $\lambda, \;\bar\lambda,$ where $\lambda =r-is$  and $r, \;s\in\mathbb R,$ with $s\neq 0$. Then $\mathcal G_{\phi}$ is isomorphic to $\mathfrak{aff}(\mathbb C),$ as in Example \ref{ex:D01n}.\end{lemma}
\begin{proof}
Without loss of generality, set $\phi(\tilde e_1)=r\tilde e_1+s \tilde e_2$, $\phi(\tilde e_2)=-s\tilde e_1+r \tilde e_2$, in the canonical basis $(\tilde e_1, \tilde e_2)$ of $\mathbb R^2 .$ 
 In the basis $e_1' =\mathbb I_{\mathbb R^2},\; e_2'=\phi$, $e_3'=\tilde e_1,\; e_4'=\tilde e_2$, the Lie bracket of $\mathcal G_{\phi}$ reads
 $[e_1',e_3'] = e_3'$, $[e_1',e_4'] = e_4'$, $[e_2',e_3'] = re_3'+s e_4'$,  $[e_2',e_4'] =- se_3'+r e_4'$. In the new  basis $X_1:=e_1'$, $X_2:=-\frac{r}{s}e_1' + \frac{1}{s}e_2'$, $X_3:=pe_3' -qe_4'$, $X_4:=qe_3' + pe_4'$, with $p^2+q^2\neq 0$, we now have
$[X_1,X_3]=X_3$, $[X_1,X_4]=X_4$, $[X_2,X_3]=X_4$, $[X_2,X_4]=-X_3$, which is the Lie bracket of  $\mathfrak{aff}(\mathbb C).$ In other words, the invertible linear map $\psi : \mathfrak{aff}(\mathbb C) \to  \mathcal G_{\phi}$, $\psi(e_j)=X_j$, $j=1,2,3,4,$ is an isomorphism between the Lie algebras $\mathfrak{aff}(\mathbb C)$ and $  \mathcal G_{\phi}$.
  \end{proof}

\begin{proposition}\label{prop:classification-multiple-complexeigenvalues}
Suppose a nonderogatory $\phi\in\mathfrak{gl}(V)$ has $n$  distinct complex eigenvalues $\lambda_j,\bar\lambda_j,$  $j=1,\dots, \frac{n}{2}.$ Then the Lie algebra $\mathcal G_\phi:=\mathbb R[\phi]\ltimes V$ is isomorphic to the direct sum $\mathfrak{aff}(\mathbb C)\oplus \cdots \oplus \mathfrak{aff}(\mathbb C)$ of  $\frac{n}{2}$ copies of the Lie algebra $\mathfrak{aff}(\mathbb C)$.
\end{proposition}
\begin{proof}Suppose  a nonderogatory $\phi\in\mathfrak{gl}(V)$ has $n$ distinct complex (nonreal) eigenvalues $\lambda_j,\bar\lambda_j,$ $j=1,\dots, \frac{n}{2}.$  Then $\mathcal G_\phi$ is isomorphic to the direct sum 
  of the ideals $\mathcal  G_{{ \phi}_{j}}:=\mathbb R  [\phi_j]\ltimes \mathcal E_j$, $j=1,\dots,\frac{n}{2}$, where $\ker\chi_{\phi_j}(\phi_j) =:\mathcal E_{j},$ as above.   From Lemma  \ref{lem:isomorphism4D}, each  $\mathcal  G_{{\phi}_{j}}$ is isomorphic to $\mathfrak{aff}(\mathbb C)$. 
\end{proof}

\subsubsection{Example: the circular permutation of the  vectors of a basis}\label{Sect:exple-complexeigenvalues}
Here is a typical example of a nonderogatory $\phi\in\mathfrak{gl}(n,\mathbb R)$ with $n$ real and complex eigenvalues, hence diagonalizable in $\mathbb C$. 
It is 
given by the circular permutation of the canonical basis $\phi(\tilde e_i) = \tilde e_{i+1}$, for $i=1,\dots,n-1$ and $\phi(\tilde e_{n})=\tilde e_1$. Any vector $\tilde e_i$ of this basis is such that   $(\tilde e_i,\phi(\tilde e_i), \dots, \phi^{n-1}(\tilde e_i))$ is again a basis of $\mathbb R^n$. The  map $\phi$ is nonderogatory and its matrix in the above basis reads $[\phi]=E_{1,n}+\sum\limits_{i=1}^{n-1}E_{i+1,i}.$
Its characteristic polynomial is $\chi_\phi(X)=X^n-1$, up to a sign. So the eigenvalues are the complex nth roots of $1.$ They are $z_k=e^{ik\frac{2\pi}{n}}$, where $k=1,2,\dots,n.$
When $n=2,$ then it reads  $[\phi]=E_{1,2}+E_{2,1}$ and has the two distinct real eigenvalues $z_1=-1$ and $z_2=1.$  So $\mathcal G_\phi$ is isomorphic to 
$\mathfrak{aff}(\mathbb R)\oplus \mathfrak{aff}(\mathbb R).$ For $n=3,$ the eigenvalues are $z_1=1,$ $ z_2=-e^{i\frac{\pi}{3}}$, 
$ z_3=-e^{-i\frac{\pi}{3}}.$ So  $\mathcal G_\phi$ is isomorphic to $\mathfrak{aff}(\mathbb R)\oplus\mathfrak{aff}(\mathbb C)$. For $n=4,$ the eigenvalues are $z_1=-1,$ $z_2= 1,$ $z_3= i,$ $z_4= -i$, thus  $\mathcal G_\phi$ is isomorphic to $\mathfrak{aff}(\mathbb R)\oplus \mathfrak{aff}(\mathbb R)\oplus\mathfrak{aff}(\mathbb C).$  When $n=5$, there are one real and $4$ complex eigenvalues and hence we get
 $\mathcal G_\phi=\mathfrak{aff}(\mathbb R)\oplus\mathfrak{aff}(\mathbb C)\oplus\mathfrak{aff}(\mathbb C).$
For $n=7$, we have one real  and six complex eigenvalues, namely  $1, -e^{i\frac{\pi}{7}},  -e^{-i\frac{\pi}{7}}, e^{i\frac{2\pi}{7}}, e^{-i\frac{2\pi}{7}},  -e^{i\frac{3\pi}{7}},- e^{-i\frac{3\pi}{7}}.$
This gives rise to $\mathcal G_\phi=\mathfrak{aff}(\mathbb R)\oplus\mathfrak{aff}(\mathbb C)\oplus\mathfrak{aff}(\mathbb C) \oplus\mathfrak{aff}(\mathbb C).$

\subsubsection{Nonderogatory $\phi\in\mathfrak{gl}(V)$ non-diagonalizable in $\mathbb C$}\label{sect:nondiagonalizableinC}

As above, $V$ is a real vector space with $\dim V=n.$
In this section, we discuss the case of nonderogatory $\phi\in\mathfrak{gl}(V)$ all of whose eigenvalues are complex (nonreal, in which case, $n$ is even), 
but which are not diagonalizable in $\mathbb C.$  
First, consider the case where $\phi$ has only $2$ eigenvalues, say $z=r+is$ and $r-is$, with $s\neq 0$. In the same basis in which the matrix of $\phi$ is in the form $M_z$ as in (\ref{Jordancomplex}), consider the nonderogatory 
 $M_{0,1}=M_s+M_n$, where $M_s$, $M_n$ are as in (\ref{Ms-Mn-M01}).
We have
\begin{eqnarray}M_sM_n&=& -\sum\limits_{j=0}^{\frac{n}{2}-2}\Big(E_{2j+1,2j+4} - E_{2j+2,2j+3}\Big)= M_nM_s\;,
\end{eqnarray}
and, of course $[ M_{0,1},M_s]=[M_n ,M_s] = [ M_n,M_{0,1}]=0$. In particular,  $M_s$ and $M_n$ are respectively the 
semisimple and the nilpotent parts of $ M_{0,1}.$ So, $M_s$ and $M_n$ are both polynomials in $M_{0,1}$.
 Thus,  $M_z=r\mathbb I_{\mathbb R^n}+sM_s+M_n$ is also a polynomial in 
 $M_{0,1}$. This induces the following  equalities  
\begin{eqnarray}
\mathbb R[ M_{0,1}] = \mathbb R[ M_z] \text{  and }   \mathcal G_{M_z}=\mathcal G_{M_{0,1}} =\mathfrak{D}_{0,1}^n.
\end{eqnarray}
 We have thus proved the
\begin{proposition}\label{prop:isomorphismcomplexeigenvaluesnonsemiimple}
If a nonderogatory $\phi\in\mathfrak{gl}(V)$ has only $2$ eigenvalues which are  complex (thus conjugate),  and if $\phi$ is non-diagonalizable in $\mathbb C$, then $  \mathcal G_{\phi}=\mathfrak{D}_{0,1}^n.$
\end{proposition}
Finally, we deduce the following.
\begin{proposition}\label{prop:all-complex-eigenvalues-nondiagonalizable}
 If all the eigenvalues $\lambda_1,\bar\lambda_1,\dots,\lambda_p,\bar\lambda_p$ of a nonderogatory $\phi\in\mathfrak{gl}(V),$ are  complex, then $  \mathcal G_{\phi}$  is  isomorphic to the direct sum $\mathfrak{D}_{0,1}^{2k_1}\oplus\dots\oplus \mathfrak{D}_{0,1}^{2k_p}$ of copies of $ \mathfrak{D}_{0,1}^{2k_j}$, where $k_j$ is the multiplicity of the eigenvalues $\lambda_j$, $\bar \lambda_j$.
\end{proposition}
\begin{proof}Factorizing the characteristic polynomial of $\phi$, via Lemma \ref{lemma:factorization}, then successively applying the Primary Decomposition Theorem, Lemma \ref{lem:splitting},  together with Proposition \ref{prop:isomorphismcomplexeigenvaluesnonsemiimple}, yield the result.  
\end{proof}

\subsubsection{More on the Lie algebra $\mathfrak{D}_{0,1}^n$}\label{sect:more-on-D01n}

Without loss of generality,  
 we suppose that the form (\ref{Jordancomplex}) is achieved in the canonical basis $(\tilde e_1,\dots,\tilde e_n)$ of $\mathbb R^n.$ 
Note that $M_s$ and $M_n$ in  (\ref{Ms-Mn-M01}) satisfy $M_s^2=-\mathbb I_{\mathbb R^n},$ $M_n^{\frac{n}{2}}=0$, $(M_s M_n)^{\frac{n}{2}}=0$ and for any $j=1,\dots, \frac{n}{2}-1,$
\begin{eqnarray}  \Big(M_n\Big)^{j}=\sum\limits_{p=1}^{n-2j} E_{p,p+2j} \text{ and }
 M_s  \Big(M_n\Big)^{j}=-\sum\limits_{p=0}^{\frac{n}{2}-j-1}\Big( E_{2p+1,2p+2j+2}-  E_{2p+2,2p+2j+1}\Big)
 .\nonumber\end{eqnarray} 
Each $ M_{0,1}^{p}  = (M_s+M_n)^p=\sum\limits_{j=0}^p \complement_p^j \; (M_s)^j\; (M_n)^{p-j},$  $p=0,1,2,\dots,n-1$, is a linear combination of $\mathbb I_{\mathbb R^n}$, $M_s$, $(M_n)^{j},$ $M_s  (M_n)^{j}$, $j=1,\dots, \frac{n}{2}-1,$  where as above, $\complement_p^j=\frac{p!}{j!(p-j)!}$, $p\ge j$.  
Thus $e_1:=\mathbb I_{\mathbb R^n},$ $ e_2:=M_s,$ $ e_{2j+1}:=(M_n)^{j},$ $ e_{2j+2}:=-M_s  (M_n)^{j},$ $ j=1,\dots, \frac{n}{2}-1,$ is another basis of  $\mathbb R[M_{0,1}]$. 
In the basis ($e_s,\tilde e_s$, $s=1,\dots,n$), the Lie bracket of $\mathfrak{D}_{0,1}^n$ reads, for any  $q=0,1,\dots,\frac{n}{2}-1$  and $i=1,\dots,n,$

$ [e_{2q+1},\tilde e_i]=\tilde e_{i-2q},\;\; \text{ if }2q+1\leq i \leq n, \text{ and } [e_{2q+1},\tilde e_i]=0 \text{ otherwise, } $

$[e_{2q+2},\tilde e_i]=\left\{
\begin{array}{cll}
-\tilde e_{i-2q+1}, & \text{ if $i$ odd and } 2q+1\leq i \leq n-1, \\
\tilde e_{i-2q-1}, & \; \text{ if $i$ even and } 2q+2\leq i \leq n, \\
  0, & \; \text{ otherwise.}
\end{array}\right.
$

So the 2-form $\partial \tilde e_1^*=- \sum\limits_{k=1}^ne_{k}^*\wedge \tilde e_{k}^*$ is non-degenerate on  $\mathfrak{D}_{0,1}^n$.
 The codimension $2$ subspace  
$\mathcal N:=span\Big((M_n)^{j}, M_s  (M_n)^{j}, \; j=1,\dots, \frac{n}{2}-1
\Big)\ltimes \mathbb R^n$  is an ideal of $\mathfrak{D}_{0,1}^n$, as it  contains the derived ideal $\mathbb R^n =[\mathfrak{D}_{0,1}^n,\mathfrak{D}_{0,1}^n]$. The equalities $(M_n)^{\frac{n}{2}} = (M_s  M_n)^{\frac{n}{2}} =0$,  show that $\mathcal N$ is $\frac{n}{2}$-step nilpotent, and is in fact the nilradical of $\mathfrak{D}_{0,1}^n$.  Indeed, if we set $\mathfrak{F}:=\mathbb R \mathbb I_{\mathbb R^n} \oplus \mathbb R M_{0,1}$, then the vector space underlying  $\mathfrak{D}_{0,1}^n$ splits as the direct sum $\mathfrak{D}_{0,1}^n = \mathfrak{F}\oplus \mathcal N$, so that, any subspace of dimension higher than $n-2$, must meet  $\mathfrak{F}$ non-trivially and hence cannot be a nilpotent subalgebra of  $\mathfrak{D}_{0,1}^n$. Thus $\mathcal N$ is the biggest nilpotent ideal of $\mathfrak{D}_{0,1}^n$.
Altogether, $\mathfrak{D}_{0,1}^n$ is an  indecomposable, non-completely solvable (the adjoint of $M_s$ has complex eigenvalues) $2$-solvable Frobenius Lie algebra with a codimension 2 non-Abelian nilradical $\mathcal N$.
One sees that $\mathfrak{D}_{0,1}^n$ is not isomorphic to 
 $\mathfrak{D}_0^n$  which is completely solvable and has a codimension $1$ non-Abelian nilradical (except for $\mathfrak{D}_0^1=\mathfrak{aff}(\mathbb R)$).

\subsection{Derivations and automorphisms of $\mathcal G_\phi$}
 
\begin{proposition}\label{prop:derivations} Let $\phi\in \mathfrak{gl}(V)$ be nonderogatory.
Set  
\begin{eqnarray}
\mathfrak{N} :=\{N\in\mathfrak{gl}(V), \text{ such that } [ N,\mathbb K[\phi]]\subset \mathbb K[\phi]\}=\mathcal N_{\mathfrak{gl}(V)}(\mathbb K[\phi]).\end{eqnarray}
Up to isomorphism, the Lie algebra $Der(\mathcal G_\phi)$ of derivations of   $\mathcal G_\phi$  is given by
\begin{eqnarray}
Der(\mathcal G_\phi)=\mathfrak{N} \ltimes  V.\end{eqnarray}
More precisely,  $D\in Der(\mathcal G_\phi)$ if and only if there are  $x_D\in V$ and $h \in\mathfrak{N},$ such that $D$
 is of the form
$D(a+x) =[h,a]+ ax_D+h(x),$
for every $a\in\mathbb K[\phi]$ and $x\in V.$ \end{proposition}
\begin{proof}
Let $D\in Der(\mathcal G_\phi)$. As the derived ideal $V=[\mathcal G_\phi,\mathcal G_\phi]$ is preserved
by $D$, we write $D(x)=h(x)$ and
$D(a)=D_{1,1}(a)+D_{1,2}(a),$  for $a\in\mathbb K[\phi]$,  $x\in
 V,$ where $D_{1,1}\in\mathfrak{gl}(\mathbb K[\phi])$, $h\in\mathfrak{gl}(V)$ and $D_{1,2}: \mathbb K[\phi]\to V$ is linear.
Now, setting  $e_1:=\mathbb I_{V}$ and  $x_D:=D_{12}(e_{1})$, the equality $D_{12}(a) =ax_D,$ for any $a\in \mathbb K[\phi]$, follows:
$0=D([e_{1}, a])$$=[D(e_{1}),a]+ [e_{1},D(a)] $$=
 [x_D,a] + [e_{1},D_{12}(a)]$$
=D_{12}(a)-ax_D.$ 
We also have
$h(ax) =  D[a,x]=
[D_{11}(a),x]+[a,h(x)] =D_{11}(a)x+ah(x)$, 
thus entailing $[h,a] = D_{11}(a) \in
\mathbb K[\phi],$ for any $a\in\mathbb K[\phi].$
So the equality $ D_{11}(a): = [h,a]$ stands as the definition of $D_{11} (a)$ and implies  $h\in  \mathfrak N$. 
Thus for any $D$ in $ Der(\mathcal G_\phi)$, there are $x_D\in V$,  $h\in  \mathfrak N$ such that $D(a+x)=[h,a]+h(x)+ax_D.$
 Conversely, any $(h,\hat x)\in\mathfrak{N}\times V$  defines a unique 
 $D\in Der(\mathcal G_\phi)$ by the formula $D(a+x)=[h,a]+h(x)+a\hat x.$ 
We get the invertible linear $\psi : Der(\mathcal G_\phi)\to \mathfrak N\ltimes V$, $D\mapsto (h,x_D)$.
 As easily seen,  for any $D_j\in Der(\mathcal G_\phi)$  given by $(h_j, x_{D_j})$, $j=1,2,$ we have
$\psi([D_1,D_2])= ([h_1,h_2],h_1(x_{D_2})
- h_2(x_{D_1}))=[\psi(D_1),\psi(D_2)] $. 
Hence $\psi$ is an isomorphism between the Lie algebras $Der(\mathcal G_\phi)$ and $ \mathfrak N\ltimes V.$
\end{proof}
 Note that in Proposition \ref{prop:derivations}, the inner derivations of $\mathcal G_\phi$ are those for which the component $h$ belongs to $\mathbb K[\phi]$.
\begin{example}\label{ex:derivationsofD0n} The space $Der(\mathfrak{D}_0^n)$ of derivations of $\mathfrak{D}_0^n$.

For $n=2,$
the normalizer of  $\mathbb R[E_{1,2}]$ in $\mathfrak{gl}(2,\mathbb R),$ is the 3-dimensional 
real algebra spanned by $E_{1,1},E_{1,2},E_{22}$, which is $1$ dimension higher than   $\mathbb R[E_{1,2}]$. For example, $E_{2,2}$ is not in  $\mathbb R[E_{1,2}]$  and will thus act as an outer derivation on $\mathfrak{D}_0^2.$ Considering elements of $\mathbb R^2$ as inner derivations, we see that $Der(\mathfrak{D}_0^2)$ is spanned by $E_{1,1},E_{1,2},E_{22}, \tilde e_1,\tilde e_2.$ 
For $n=3,$
the normalizer of  $\mathbb R[E_{1,2}+E_{2,3}]$ in $\mathfrak{gl}(3,\mathbb R),$ is the $5$-dimensional algebra  of matrices spanned by $E_{1,1}-E_{3,3},$ $E_{22}+2E_{33}$, $E_{1,2}$, $E_{2,3}$, $E_{1,3}$.  For example,  $E_{1,1}-E_{3,3},$ $E_{22}+2E_{33}$, are not in  $\mathbb R[E_{1,2}+E_{2,3}]$, so they represent outer derivations of $\mathfrak{D}_0^3.$ Counting elements of $\mathbb R^3$ in, as inner derivations,  we get $\dim Der(\mathfrak{D}_0^3)=8.$
More generally, for a given $n\ge 4$, if we set $M_0:=E_{1,2}+E_{2,3}+\dots+E_{n-1,n}$ as in (\ref{principalnilpotent}), then the normalizer of $\mathbb R[M_0]$ in $\mathfrak{gl}(n,\mathbb R),$  is of dimension $2n-1$ and is spanned by 
$D_k:=E_{1,k}-\sum\limits_{j=3}^{n-k+1}(j-2)E_{j,j+k-1}$, $D_{k}':=E_{2,k+1}+ \sum\limits_{j=3}^{n-k+1}(j-1)E_{j,j+k-1}$, $k=1,\dots,n-2$ and $D_{n-1}:=E_{1,n-1}$, $D_{n-1}':=E_{2,n}$, $D_n:=E_{1,n}.$ For example, $D_1$ and $D_1'$ are not in $\mathbb R[M_0]$ and will act as non-trivial outer derivations of $\mathfrak{D}_0^n$.  In particular $\mathbb R[M_{0}]$ is not a Cartan subalgebra of $\mathfrak{gl}(n,\mathbb R)$ and $\dim Der(\mathfrak{D}_0^n)=3n-1.$
\end{example}

\begin{example}\label{ex:derivationsofD01n} On the space $Der(\mathfrak{D}_{0,1}^n)$ of derivations of $\mathfrak{D}_{0,1}^n$.

For $n=4$, consider the $n\times n$ matrix $M_{0,1}=E_{2,1} -E_{1,2}+E_{4,3}-E_{3,4}+E_{1,3} +E_{2,4} $ as in Example \ref{ex:D01n}. The normalizer of $\mathbb R[M_{0,1}]$ in $\mathfrak{gl}(4,\mathbb R),$ is of dimension $6$  and is spanned by the matrices  $E_{1,1}+E_{2,2},$ $E_{3,3}+E_{4,4},$ $E_{1,2}-E_{2,1},$ $E_{3,4}-E_{4,3},$  $E_{1,3}+E_{2,4},$ $E_{2,3}-E_{1,4}$. In particular the two matrices $E_{1,1}+E_{2,2},$ $E_{3,3}+E_{4,4},$ are not elements of $\mathbb R[M_{0,1}]$. So they both represent non-trivial outer derivations of $\mathfrak{D}_{0,1}^4$. 
More generally, for any $n\ge 6$, we let again $M_{0,1}$ stand for the $n\times n$ matrix $M_{0,1}=M_s+M_n$ as in (\ref{Ms-Mn-M01}).
 The $n\times n$ matrix   $
Z_1:=E_{1,1}+E_{2,2}-\sum\limits_{j=2}^{\frac{n}{2}-1}(j-1)(E_{2j+1,2j+1}+E_{2j+2,2j+2})
$
is in the normalizer of $\mathbb R[M_{0,1}]$ in $\mathfrak{gl}(n,\mathbb R)$, but not in $\mathbb R[M_{0,1}]$. Indeed, for any $s\ge 1,$ we have $  [Z_1, \Big(M_n\Big)^{s}] 
= s\sum\limits_{j=1}^{n-2s}E_{j,j+2s} = s \Big(M_n\Big)^{s}$, $[Z_1, M_s] =0$ and  $[Z_1, M_s\Big(M_n\Big)^{s}] = M_s[Z_1, \Big(M_n\Big)^{s}] + [Z_1, M_s] \Big(M_n\Big)^{s}=  sM_s\Big(M_n\Big)^{s}.$ So the linear map $M\mapsto [Z_1,M]$ preserves $\mathbb R[M_{0,1}]$ and has a diagonal matrix diag($0,0,1,2,\dots,\frac{n}{2}-1, 1,2,\dots,\frac{n}{2}-1$) in the basis $(\mathbb I_{\mathbb R^n}, M_s, M_n,(M_n)^2,\dots,(M_n)^{\frac{n}{2}-1}, M_s M_n, M_s(M_n)^2,\dots,M_s(M_n)^{\frac{n}{2}-1})$ and  $Z_1\notin\mathbb R[M_{0,1}]$, given that $\mathbb R[M_{0,1}]$ is Abelian. In particular $\mathbb R[M_{0,1}]$ is not a Cartan subalgebra of $\mathfrak{gl}(n,\mathbb R)$. Also, $Z_1$ will act as an outer derivation of $\mathfrak{D}_{0,1}^n$.
\end{example}
We have the following.
\begin{theorem}\label{prop:abeliannilrad}Let $\phi\in\mathfrak{gl}(n,\mathbb R)$ be nonderogatory. The following are equivalent. 

(1) Every derivation of $\mathcal G_\phi$ is an inner derivation.

(2) $\mathbb R[\phi]$ is a Cartan subalgebra of $\mathfrak{gl}(n,\mathbb R)$.

(3)  $\phi$ has $n$ distinct  (real or  complex) eigenvalues. 

(4) The nilradical of $\mathcal G_\phi$ is Abelian.

(5)  $\mathcal G_\phi$ is the direct sum 
of only copies of $\mathfrak{aff}(\mathbb R)$ and $\mathfrak{aff}(\mathbb C)$.
\end{theorem}
\begin{proof}The equivalence between (1) and (2) directly follows from Proposition \ref{prop:derivations}. Indeed, every derivation of $\mathcal G_\phi$ is inner, if and only if the normalizer of $\mathbb R[\phi]$ in  $\mathfrak{gl}(n,\mathbb R),$  coincides with $\mathbb R[\phi]$. 
The equivalence between (2), (3) and (5) has been shown in Theorem \ref{thm:cartansubalgebras}.
The proof that (5) implies (4) directly follows from the fact that each copy, of either  $\mathfrak{aff}(\mathbb R)$ or $\mathfrak{aff}(\mathbb C)$, has an Abelian nilradical.  We now prove that (4) implies (5).
In the decomposition of $\mathcal G_\phi$ (Theorem \ref{thm:classification}), the nilradical of $\mathcal G_\phi$ is the sum of the nilradicals of  its ideals. From Examples \ref{ex:derivationsofD0n}, \ref{ex:derivationsofD01n}, the nilradicals of $\mathcal D_0^k$  and $\mathcal D_{0,1}^k$  are not Abelian, 
for any $k\ge 2$. So in order for 
$\mathcal G_\phi$ to have an Abelian nilradical, it must not contain an ideal isomorphic $\mathcal D_0^k$  or $\mathcal D_{0,1}^k$, $k\ge 2$.  
\end{proof}

Note that Theorem \ref{prop:abeliannilrad} is in agreement with the classification of Cartan subalgebras of $\mathfrak{gl}(n,\mathbb R)$ supplied by Kostant (\cite{kostant}) and Sugiura (\cite{sugiura}).
  Theorem \ref{prop:abeliannilrad}  also implies that  
 when $\mathbb R[\phi]$ is a Cartan subalgebra of $\mathfrak{gl}(n,\mathbb R)$, then  $\mathcal G_\phi$ is a part of the 2-solvable Lie algebras studied in \cite{alvarez-abelian-nilradical}, \cite{ndogmo-winternitz}.

\subsection{Proof of Theorem \ref{thm:cartansubalgebras}}\label{proofofCartan-subalgebras-sl(n,R)} 
From Proposition \ref{prop:abeliannilrad}, for a given $n$, the number of isomorphism classes of $2$-solvable Frobenius Lie algebras 
of dimension $2n$  of the form $\mathcal G_\phi :=\mathbb R[\phi]\ltimes \mathbb R^n$, where $\mathbb R[\phi]$ is a Cartan subalgebra of $\mathfrak{gl}(n,\mathbb R),$  is
 exactly $[\frac{n}{2}]+1$. Indeed, one can look at $[\frac{n}{2}]+1$ as the number (counting from zero)
of possible copies of $\mathfrak{aff}(\mathbb C)$ that one can count in a decomposable Lie algebra 
containing only copies of either $\mathfrak{aff}(\mathbb C)$ or $\mathfrak{aff}(\mathbb R)$.
On the other hand, from e.g. \cite{sugiura}, there are exactly $[\frac{n}{2}]+1$ non-conjugate Cartan subalgebras of 
$\mathfrak{gl}(n,\mathbb R).$ So we have derived Theorem \ref{thm:cartansubalgebras}, in a simple and direct way. More precisely we prove it as follows.
\begin{proof}
 Let $\mathfrak{h}$ be a Cartan subalgebra of
 $\mathfrak{sl}(n,\mathbb R)$ with a $k$-dimensional toroidal part. From \cite{sugiura}, up to conjugacy under an element of the Weyl group, there is a basis ($\tilde e_1,\dots, \tilde e_n$), considered here as the canonical basis of $\mathbb R^n$, in which   $\mathfrak{h}$ is of the form
$\mathfrak{h}=\Big\{\begin{pmatrix} D_1&-D_2&{\mathbf 0}\\
D_2 &D_1&{\mathbf 0}\\
{\mathbf 0}&{\mathbf 0}&D_3
\end{pmatrix}\Big\}$, where $D_1=\text{diag}(h_1,\dots,h_{k})$, 
$D_2=\text{diag}(h_{k+1},\dots,h_{2k}),$ $D_3=\text{diag}(h_{2k+1},\dots,h_{n})$ with $ h_j\in\mathbb R, j=1,\dots,n$. So that, $\mathfrak{h}$ is conjugate to 
$\mathfrak{h}'=\Big\{M:=\text{diag}(D_1',\dots,D_k', D_3),$  with  
$D_j'=\begin{pmatrix} { h_j} & -h_{k+j} \\
h_{k+j} &{ h_j}
\end{pmatrix}$ and $ D_3=\text{diag}(h_{2k+1},\dots,h_{n}), 
 h_j\in\mathbb R, j=1,\dots,n\;\Big\}$,
 obtained by reordering the canonical basis of 
$\mathbb R^n$ into
 $(\tilde e_1, \tilde e_{k+1}, \tilde e_2, \tilde e_{k+2}, \dots,\tilde e_k,
 \tilde e_{2k},\tilde e_{2k+1},\tilde e_{2k+2},\dots,\tilde e_{n}).$ An $M$ in $\mathfrak{h}'$ is
 nonderogatory if and only if $(h_i,h_{k+i}) \neq (h_j,h_{k+j})$, whenever $i\neq j$ and $h_{2k+s}\neq h_{2k+l}$ whenever $s\neq l.$ But the existence of an (regular) element satisfying these
conditions is guaranteed by the fact that $\mathfrak{h}$ is a Cartan subalgebra. Hence, there exists a nonderogatory $M$ with the $n$ distinct eigenvalues $h_1+ih_{k+1}, h_1-ih_{k+1},\dots , h_k+ih_{2k}, h_k-ih_{2k},$ $h_{2k+1},\dots, h_{n}$, such that, up to a conjugation, $\mathbb I_{\mathbb R^n}\oplus\mathfrak{h}=\mathbb R[M]=\mathfrak{B}_{\mathfrak{h}}$. So (a) implies (b).   Furthermore, Theorem \ref{thm:classification} ensures that the Lie algebra $\mathfrak{B}_{\mathfrak{h}}\ltimes \mathbb R^n$ is isomorphic to the direct sum $\mathfrak{aff}(\mathbb C)\oplus\dots\oplus\mathfrak{aff}(\mathbb C)\oplus\mathfrak{aff}(\mathbb R)\oplus\dots\oplus\mathfrak{aff}(\mathbb R)$ of $k$ copies of $\mathfrak{aff}(\mathbb C)$ and $(n-k)$ of $\mathfrak{aff}(\mathbb R)$. In fact, the equivalence between (b) and (c) has already been proved  by Theorem \ref{thm:classification}. 
Conversely, suppose $M\in\mathfrak{gl}(n,\mathbb R)$ is nonderogatory and has $n$ distinct eigenvalues, 2k of which are complex. 
From Proposition \ref{prop:abeliannilrad}, $\mathbb R[M]$ is a Cartan subalgebra of $\mathfrak{gl} (n,\mathbb R)$ and so (b) implies (a). Furthermore, using the Primary Decomposition Theorem and results from Section \ref{Sect:complexeigenvalues}, we put $\mathbb R[M]$ in the form $\mathfrak{h}'$.
\end{proof}

\section{Classification of low dimensional 2-solvable Frobenius Lie algebras}\label{sect:6D}
Applying Theorem \ref{thm:classification} and the analysis done above,
we get Theorem \ref{thm:classification-D6} below which provides, up to isomorphism, a complete list of all 2-solvable Frobenius Lie algebras of dimension $2$, $4,$ $6$ or $8.$ 
We note that $\mathfrak{D}_0^2$, $\mathfrak{aff}({\mathbb C})$
and $\mathfrak{aff}({\mathbb R})\oplus \mathfrak{aff}({\mathbb R})$ correspond to the family PHC7 of \cite{4d-para-hypercomplex} 
and to S11, S8 and S10 respectively in \cite{snow}. We also note that $\mathfrak{D}_0^2\oplus \mathfrak{D}_0^2$ is missing from the list in  \cite{winternitz-zassenhaus}, see Section  \ref{revisit-Winternitz-Zassenhaus}. It is also worth recalling here, that the Lie algebras $\mathcal G_{4,2}$ of Example  \ref{example-general-nD} and $\mathfrak{h}_{4,2}$ of Example \ref{example2a}  are isomorphic via the isomorphism (\ref{isom:Gn2-hn2}).
\begin{theorem} \label{thm:classification-D6} A 2-solvable Frobenius Lie algebra of dimension $\leq 6$ is either isomorphic to
 $\mathfrak{aff}({\mathbb R})$, $\mathfrak{D}_0^2$, $\mathfrak{aff}({\mathbb C})$, $\mathfrak{D}_0^3$, $\mathcal G_{3,1}$ as in Example \ref{example-general-nD}, or to one of the direct sums 
$\mathfrak{aff}({\mathbb R})\oplus \mathfrak{aff}({\mathbb R})$,  
 $\mathfrak{D}_0^2 \oplus\mathfrak{aff}({\mathbb R})$, 
$\mathfrak{aff}({\mathbb C}) \oplus \mathfrak{aff}({\mathbb R})$, 
$\mathfrak{aff}({\mathbb R})\oplus \mathfrak{aff}({\mathbb R}) \oplus \mathfrak{aff}({\mathbb R})$ of their copies. 
%
In dimension $8$, there are $14$  non-isomorphic  $2$-solvable Frobenius Lie algebras, $9$ of which are of the form $\mathcal G_M$, for some  nonderogatory $M\in\mathfrak{gl}(4,\mathbb R)$, namely
${\mathfrak D}_{0,1}^{4}$,  $\mathfrak{D}_0^4$,  the direct sums $\mathfrak{D}_0^3 \oplus \mathfrak{aff}({\mathbb R}),$
 $\mathfrak{D}_0^2\oplus \mathfrak{D}_0^2$, 
 $\mathfrak{D}_0^2\oplus \mathfrak{aff}({\mathbb C})$, 
 $\mathfrak{D}_0^2\oplus \mathfrak{aff}({\mathbb R}) \oplus \mathfrak{aff}({\mathbb R})$,
 $\mathfrak{aff}({\mathbb C})\oplus \mathfrak{aff}({\mathbb C})$,
$\mathfrak{aff}({\mathbb C})\oplus \mathfrak{aff}({\mathbb R}) \oplus \mathfrak{aff}({\mathbb R})$,
$\mathfrak{aff}({\mathbb R}) \oplus \mathfrak{aff}({\mathbb R})\oplus \mathfrak{aff}({\mathbb R}) \oplus \mathfrak{aff}({\mathbb R})$,
and $5$ of which are   not given by a nonderogatory $M\in\mathfrak{gl}(4,\mathbb R)$, namely 
$\mathcal G_{4,1}$ of Example \ref{example-general-nD}, $\mathfrak h_{4,2}$ and $\mathfrak h_{4,3}$ of Example \ref{example2a}, $\mathcal G_{4,4}'$  of Example \ref{example23}, $\mathcal G_{3,1}\oplus \mathfrak{aff}({\mathbb R})$.  
\end{theorem}
\subsection{Proof of Theorem \ref{thm:classification-D6}}
First, let us recall the classification list of nilpotent 3-dimensional associative real algebras $A_{4,j},$ $j=1,2,\dots,6,$ given in \cite{degraaf}. We will need the commutative  ones for  the proof of Theorem \ref{thm:classification-D6}.  They have a basis $(a,b,c)$ in which the (non-zero) products are as follows. 
\begin{eqnarray}\label{degraaflist}  A_{3,1}:  &&(A_{3,1})^2=0,
\nonumber\\
 A_{3,2}: && a^2=c,    
 \nonumber\\ 
  A_{3,3}^s: &&a^2=sc, \;\; b^2=c, \; \;s\in\mathbb R,  \;\; s\neq 0,  
\nonumber\\ 
A_{3,4}^s: && a^2=s c, \; \;b^2=c,\;\; ab=c,  \; \;s\in\mathbb R,  
\nonumber\\
 A_{3,5}:&& ab=c, \;\; ba=-c,  
\nonumber\\
A_{3,6}: &&   
 a^2=b, \;\; ab=ba=c.
\end{eqnarray}
Note that  $A_{3,4}^s$ and  $A_{3,5}$ are not commutative, since $ab\neq ba.$ As noted in \cite{degraaf},   $A_{3,3}^s$ and   $A_{3,3}^t$ are isomorphic if and only if  $t=\epsilon^2 s$ for some $\epsilon\neq 0.$ In our case where the ground field is $\mathbb R,$ there are only two isomorphism classes corresponding to $s >0$, say,  $A_{3,3}^{1}$ and $s<0$, say  $A_{3,3}^{-1}$.

The above being said, let us now dive into the proof of Theorem \ref{thm:cartansubalgebras}. 
 
\noindent
The case $n=1$ is trivial, as every 2-dimensional non-Abelian Lie algebra is isomorphic to $\mathfrak{aff}(\mathbb R) =\mathcal G_\psi$, where $\psi=\mathbb I_{\mathbb R}$, as in Example \ref{ex:D0n}.
 Let $\mathcal G $ be a 2-solvable Frobenius Lie algebra of dimension $2n$, with $2\leq n\leq 3.$ Write $\mathcal G=\mathfrak{B}\ltimes \mathbb R^n$ where $\mathfrak{B}$ is an $n$-dimensional MASA  of $\mathfrak{gl}(n,\mathbb R)$. Note that, when $n\le 3,$ then $[\frac{n^2}{4}]+1=n$, so that from Jacobson's theorem  (\cite{jacobson-schur}), every $n$-dimensional Abelian subalgebra of $\mathfrak{gl}(n,\mathbb R)$ is a MASA. Further set $\mathfrak{B}= \mathbb R \mathbb I_{\mathbb R^n} \oplus L$, where  $L$ is a MASA of $\mathfrak{sl}(n,\mathbb R)$. 
When $n=2,$ then $L=\mathbb R M$, for some nonzero $M\in\mathfrak{sl} (2,\mathbb R)$. But every nonzero $M\in\mathfrak{sl} (2,\mathbb R)$ is nonderogatory. 
Thus, $\mathfrak{B}= \mathbb R [M]$ and  $\mathfrak{B}\ltimes\mathbb R ^2 = \mathcal G_M$.
From Theorem \ref{thm:classification},  $\mathcal G_M$  is isomorphic to $\mathfrak{D}_0^2$, $\mathfrak{aff} (\mathbb C)$, or $\mathfrak{aff}(\mathbb R)\oplus\mathfrak{aff}(\mathbb R)$.
For $n=3,$ using the Lie algebras of lower dimensions listed above, the list of decomposable 2-solvable Frobenius Lie algebras of dimension $6$ simply reads
 $\mathfrak{D}_0^2 \oplus\mathfrak{aff}({\mathbb R})$, 
$\mathfrak{aff}({\mathbb C}) \oplus \mathfrak{aff}({\mathbb R})$, 
$\mathfrak{aff}({\mathbb R})\oplus \mathfrak{aff}({\mathbb R}) \oplus \mathfrak{aff}({\mathbb R})$. As regards the non-decomposable ones, they are all of the form $\mathfrak{g}=(\mathbb R\mathbb I_{\mathbb R^3}\oplus \mathcal A)\ltimes \mathbb R^3$  with $\mathcal A$ a $2$-dimensional AID MASA of $\mathfrak{sl}(3,\mathbb R),$ since $n$ is odd. We have $2$ possibilities. The first is $\mathcal A^2\neq 0$, in which  case $\mathbb R \mathbb I_{\mathbb R^3}\oplus \mathcal A=\mathbb R[M],$ yielding  the only solution $M=M_0$ and $\mathfrak{g}=\mathfrak{D}_0^3.$ 
The second possibility is $\mathcal A^2= 0.$ From Theorem \ref{thm:class2-kravchuk-(n-1)-0-1}, $\mathcal A$  is unique and the corresponding Lie algebra is $\mathbb R\mathbb I_{\mathbb R^3}\oplus \mathcal A=\mathfrak{g}_{3,1},$ as in Example \ref{example-general-nD}. 
For the case $n=4,$ again using the Lie algebras of lower dimensions, we get the following list of decomposable 2-solvable Frobenius Lie algebras of dimension $8$:
$\mathfrak{D}_0^3 \oplus \mathfrak{aff}({\mathbb R}),$
 $\mathfrak{D}_0^2\oplus \mathfrak{D}_0^2$, 
 $\mathfrak{D}_0^2\oplus \mathfrak{aff}({\mathbb C})$, 
 $\mathfrak{D}_0^2\oplus \mathfrak{aff}({\mathbb R}) \oplus \mathfrak{aff}({\mathbb R})$,
 $\mathfrak{aff}({\mathbb C})\oplus \mathfrak{aff}({\mathbb C})$,
$\mathfrak{aff}({\mathbb C})\oplus \mathfrak{aff}({\mathbb R}) \oplus \mathfrak{aff}({\mathbb R})$,
$\mathfrak{aff}({\mathbb R}) \oplus \mathfrak{aff}({\mathbb R})\oplus \mathfrak{aff}({\mathbb R}) \oplus \mathfrak{aff}({\mathbb R})$ and  $\mathcal G_{3,1}\oplus \mathfrak{aff}({\mathbb R}).$
We put the non-decomposable ones in the form $\mathfrak{g}=(\mathbb R\mathbb I_{\mathbb R^4}\oplus \mathcal A)\ltimes \mathbb R^4$  where $\mathcal A$ is either an AID or NAID $3$-dimensional MASA of $\mathfrak{sl}(4,\mathbb R).$ 
First,  in the case where $\mathcal A$ is a  NAID  MASA of $\mathfrak{sl}(4,\mathbb R),$ we write $\mathcal A = \mathbb RM_s\oplus \mathcal A'$, where $\mathcal A'$ a $2$-dimensional MANS of $\mathfrak{sl}(2,\mathbb C).$  Using a direct approach and writing elements of  $\mathcal A'$ as upper triangular $4\times 4$ matrices commuting with $M_s$, one easily sees that $\mathcal A'=\{m_1(E_{1,3}+E_{2,4})+m_2(E_{2,3}-E_{1,4}), m_1,m_2\in\mathbb R\}$. This leads to  $\mathbb R\mathbb I_{\mathbb R^4}\oplus \mathcal A=\mathbb R[M_{0,1}]$ and 
$(\mathbb R\mathbb I_{\mathbb R^4}\oplus \mathcal A)\ltimes \mathbb R^4=\mathfrak{D}_{0,1}^4,$ as in Example \ref{ex:D01n}. Note that using lower triangular matrices leads to $\mathbb R[(M_{0,1})^T]$ which is conjugate to $\mathbb R[M_{0,1}]$.
Now in the case  where $\mathcal A$ is an AID, 
there are $3$ possibilities. (A) $\mathcal A^3\neq 0$ and $\mathcal A^4= 0$, there exists $a\in\mathcal A$ such that $a^3\neq 0$ and $a^4=0,$ so that $\mathbb R\mathbb I_{\mathbb R^4}\oplus \mathcal A=\mathbb R[a].$ 
From Theorem \ref{thm:classification},  we have $\mathbb R\mathbb I_{\mathbb R^4}\oplus \mathcal A=\mathbb R[M_0]$ and  $(\mathbb R\mathbb I_{\mathbb R^4}\oplus \mathcal A)\ltimes \mathbb R^4=\mathfrak{D}_{0}^4$. This corresponds to the case $A_{3,6}$ in \cite{degraaf}, as $A_{3,6}=span(a, a^2=b, a^3=c).$
(B) The second possibility is $\mathcal A^3=0$ and $\mathcal A^2\neq 0.$ 
 From Lemma \ref{2-nilp-MASA}, we have $\dim Im(\mathcal A^2)=1.$ This splits into 3 cases (\ref{degraaflist}):  
(a)  $A_{3,2}$: with $(A_{3,2})^2=\mathbb R c$.
 A representative of this class is $\mathcal{P}_{4,2}$ as in Example \ref{example2a} with $a=E_{1,2}+E_{2,4}$, $b=E_{1,3}$, $c=E_{1,4}$, so the corresponding Lie algebra is   $(\mathbb R\mathbb I_{\mathbb R^4}\oplus A_{3,2})\ltimes \mathbb R^4=\mathfrak{h}_{4,2}$.  
(b)  $A_{3,3}^1$, with  $(A_{3,3}^1)^2=\mathbb Rc.$  
 A representative of this class is $\mathcal{P}_{4,3}$ as in Example \ref{example2a} with $a=E_{1,2}+E_{2,4}$, $b=E_{1,3}+E_{3,4}$, $c=E_{1,4}$, and   $(\mathbb R\mathbb I_{\mathbb R^4}\oplus A_{3,3}^1)\ltimes \mathbb R^4=\mathfrak{h}_{4,3}$.  
(c)  $A_{3,3}^{-1}$ :  
A representative of this class is $L_{4,4}'$ as in Example \ref{example23} with $a=E_{1,2}+E_{3,4}-E_{1,3}-E_{2,4},$
$b=E_{1,2}+E_{3,4}+E_{1,3}+E_{2,4},$  $c=2 E_{1,4}$,  so that   $(\mathbb R\mathbb I_{\mathbb R^4}\oplus A_{3,3}^{-1})\ltimes \mathbb R^4=\G_{4,4}'$.  
(C) The last case is given by the condition $\mathcal A^2=0$, which, from Theorem \ref{thm:class2-kravchuk-(n-1)-0-1}, provides a unique $\mathcal A$ and $(\mathbb R\mathbb I_{\mathbb R^4}\oplus \mathcal A)\ltimes\mathbb R^4=\mathfrak{g}_{4,1},$ as in Example \ref{example-general-nD}. Note that this corresponds to  $\mathcal A=A_{3,1}$ in \cite{degraaf}. 
 \qed 
\subsection{A discussion on the lists  by  Winternitz and  Zassenhaus }\label{revisit-Winternitz-Zassenhaus}  

We revisit the classification lists of MASAs of $\mathfrak{sl}(n,\mathbb R)$, $n=3,4$ provided in  \cite{winternitz} and  \cite{winternitz-zassenhaus}. A systematic comparison shows a match with our classification, except for one missing item which we complete and correct some misprint on another item.

\noindent
For $n=3,$ according to \cite{winternitz}, there are  six classes of non-mutually conjugate MASAs   of $\mathfrak{sl}(3,\mathbb R)$. With the same notation as in  \cite{winternitz}, we denote them by  $L_{2,i}$, $i=1,\dots,6.$

\noindent
(1) The first is  $L_{2,1}:=\{  diag(k_1+k_2, -k_1+k_2,-2k_2), \;k_1,k_2\in\mathbb R\}.$ 

\noindent
We note that  $L_{2,1}:=\{  (k_2-k_1)S_{2,1}^0 +\frac{1}{2}(k_1+3k_2) S_{2,1} +\frac{3}{2}(k_1-k_2) S_{2,1}^2, \;k_1,k_2\in\mathbb R\}$$\subset  \mathbb R[ S_{2,1}],$ where $S_{2,1}=diag(1,0,-1),$  
 and  $ \mathbb R \mathbb I_{\mathbb R^n} \oplus L_{2,1}= \mathbb R[ S_{2,1}].$  So  $(\mathbb R \mathbb I_{\mathbb R^n} \oplus L_{2,1})\ltimes \mathbb R^3$ is isomorphic to $ \mathfrak{aff}(\mathbb R)\oplus \mathfrak{aff}(\mathbb R)\oplus \mathfrak{aff}(\mathbb R)$, as   
 $S_{2,1}$ is nonderogatory with $3$ eigenvalues.

\noindent
(2) In $L_{2,2}:=\Big\{k_1(E_{1,1}+E_{2,2}-2E_{3,3}) +k_2(E_{1,2}-E_{2,1}), k_1, k_2\in\mathbb R \Big\},$ every element  
  is a nonderogatory matrix with one real eigenvalue $-2k_1$ and two complex conjugate eigenvalues $k_1+ik_2$ and $k_1-ik_2,$ except when $k_2=0.$ 
For instance, 
$ S_{2,2}:=E_{1,2}-E_{2,1}\in L_{2,2}$ is nonderogatory,   
so  $\mathbb R \mathbb I_{\mathbb R^3} \oplus L_{2,2} =\mathbb R[ S_{2,2}].$ From Theorem \ref{thm:classification}, 
$(\mathbb R \mathbb I_{\mathbb R^3} \oplus L_{2,2})\ltimes \mathbb R^3$ is isomorphic to $ \mathfrak{aff}(\mathbb C)\oplus \mathfrak{aff}(\mathbb R).$

\noindent
(3) In $L_{2,3}:=\Big\{k_1(E_{1,1}+E_{2,2}-2E_{3,3})+k_2E_{1,2}, k_1,k_2\in\mathbb R \Big\}$, each element
  is of the form $k_1S_{2,3}^0+ k_2 S_{2,3} - (3k_1+k_2)S_{2,3}^2$ where $S_{2,3}:=E_{1,2}+E_{3,3}$ is nonderogatory 
with a double eigenvalue $0$ and a simple eigenvalue 1. Hence,  $\mathbb R \mathbb I_{\mathbb R^n} \oplus L_{2,3}=\mathbb R[ S_{2,3}]$ and
 $(\mathbb R \mathbb I_{\mathbb R^n} \oplus L_{2,3})\ltimes \mathbb R^3 = \mathcal G_{ S_{2,3} }$ is isomorphic to $\mathfrak{aff}(\mathbb R)\oplus \mathfrak{D}_0^2$ (Theorem \ref{thm:classification}).    

\noindent
(4) As regards  $L_{2,5}:=\Big\{k_1E_{1,2}+k_2E_{1,3}, k_1,k_2\in\mathbb R\Big\}$, 
 the algebra $\mathbb R\mathbb I_{\mathbb R^3}\oplus L_{2,5}$
 is the same as $\mathfrak{B}_{3,1}$ in Example \ref{example-general-nD}  and $\mathfrak{B}_{3,1}\ltimes \mathbb R^3=\mathcal G_{3,1}$ is indecomposable. 

\noindent
(5)  $L_{2,6}:=\Big\{k_1(E_{1,2}+E_{2,3})+k_2E_{1,3}, k_1,k_2\in\mathbb R\Big\}$ is such that    $\mathbb R\mathbb I_{\mathbb R ^3}\oplus L_{2,6}=\mathbb R[M_0]$, where 
 the nonderogatory matrix  $M_{0}=E_{1,2}+E_{2,3}$ has
   $0$  as a unique eigenvalue of multiplicity 3. Thus,   $(\mathbb R\mathbb I_{\mathbb R ^3}\oplus L_{2,6})\ltimes \mathbb R^3 =\mathfrak{D}_0^3,$  (Theorem \ref{thm:classification}).  

\noindent
(6) For the remaining algebra  
 $L_{2,4}:=\Big\{k_{1,3}E_{1,3}+k_{2,3}E_{2,3},k _{1,3},k_{2,3}\in\mathbb R\Big\}$
   in the list in  \cite{winternitz},  it is easy to see that  $(\mathbb R \mathbb I_{\mathbb R^3} \oplus L_{2,4}) \ltimes \mathbb R^3$ is not a Frobenius Lie algebra, as  every linear form $\alpha$ satisfies $(\partial\alpha)^3=0.$  See also Remark \ref{rem:masa-no-frobenius}.

\noindent
For  $n=4$, we  treat the pair-wise non-conjugate $16$ MASAs of $\mathfrak{sl}(4,\mathbb R)$ in the order in which they appear in the classification  list  in \cite{winternitz-zassenhaus} and name them $\mathcal{Y}_j$, $j=1,\dots,16.$ We correct the algebra $\mathcal{Y}_8$  from its original expression in   \cite{winternitz-zassenhaus} which was not commutative and complete the list of MASAs from which the MASA  $\mathcal{Y}_{17}$ giving rise to $\mathfrak{D}_0^2\oplus \mathfrak{D}_0^2$ is missing.
\begin{enumerate}
 \item  $\mathcal{Y}_1:=\{k_{1,3}E_{1,3}+k_{1,4}E_{1,4}+k_{2,3}E_{2,3}+k_{2,4}E_{2,4}$, $k_{1,3},k_{1,4},k_{2,3},k_{2,4}\in\mathbb R\}$,  is $4$-dimensional, thus it is not relevant to our study.
 \item $\mathcal{Y}_2:=\{k_{1,2}E_{1,2}+k_{1,3}E_{1,3}+k_{1,4}E_{1,4}$, $k_{1,2},k_{1,3},k_{1,4}\in\mathbb R\}$ is the algebra $L_{4,1}$ in Example  \ref{example-general-nD}, so the $2$-solvable Frobenius Lie algebra $(\mathbb I_{\mathbb R^4}\oplus\mathcal{Y}_2)\ltimes \mathbb R^4$ is isomorphic to $\mathcal G_{4,1}.$
 \item   $\mathcal{Y}_3:=\{k_{1,4}E_{1,4}+k_{2,4}E_{2,4}+k_{3,4}E_{3,4}$, $k_{1,4},k_{2,4},k_{3,4}\in\mathbb R\}$ is the algebra $L_n$ in Remark \ref{rem:masa-no-frobenius}, when $n=4$. So $(\mathbb R \mathbb I_{\mathbb R^4}\oplus\mathcal{Y}_3)\ltimes \mathbb R^4$ is not a Frobenius Lie algebra.
\item  $\mathcal{Y}_4:=\{k_{1,3}(E_{1,3}+E_{3,4})+k_{1,4}E_{1,4}+k_{2,4}E_{2,4}$, $k_{1,3},k_{1,4},k_{2,4}\in\mathbb R\}$  yields a 2-solvable Lie algebra which is not a Frobenius Lie algebra. Indeed, if we set $e_1:=\mathbb I_{\mathbb R^4},$ $e_2:=E_{1,3}+E_{3,4},$ $e_3:=E_{1,4},$ $e_4=E_{2,4}, $
 the Lie bracket of $(\mathbb I_{\mathbb R^4}\oplus \mathcal{Y}_4)\ltimes \mathbb R^4$ is $[e_1,\tilde e_j] =\tilde e_j,$ $j=1,\dots,4,$ $[e_2,\tilde e_3]=\tilde e_1,$
$[e_2,\tilde e_4]=\tilde e_3,$ $[e_3,\tilde e_4]=\tilde e_1,$ $[e_4,\tilde e_4]=\tilde e_2.$

\noindent
Any linear form
 $\alpha=k_1e_1^*+k_2e_2^*+k_3e_3^*+k_4e_4^*+s_1\tilde e_1^*+s_2\tilde e_2^*+s_3\tilde e_3^*+s_4\tilde e_4^*$ satisfies 
$(\partial\alpha)^4=0$, as we have $\partial\alpha= -e_1^*\wedge\eta_1-e_2^*\wedge \eta_2 -\eta_3\wedge\tilde  e_4^*$, where 
$\eta_1^*:=s_1\tilde e_1^*+s_2\tilde e_2^*+s_3\tilde e_3^*+s_4\tilde e_4^*,$ $\eta_2:=s_1\tilde e_3^*+s_3\tilde e_4^*,$ $\eta_3:=s_1e_3^*+s_2e_4^*.$
 \item $\mathcal{Y}_5:=\{k_{1,2}(E_{1,2}+E_{2,4})+k_{1,3}E_{1,3}+k_{1,4}E_{1,4},k_{1,2},k_{1,3},k_{1,4}\in\mathbb R\}$ is the MASA $\mathcal{P}_{4,2}$ of $\mathfrak{sl}(4,\mathbb R)$ given in Example \ref{example2a}. So 
$(\mathbb I_{\mathbb R^4}\oplus \mathcal{Y}_5)\ltimes \mathbb R^4$ is  the $2$-solvable Frobenius Lie algebra $\mathfrak{h}_{4,2}$ in Example \ref{example2a}.
 \item  $\mathcal{Y}_{6,\varepsilon}:=\{k_{1,2}(E_{1,2}+E_{2,4})+k_{1,3}(E_{1,3}+\varepsilon E_{3,4})+k_{1,4}E_{1,4}, \; k_{1,2},k_{1,3},k_{1,4}\in\mathbb R\}$, $\varepsilon=\pm 1$, yields a 2-solvable Frobenius Lie algebra $(\mathbb I_{\mathbb R^4}\oplus \mathcal{Y}_{6,\varepsilon})\ltimes \mathbb R^4$ 
 with the following Lie brackets, in the basis $e_1:=\mathbb I_{\mathbb R^4},$ $e_2:=E_{1,2}+E_{2,4},$ $e_3:=E_{1,3}+\varepsilon E_{3,4},$ $e_4=E_{1,4}, $ $\tilde e_j,$ $1\leq j\leq 4:$
 $[e_1, \tilde e_j] =\tilde e_j,$ $j=1,\dots,4,$ $[e_2,\tilde e_2]=\tilde e_1,$
$[e_2,\tilde e_4]=\tilde e_2,$
$[e_3,\tilde e_3]=\tilde e_1,$ $[e_3,\tilde e_4]=\varepsilon \tilde e_3,$ $[e_4,\tilde e_4]=\tilde e_1.$
Furthermore  
$\partial\tilde e_1^*= -e_1^*\wedge\tilde e_1^* -  e_2^*\wedge\tilde e_2^* - e_3^*\wedge\tilde e_3^* - e_4^*\wedge\tilde e_4^*$ is nondegenerate.
 Note that $\mathcal{Y}_{6,1}$ coincides with $\mathcal{P}_{4,3}$ in Example \ref{example2a}.  
So 
$(\mathbb I_{\mathbb R^4}\oplus \mathcal{Y}_{6,1})\ltimes \mathbb R^4$ is  the same as $\mathfrak{h}_{4,3}$ in Example \ref{example2a}.
When $\varepsilon=-1$, the linear map defined by $\psi: (\mathbb I_{\mathbb R^4}\oplus \mathcal{Y}_{6,-1})\ltimes \mathbb R^4 \to\G_{4,4}',$ 
$\psi(e_1)=e_1',$ 
$\psi(e_2)=e_2'-e_3',$  $\psi(e_3)=-e_2'-e_3',$ $\psi(e_4)=-2e_4',$
$\psi(\tilde e_2)=-\tilde e_2'+ \tilde e_3',$ $\psi(\tilde e_3)=\tilde e_2'+ \tilde e_3',$ $\psi(\tilde e_1)=-4\tilde e_1',$
is a Lie algebra isomorphism, where $\G_{4,4}'$ is as in Example \ref{example23}. So $\mathcal{Y}_{6,-1}$ is conjugate to $L_{4,4}'.$
 \item $\mathcal{Y}_7:=\{k_{1,2}(E_{1,2}+E_{2,3}+E_{3,4})+k_{1,3}(E_{1,3}+E_{2,4})+k_{1,4}E_{1,4}, \; k_{1,2},k_{1,3},k_{1,4}\in\mathbb R\}$ satisfies
  $\mathbb I_{\mathbb R^4}\oplus \mathcal{Y}_{7}=\mathbb R[M_{0}],$  with $M_{0}=E_{1,2}+E_{2,3}+E_{3,4}.$ So $(\mathbb I_{\mathbb R^4}\oplus \mathcal{Y}_{7})\ltimes \mathbb R^4=\mathfrak{D}_{0}^4$.
\item  $\mathcal{Y}_8:=\{k_{1,2}(E_{1,2}-E_{2,1}+E_{3,4}-E_{4,3})+k_{1,3}(E_{1,3}+E_{2,4})+k_{1,4}(E_{1,4}-E_{2,3}),$ $ k_{1,2},k_{1,3},k_{1,4}\in\mathbb R\}$ is such that $\mathbb I_{\mathbb R^4}\oplus \mathcal{Y}_{8}=\mathbb R[M_{0,1}]$, where $M_{0,1}=
E_{1,2}-E_{2,1}+E_{3,4}-E_{4,3}+E_{1,3}+E_{2,4}$. Thus $(\mathbb I_{\mathbb R^4}\oplus \mathcal{Y}_{8})\ltimes \mathbb R^4=\mathfrak{D}_{0,1}^4$. Note that, in the original  list in \cite{winternitz-zassenhaus}, this item was listed with some mistakes as $\{k_{1,2}(E_{1,2}-E_{2,1}+E_{2,4}-E_{4,3})+k_{1,3}(E_{1,3}+E_{2,4})+k_{1,4}E_{1,4}, $ $ k_{1,2},k_{1,3},k_{1,4}\in\mathbb R\}$ which is not commutative.
\item  $\mathcal{Y}_9:=\{k_{1,1}(E_{1,1}+E_{2,2}+E_{3,3}-3E_{4,4})+k_{1,2}E_{1,2}+k_{1,3}E_{1,3},$ $ k_{1,1},k_{1,2},k_{1,3}\in\mathbb R\}$, consider the basis 
$e_1=\mathbb I_{\mathbb R^n},$ $e_2=E_{1,1}+E_{2,2}+E_{3,3}-3E_{4,4},$ $e_3=E_{1,2},$ $e_4:=E_{1,3}$ of $\mathbb I_{\mathbb R^4}\oplus\mathcal{Y}_9$
and change it into
$e_1':=\frac{1}{4}(3e_1+e_2),$
$e_2'=e_3,$ 
$e_3'=e_4,$
$e_4':=\frac{1}{4}(e_1-e_2).$ The Lie bracket of $(\mathbb I_{\mathbb R^4}\oplus\mathcal{Y}_9)\ltimes \mathbb R^4$ reads
$[e_1',\tilde e_1]=\tilde e_1,$ $[e_1',\tilde e_2]=\tilde e_2,$ $[e_1',\tilde e_3]=\tilde e_3,$
$[e_2',\tilde e_2]=\tilde e_1,$
$[e_3,\tilde e_3]=\tilde e_1,$ 
$[e_4',\tilde e_4]=\tilde e_4.$ 
That is clearly the Lie bracket of $\mathcal G_{3,1}\oplus\mathfrak{aff}(\mathbb R)$, where
 $\mathcal G_{3,1}=span(e_1',e_2',e_3',\tilde e_1,\tilde e_2,\tilde e_3)$ corresponds to $n=3, p=1$ in Example  \ref{example-general-nD}.
\item  $\mathcal{Y}_{10}:=\{k_{1,1}(E_{1,1}+E_{2,2}+E_{3,3}-3E_{4,4})+k_{1,3}E_{1,3}+k_{2,3}E_{2,3}, \; k_{1,1},k_{1,3},k_{2,3}\in\mathbb R\}$, a basis of $\mathbb I_{\mathbb R^4}\oplus\mathcal{Y}_{10}$ is
$e_1=\mathbb I_{\mathbb R^4},$ $e_2=E_{1,1}+E_{2,2}+E_{3,3}-3E_{4,4},$ $e_3=E_{1,3},$ $e_4:=E_{2,3}$
which we change into 
$e_1':=\frac{1}{4}(3e_1+e_2),$
$e_2'=e_3,$ 
$e_3'=e_4,$
$e_4':=\frac{1}{4}(e_1-e_2).$
The Lie bracket of $(\mathbb I_{\mathbb R^4}\oplus\mathcal{Y}_{10})\ltimes \mathbb R^4$ reads
$[e_1',\tilde e_1]=\tilde e_1,$ $[e_1',\tilde e_2]=\tilde e_2,$ $[e_1',\tilde e_3]=\tilde e_3,$
$[e_2',\tilde e_3]=\tilde e_1,$
$[e_3',\tilde e_3]=\tilde e_2,$ 
$[e_4',\tilde e_4]=\tilde e_4,$ 
which is the same as that of $(\mathfrak{B}_{3}\ltimes\mathbb R^3)\oplus\mathfrak{aff}(\mathbb R)$, where $\mathfrak{B}_{3}=$span($e_1',e_2',e_3')$ is as in Remark  \ref{rem:masa-no-frobenius}, when $n=3$ and   $\mathbb{R}^{3}=$span($\tilde e_1,\tilde e_2,\tilde e_3)$.  As  $\mathfrak{B}_{3}\ltimes\mathbb R^3$ is not a Frobenius Lie algebra (Remark  \ref{rem:masa-no-frobenius}), neither is  
 $(\mathbb I_{\mathbb R^4}\oplus\mathcal{Y}_{10})\ltimes \mathbb R^4$.

\item $\mathcal{Y}_{11}:=\{k_{1,1}(E_{1,1}+E_{2,2}+E_{3,3}-3E_{4,4})+k_{1,2}(E_{1,2}+E_{2,3})+k_{1,3}E_{1,3},$ $ k_{1,1},$ $k_{1,2},k_{1,3}\in\mathbb R\}$, a basis of $\mathbb R\mathbb I_{\mathbb R^n}\oplus \mathcal{Y}_{11}$ is
$e_1=\mathbb I_{\mathbb R^n},$ $e_2=E_{1,1}+E_{2,2}+E_{3,3}-3E_{4,4},$ $e_3=E_{1,2}+E_{2,3},$ $e_4:=E_{1,3} = e_2^2$,
which we change into 
$e_1':=\frac{1}{4}(3e_1+e_2),$
$e_2'=e_3,$ 
 $e_3'=e_4,$
$e_4':=\frac{1}{4}(e_1-e_2),$
so that the Lie bracket of $(\mathbb R\mathbb I_{\mathbb R^n}\oplus \mathcal{Y}_{11})\ltimes \mathbb R^4$ reads
$[e_1',\tilde e_1]=\tilde e_1,$ $[e_1',\tilde e_2]=\tilde e_2,$ $[e_1',\tilde e_3]=\tilde e_3,$
$[e_2',\tilde e_2]=\tilde e_1,$
$[e_2',\tilde e_3]=\tilde e_2,$ 
$[e_3',\tilde e_3]=\tilde e_1,$
$[e_4',\tilde e_4]=\tilde e_4.$  
This is  the Lie bracket of $\mathfrak{D}_{0}^3\oplus \mathfrak{aff}(\mathbb R)$, where  $\mathfrak{D}_{0}^3=span(e_1',e_2',e_3',\tilde e_1,\tilde e_2,\tilde e_3)=\mathbb R[e_2']\ltimes \mathbb R^3$. 
\item $\mathcal{Y}_{12}:=\{k_{1,1}(E_{1,1}+E_{2,2}-E_{3,3}-E_{4,4})+k_{1,2}(E_{1,2}-E_{2,1})+k_{3,4}(E_{3,4}-E_{4,3}),  $ $k_{1,1}, k_{1,2},$ $k_{3,4}\in\mathbb R\}$, a basis of $\mathbb R\mathbb I_{\mathbb R^4}\oplus \mathcal{Y}_{12}$ is
$e_1=\mathbb I_{\mathbb R^n},$ $e_2=E_{1,1}+E_{2,2}-E_{3,3}-E_{4,4},$ $e_3=E_{1,2} - E_{2,1},$ $e_4:=E_{3,4}-E_{4,3},$
 we change it into 
$e_1'=\frac{1}{2}(e_1+e_2),$ $e_2'=e_3,$  $e_3'=\frac{1}{2}(e_1-e_2),$ $e_4'=e_4.$
The Lie bracket of $(\mathbb R\mathbb I_{\mathbb R^n}\oplus \mathcal{Y}_{12})\ltimes \mathbb R^4$ is given by
$[e_1',\tilde e_1]=\tilde e_1,$ $[e_1',\tilde e_2]=\tilde e_2,$
 $[e_2',\tilde e_1]=-\tilde e_2,$
$[e_2',\tilde e_2]=\tilde e_1,$
$[e_3',\tilde e_3]=\tilde e_3,$ 
$[e_3',\tilde e_4]=\tilde e_4,$
$[e_4',\tilde e_3]=-\tilde e_4,$
$[e_4',\tilde e_4]=\tilde e_3.$    
This is  the direct sum $\mathcal I_1\oplus \mathcal I_2$ of two ideals $\mathcal I_1=span(e_1',e_2',\tilde e_1,\tilde e_2)$ and $\mathcal I_2=span(e_3',e_4',\tilde e_3,\tilde e_4).$  Both $\mathcal I_1$ and $\mathcal I_2$ are isomorphic to   $ \mathfrak{aff}(\mathbb C).$ So  $(\mathbb R\mathbb I_{\mathbb R^4}\oplus \mathcal{Y}_{12})\ltimes \mathbb R^4$ is isomorphic to  $ \mathfrak{aff}(\mathbb C) \oplus \mathfrak{aff}(\mathbb C).$

\item $\mathcal{Y}_{13}:=\{k_{1,1}(E_{1,1}+E_{2,2}-E_{3,3}-E_{4,4})+k_{1,2}(E_{1,2}-E_{2,1})+k_{3,4}E_{3,4},$ $ k_{1,1},k_{1,2},$ $k_{3,4}\in\mathbb R\}$, consider the following basis of $\mathbb R\mathbb I_{\mathbb R^4}\oplus \mathcal{Y}_{13}$:
$e_1=\mathbb I_{\mathbb R^n},$ $e_2=E_{1,1}+E_{2,2}-E_{3,3}-E_{4,4},$ 
%
 $e_3=E_{1,2} - E_{2,1},$ $e_4:=E_{3,4}$. In the new basis 
$e_1'=\frac{1}{2}(e_1+e_2),$ $e_2'=e_3,$  $e_3'=\frac{1}{2}(e_1-e_2),$ $e_4'=e_4,$  we get the following  Lie bracket of $(\mathbb R\mathbb I_{\mathbb R^4}\oplus \mathcal{Y}_{13})\ltimes \mathbb R^4$:
$[e_1',\tilde e_1]=\tilde e_1,$ $[e_1',\tilde e_2]=\tilde e_2,$
 $[e_2',\tilde e_1]=-\tilde e_2,$
$[e_2',\tilde e_2]=\tilde e_1,$
$[e_3',\tilde e_3]=\tilde e_3,$ 
$[e_3',\tilde e_4]=\tilde e_4,$
$[e_4',\tilde e_4]=\tilde e_3.$    
This is the Lie bracket of $ \mathfrak{aff}(\mathbb C) \oplus \mathfrak{D}_{0}^2,$ with $\mathfrak{aff}(\mathbb C) =span(e_1',e_2',\tilde e_1,\tilde e_2)$ and 
$\mathfrak{D}_{0}^2 =span(e_3',e_4',\tilde e_3,\tilde e_4).$
\item $\mathcal{Y}_{14}:=\{ k_{1,1}(E_{1,1}+E_{2,2}-E_{3,3}-E_{4,4})+k_{1,2}(E_{1,2}-E_{2,1})+k_{3,3}(E_{3,3}-E_{3,4}), $ $k_{1,1},$ $k_{1,2},$ $k_{3,3}\in\mathbb R\}$, consider the following basis of $\mathbb R\mathbb I_{\mathbb R^4}\oplus \mathcal{Y}_{14}$:
$e_1=\mathbb I_{\mathbb R^n},$ $e_2=E_{1,1}+E_{2,2}-E_{3,3}-E_{4,4},$ $e_3=E_{1,2} - E_{2,1},$ $e_4:=E_{3,3}-E_{4,4}$.  In the basis
$e_1'=\frac{1}{2}(e_1+e_2),$ $e_2'=e_3,$ 
 $e_3'=\frac{1}{4}(e_1-e_2+2e_4),$ $e_4'=\frac{1}{4}(e_1-e_2-2e_4),$  $\tilde e_j$, $j=1,2,3,4,$ the  Lie bracket of $(\mathbb R\mathbb I_{\mathbb R^4}\oplus \mathcal{Y}_{14})\ltimes \mathbb R^4$ is:
$[e_1',\tilde e_1]=\tilde e_1,$ $[e_1',\tilde e_2]=\tilde e_2,$
 $[e_2',\tilde e_1]=-\tilde e_2,$
$[e_2',\tilde e_2]=\tilde e_1,$
$[e_3',\tilde e_3]=\tilde e_3,$ 
$[e_4',\tilde e_4]=\tilde e_4.$    
This is the Lie bracket of $ \mathfrak{aff}(\mathbb C) \oplus\mathfrak{aff}(\mathbb R) \oplus \mathfrak{aff}(\mathbb R) ,$ where  $ \mathfrak{aff}(\mathbb C)=$span($e_1',e_2',\tilde e_1,\tilde e_2$) and the two copies of  $ \mathfrak{aff}(\mathbb R)$ are span($e_3',\tilde e_3$), span($e_4',\tilde e_4$).
 \item $\mathcal{Y}_{15}:=\{ k_{1,1}(E_{1,1}+E_{2,2}-E_{3,3}-E_{4,4})+k_{1,2}E_{1,2}+k_{3,3}(E_{3,3}-E_{4,4}), \; k_{1,1},k_{1,2},k_{3,3}\in\mathbb R\}$, change the basis 
$e_1=\mathbb I_{\mathbb R^n},$ $e_2=E_{1,1}+E_{2,2}-E_{3,3}-E_{4,4},$ $e_3=E_{1,2},$ $e_4=E_{3,3}-E_{4,4}$ of $\mathbb R\mathbb I_{\mathbb R^4}\oplus \mathcal{Y}_{15}$, 
 into
$e_1'=\frac{1}{2}(e_1+e_2),$ $e_2'=e_3,$ 
 $e_3'=\frac{1}{4}(e_1-e_2+2e_4),$ $e_4'=\frac{1}{4}(e_1-e_2-2e_4).$  The  Lie bracket of $(\mathbb R\mathbb I_{\mathbb R^4}\oplus \mathcal{Y}_{15})\ltimes \mathbb R^4$ is 
$[e_1',\tilde e_1]=\tilde e_1,$ $[e_1',\tilde e_2]=\tilde e_2,$
$[e_2',\tilde e_2]=\tilde e_1,$
$[e_3',\tilde e_3]=\tilde e_3,$ 
$[e_4',\tilde e_4]=\tilde e_4.$    
This is the Lie bracket of  $ \mathfrak{D}_0^2 \oplus\mathfrak{aff}(\mathbb R) \oplus \mathfrak{aff}(\mathbb R) ,$ where $ \mathfrak{D}_0^2=$span($e_1',e_2,\tilde e_1,\tilde e_2$).
\item $\mathcal{Y}_{16}:=\{k_{1}(E_{1,1}+E_{2,2}+E_{3,3}-3E_{4,4})+k_{2}(E_{1,1}+E_{2,2}-2E_{3,3})+k_{3}(E_{1,1}$$-E_{2,2}) $, $ k_{1},$$k_{2},$ $k_{3}\in\mathbb R\}$, we consider the basis 
$e_1=\mathbb I_{\mathbb R^n},$ $e_2=E_{1,1}+E_{2,2}+E_{3,3}-3E_{4,4},$ $e_3=E_{1,1}+E_{2,2}-2E_{3,3},$ $e_4:=E_{1,1}-E_{2,2}$ of $\mathbb R\mathbb I_{\mathbb R^4}\oplus \mathcal{Y}_{16}$, then change it to 
$e_1'=\frac{1}{12}(3e_1+e_2+2e_3+6e_4),$ 
$e_2'=\frac{1}{12}(3e_1+e_2+2e_3-6e_4),$
 $e_3'=\frac{1}{4}(3e_1+e_2-4e_3),$ 
$e_4'=\frac{1}{4}(e_1-e_2).$  The Lie bracket of $(\mathbb R\mathbb I_{\mathbb R^4}\oplus \mathcal{Y}_{16})\ltimes \mathbb R^4$ is:
$[e_1',\tilde e_1]=\tilde e_1,$ 
$[e_2',\tilde e_2]=\tilde e_2,$
$[e_3',\tilde e_3]=\tilde e_3,$ 
$[e_4',\tilde e_4]=\tilde e_4.$    
This is the Lie bracket of $\mathfrak{aff}(\mathbb R) \oplus \mathfrak{aff}(\mathbb R)\oplus\mathfrak{aff}(\mathbb R) \oplus \mathfrak{aff}(\mathbb R)$, where the copies of  $\mathfrak{aff}(\mathbb R)$ are $span(e_j',\tilde e_j)$, $j=1,2,3,4,$ respectively.
\item $\mathcal{Y}_{17}:=\{k_{1,1}(E_{1,1}+E_{2,2}-E_{3,3}-E_{4,4})+k_{1,2}E_{1,2}+k_{3,4}E_{3,4}, \; k_{1,1},k_{1,2},k_{3,4}\in\mathbb R\}$. Consider the basis 
$e_1=\mathbb I_{\mathbb R^n},$ $e_2=E_{1,1}+E_{2,2}-E_{3,3}-E_{4,4},$ $e_3=E_{1,2},$ $e_4:=E_{3,4}$ of $\mathbb R\mathbb I_{\mathbb R^4}\oplus \mathcal{Y}_{17}$ and further  change it into 
$e_1'=\frac{1}{2}(e_1+e_2),$  $e_2'=e_3,$ $e_3'=\frac{1}{2}(e_1-e_2),$ $e_4'=e_4.$  The  Lie bracket of $(\mathbb R\mathbb I_{\mathbb R^4}\oplus \mathcal{Y}_{17})\ltimes \mathbb R^4$ reads:
$[e_1',\tilde e_1]=\tilde e_1,$ 
$[e_1',\tilde e_2]=\tilde e_2,$
$[e_2',\tilde e_2]=\tilde e_1,$ 
$[e_3',\tilde e_3]=\tilde e_3,$
$[e_3',\tilde e_4]=\tilde e_4,$ 
$[e_4',\tilde e_4]=\tilde e_3.$    
This is the Lie bracket of the direct sum $\mathfrak{D}_0^2 \oplus \mathfrak{D}_0^2$, where the copies of  $\mathfrak{D}_0^2$ are span$(e_1',e_2',\tilde e_1,\tilde e_2)$ and $span(e_3',e_4',\tilde e_3,\tilde e_4)$. The present item is missing in the list in \cite{winternitz-zassenhaus}.
\end{enumerate}


\section{Appendix: Jordan form of nonderogatory real Matrices}\label{sect:proof-of-Jordanization}

\begin{theorem}[Jordan form of real endomorphisms]\label{thm:Jordanization}
Let $V$ be a real vector space, $\dim V=n.$ Let $\phi\in\mathfrak{gl}(V)$ be nonderogatory. There exists a basis of $V$ in which  the matrix $[\phi]$ of $\phi$ has a Jordan form, more precisely, the following hold.

\noindent
(A) If $\phi$ has a unique eigenvalue $\lambda$ with multiplicity $n$, 
then $[\phi] =\lambda \mathbb I_{V} + \sum\limits_{i=1}^{n-1}E_{i,i+1}.$

\noindent
(B) If $\phi$ has $n$ distinct complex eigenvalues $\lambda_j$, $\bar \lambda_j$, $j=1,\dots,\frac{n}{2}$ then $[\phi]$$=$ diag($J_1,\dots,J_{\frac{n}{2}}$), where $J_j=\begin{pmatrix}Re(\lambda_j)&-Im(\lambda_j)\\ Im(\lambda_j)&Re(\lambda_j)\end{pmatrix},$ $Im(\lambda_j)\neq 0.$

\noindent
(C) If $\phi$ has only two  eigenvalues which are both complex $z=r+is$ and $\bar z,$ where $r,s\in\mathbb R,$ $s\neq 0$, and if $n\ge 4$ (in which case $\phi$ is not diagonalizable in $\mathbb C$), then  $[\phi]$ has the form, with $M_s$, $M_n$ as in (\ref{Ms-Mn-M01}):  
\begin{eqnarray}\label{Jordancomplex}
[\phi]=r\mathbb I_{V}+sM_s+M_n=:M_z\;.
\end{eqnarray}

\noindent
(D) In the  general case, $[\phi]$ is a block diagonal matrix with blocks in the form (A), (B), (C) or diagonal with distinct diagonal real entries.
\end{theorem}
\subsection{Proof of Theorem \ref{thm:Jordanization}}
Let $\phi\in\mathfrak{gl}(V)$ be nonderogatory, where $V$ is a real $n$-dimensional vector space. We write   $\chi_\phi(X) = Q_1(X)Q_2(X)$ where $Q_1(X)$ has only complex (nonreal) zeros and the zeros of  $Q_2(X)$ are all real, so that $V=\ker(\chi_\phi(\phi)) = \ker(Q_1(\phi))\oplus \ker(Q_2(\phi)).$ The result follows  by noting  again that $\phi$ preserves both $\ker(Q_1(\phi))$ and $\ker(Q_2(\phi))$ and its restriction to each of them has a Jordan form, as discussed below.  In fact, $Q_1(X)$ factorizes as $Q_1(X) = Q_{1a}(X)Q_{1b}(X)$ where all the zeros of $Q_{1a}(X)$  have multiplicity $1$  and we can further factorize $Q_{1b}(X)$ into factors, each having only $2$ (complex conjugate) zeros with multiplicity greater than $1.$ So the latter boils down to the case where $\phi$ has only two complex conjugate eigenvalues, as discussed in (C) below. 

\noindent
(A) Suppose $\phi$ has p distinct eigenvalues $\lambda_1, \dots,\lambda_p,$  all of which are real and of respective multiplicity $k_1,\dots, k_p$, with 
$k_1+\dots+ k_p=n$, so that its characteristic polynomial factorizes as
 $\chi_\phi(X)=\chi_1(X)\chi_2(X)\cdots\chi_p(X)$, 
where $\chi_j(X)=(X-\lambda_j)^{k_j}$, for $j=1,\dots,p$. The polynomials $\chi_j(X)$
 are pairwise relatively prime. 
By the 
Primary Decomposition Theorem and Cayley-Hamilton theorem,  we have 
\begin{eqnarray}
\label{eq:split}V=\ker\chi_\phi(\phi) = \ker\chi_1(\phi)\oplus\ker\chi_2(\phi)\oplus\cdots\oplus\ker\chi_p(\phi).\end{eqnarray} 
Of course the subspaces $\mathcal E_j:=\ker(\phi-\lambda_j)^{k_j},$ $j=1,\dots,p,$ are all stable by $\phi$ and the restriction of $\phi$ to $\mathcal E_j$ is again a nonderogatory endomorphism of $\mathcal E_j$ with a unique eigenvalue,
for every $j=1,\dots,p.$ So this reduces to the case where $\phi$ has $n$ distinct real eigenvalues and is hence diagonalizable, or $\phi$ has a unique real eigenvalue $\lambda$ with multiplicity $n.$  We now assume the latter case.  Up to a sign, $ \chi_\phi(X)=(X-\lambda)^n=\sum\limits_{j=0}^n \complement_n^j(-1)^{n-j} \lambda^{n-j} X^{j},$  with $\complement_p^q=\frac{p!}{q!(p-q)!}$, for $p\ge q$. 
By Cayley-Hamilton's Theorem 
$\phi^n=\sum\limits_{j=0}^{n-1} \complement_n^j(-1)^{n-j+1} \lambda^{n-j} \phi^{j} $.
Choose $\bar x\in V$ such that   $(\bar e_1,\dots,\bar e_{n})=$$(  \phi^{n-1}\bar x, \; \phi^{n-2}\bar x, \; \dots,  \; \phi^{n-j}\bar x, \dots, \; \bar x)$ is a basis of $V.$  We thus have  
$\phi\bar e_1=\phi^n\bar x=\sum\limits_{j=0}^{n-1} \complement_n^j(-1)^{n-j+1} \lambda^{n-j} \phi^{j}\bar x =
\sum\limits_{j=1}^{n} \complement_n^j(-1)^{j+1} \lambda^{j} \bar e_j$ and $\phi\bar e_j=\bar e_{j-1}$, for $j=2,\dots,n$. So, the matrix of $\phi$ has the form
$\tilde M_\lambda=\sum\limits_{j=1}^{n} \tilde k_j E_{j,1}+\sum\limits_{j=1}^{n-1}E_{j,j+1},$
with $\tilde  k_j=\complement_n^j(-1)^{j+1} \lambda^{j}.$
 If $\lambda=0$, then we are done. If $\lambda\neq 0,$ we use a direct approach by looking for the coefficients $p_{l,j}$ of a matrix $P$ such that  $\tilde  M_\lambda P=PM_\lambda,$ where $M_\lambda =\lambda \mathbb I_{V} + \sum\limits_{l=1}^{n-1}E_{l,l+1}.$
We get 
$p_{k,l}=\sum\limits_{j=0}^{l-1} (-1)^{k-l+j} \lambda^{k-l+j} \complement_{n-l+j}^{k-l+j} p_{1,j+1}$,  
where  the numbers $ p_{1,j+1}$, $j=0,\dots,n-1,$ are seen as parameters.
In particular, the matrix $P$ whose coefficients in the above basis ($\bar e_s$) are 
$p_{kl}=\left\{\begin{array}{cll}
(-1)^{k-l}\lambda^{k-l}\complement_{n-l}^{k-l} &\text{ if }& k\ge l, \\
0 &\text{ if }& k<l
\end{array}\right.
$
is a solution.

\noindent (B) Suppose all the eigenvalues  of $\phi$  are complex (and nonreal) and $\phi$ is diagonalizable in $\mathbb C$. Equivalently, $\phi$  has $n$ distinct complex eigenvalues, say, $\lambda_j,$ $\bar \lambda_j,$ $j=1,\dots,\frac{n}{2},$ where of course, here, $\bar \lambda_j$ is the complex conjugate of $ \lambda_j.$  Let $\lambda =\lambda_R-i\lambda_I$  be an eigenvalue of $\phi$,  where $\lambda_R,\lambda_I$ are real numbers and $\lambda_I\neq 0$.
 Consider an  eigenvector  $v$ (with complex components) of $\phi$ with corresponding eigenvalue $\lambda$. Further set $\mathcal E_\lambda:=
\{ z v+\overline{zv}, z\in \mathbb C\}.$  Note that, as $\bar v$ is also an eigenvector of $\phi$ with eigenvalue $\bar \lambda$, we also have $\mathcal E_\lambda=\mathcal E_{\bar\lambda}.$ Furthermore,  $\mathcal E_\lambda$ is a real 2-dimensional vector subspace of $V$ 
which is stable by $\phi$ and 
 the vectors $Re(v)$ and $Im(v)$ form a basis of  $\mathcal E_\lambda$ in which the matrix  of the restriction of $\phi$ is
 $M_\lambda:=\begin{pmatrix}\lambda_R&-\lambda_I\\
\lambda_I &\lambda_R\end{pmatrix}.$ 
Indeed, any element of  $\mathcal E_\lambda$ is of the form 
$  z v+\overline{zv} = 2Re(z)\; Re(v) -2 Im(z)\; Im(v) $
and taking the real and imaginary parts of both sides of the equation 
\begin{eqnarray}\phi	v&=&\lambda v = \Big(\lambda_R-i\;\lambda_I\Big)\; \Big(Re(v)+i\;Im(v)\Big) \nonumber\\
&=& \Big(\lambda_R\; Re(v) + \lambda_I\; Im(v)\Big) + i \;\Big(\lambda_R\; Im(v) -\lambda_I \;Re(v)\Big), \end{eqnarray}
 one gets 
\begin{eqnarray} 
&& \phi Re(v)=\lambda_R\; Re(v) + \lambda_I\; Im(v)\;, \;\;  
 \phi Im(v)=- \lambda_I\; Re(v) + \lambda_R\; Im(v). 
\end{eqnarray} 
Note that  $M_\lambda$ is  nonderogatory  with $\chi_{M_\lambda}(X)=(X-\lambda_R)^2+\lambda_I^2$ and both $Re(v)$ and $Im(v)$ are in $\ker\chi_{M_\lambda}(\phi),$  more precisely, $\ker\chi_{M_\lambda}(\phi)=\mathcal E_\lambda .$ 

 In the basis $(Re(v_j), Im(v_j)),$ $j=1,\dots,\frac{n}{2},$ of $V,$ the map $\phi$ assumes  
 the following Jordan form diag($M_{\lambda_1},\dots,M_{\lambda_{\frac{n}{2}}}$), where $v_j$ is an eigenvector of $\phi$, as above, with eigenvalue $\lambda_j$. 

\noindent
(C) Suppose $\phi$ has only 2 eigenvalues which are complex $z = r+is$ and $\bar z,$ and $n\ge 4$. So $\phi$ is not diagonalizable and $\chi_\phi (X) =\Big((X-r)^2+s^2\Big)^{\frac{n}{2}}=\sum\limits_{j=0}^{\frac{n}{2}}\complement_{\frac{n}{2}} ^j s^{n-2j}(X-r)^{2j}.$  We write the latter as  
$\chi_\phi(X)=\sum\limits_{k=0}^{n} D_{n,k}(z)X^k\;,$ 
where, by a direct calculation, we obtain the coefficients  $D_{n,j}:\mathbb C\to\mathbb R$ which are given by
\begin{eqnarray}
 D_{n,j}(z)= 
(-1)^{j}\sum\limits_{p=\varepsilon_j}^{\frac{n}{2}}\complement_{\frac{n}{2}} ^p  \complement_{2p}^{j}\Big(\frac{z-\bar z}{2i}\Big)^{n-2p}\Big(\frac{z+\bar z}{2}\Big)^{2p-j} \;,
\end{eqnarray}
 where $2\varepsilon_j = j+ \frac{1}{2}(1-(-1)^j).$ By Cayley-Hamilton's Theorem, 
$\phi^n= -\sum\limits_{j=0}^{n-1}D_{n,j}(z)\phi^{j}.$ 
Following Lemma \ref{nonderogatorymatrices}, let $\bar x\in V$ such that $(\bar e_j:= \phi^{n-j}\bar x\; , \;j=1,\dots,n)$ is a basis of $V$. We have
\begin{eqnarray}\phi\bar e_j=\bar e_{j-1}, \; 2\leq j\leq n,\;\; \phi \bar e_1=\phi^n\bar  x=  -\sum\limits_{j=0}^{n-1}D_{n,j}(z)\phi^{j}\bar x=-\sum\limits_{j=0}^{n-1}D_{n,j}(z)\bar e_{n-j},\nonumber
\end{eqnarray}
 which in matrix form,
say $\tilde M_z$, reads 
$\tilde M_z  
=
\sum\limits_{j=1}^{n-1}E_{j,j+1}-\sum\limits_{j=1}^nD_{n,n-j}(z) E_{j,1}.$  
Now one may use a direct approach by looking for the explicit expressions of the  coefficients $p_{k,l}$ of a matrix $P$ such that  $\tilde M_z P=PM_z,$
where $M_{z}=r\mathbb I_{\mathbb R^n}+sM_s+M_n$, with $M_s$, $M_n$ as in (\ref{Ms-Mn-M01}).
Before proceeding, it is interesting to note the following. If $P_1$ and $P_2$ 
are invertible and  satisfy the equation  $\tilde M_zP=PM_z,$ then $P_1^{-1}P_2$ is in the isotropy subgroup of $M_z,$ that is, the subgroup of   GL$(V)$ consisting  of those $P_3$  satisfying $P_3M_zP_3^{-1}=M_z.$ 
Since $M_z$ is nonderogatory, its isotropy subgroup is a maximal subgroup of GL$(V)$. In fact, it is the group of invertible elements of $\mathbb R[M_z]$, it is connected, as the intersection $\mathbb R[M_z]\cap \det^{-1}(]0,+\infty[)$ of two connected subsets and hence it coincides with $\exp\Big(\mathbb R[M_z]\Big)$. 
So, there exists $Q\in\mathbb R[M_z]$ such that $P_1^{-1}P_2=\exp(Q).$ Thus, for a fixed 
invertible solution $P_1$, the  map $P_2\mapsto P_1^{-1}P_2$ is a 1-1 correspondence  between the set of invertible solutions of the equation  $\tilde  M_zP=PM_z$ and $\exp\Big(\mathbb R[M_z]\Big).$ We summarize this as follows.
\begin{lemma}Let $P_1$ be an invertible solution of the equation $\tilde  M_zP=PM_z.$ Then any other invertible solution $P$ is of the form  $P=P_1e^{Q},$ where $Q\in\mathbb R[M_z].$ The linear map $P_2\mapsto P_1^{-1}P_2$ is a 1-1 correspondence between the space of solutions 
and $\mathbb R[M_z],$ 
and
 invertible solutions are mapped to elements of the form $e^{Q}$,   $Q\in\mathbb R[M_z].$
\end{lemma}
Now from  the explicit expressions of the equation $\tilde  M_z P=PM_z$,  we 
extract the following linear recurrence relations, where of course, $s=\frac{z-\bar z}{2i}$ and $r=\frac{z+\bar z}{2}$ and $D_{k,l}=D_{k,l}(z)$:
\begin{eqnarray}\label{recurrencerelations}
p_{k+1,1}&=&p_{1,1}D_{n,n-k}+   rp_{k,1}+sp_{k,2} \; , \;  
p_{k+1,2} =p_{1,2}D_{n,n-k}+rp_{k,2}-sp_{k,1},\nonumber
\\
p_{k+1,2j+3} &=&p_{1,2j+3}D_{n,n-k}+rp_{k,2j+3}+sp_{k,2j+4}+p_{k,2j+1},\\
p_{k+1,2j+4}&=& p_{1,2j+4}D_{n,n-k}+rp_{k,2j+4}-sp_{k,2j+3}+ p_{k,2j+2},
\nonumber\\
 && j=0,\dots,\frac{n}{2}-2,\; k=1,\dots,n-1.\nonumber
\end{eqnarray}
Obviously, these linear recurrence relations admit explicit solutions $p_{k,l}$ which are linear combinations of $p_{1,j}$, $j=1,\dots,n$, the coefficients being polynomials in $z$, $\bar z.$ One easily checks that 
\beqn \det(P)= (\frac{z-\bar z}{2i})^{\frac{n^2}{4}}(p_{1, 1}^2+p_{1, 2}^2)^{\frac{n}{2}}q_n,
\eeqn
 where
 $q_4=4$ and for $n>4 $,
$q_n=(-1)^{\frac{n}{2}}(\frac{n}{2}-1)^n$. 
 This proves that invertible solutions, of the equation  $\tilde M_z P=PM_z,$ always exist, they correspond to $p_{1, 1}^2+p_{1, 2}^2\neq 0$. More precisely, from (\ref{recurrencerelations}), the  two first columns of $P$ can be rewritten, for any $k=1,\dots,n-1$,  as
\beqn \begin{pmatrix}p_{k+1,1}\\
p_{k+1,2}\end{pmatrix}&=&
\begin{pmatrix}r & s\\
-s &r\end{pmatrix} \begin{pmatrix}p_{k,1}\\
p_{k,2}\end{pmatrix}+ D_{n,n-k} \begin{pmatrix}p_{1,1}\\
p_{1,2}\end{pmatrix}\nonumber
\eeqn which gives rise to the explicit expression
\beqn \begin{pmatrix}p_{k+1,1}\\
p_{k+1,2}\end{pmatrix}&=& \Big[
\sum\limits_{l=0}^{k} D_{n,n-(k-l)} \begin{pmatrix}r & s\\
-s &r\end{pmatrix}^l \Big] \begin{pmatrix}p_{1,1}\\
p_{1,2}\end{pmatrix}\;.
\eeqn
 Again from (\ref{recurrencerelations}), the columns $2j+3$ and $2j+4$ of $P$ read, for any $j=0,\dots,\frac{n}{2}-2$ and $k=1,\dots,n-1$,  as
\beqn\label{recurrencerelations2}
 \begin{pmatrix}p_{k+1,2j+3}\\
p_{k+1,2j+4}\end{pmatrix}&=&
D_{n,n-k} \begin{pmatrix}p_{1,2j+3}\\
 p_{1,2j+4}\end{pmatrix}
+ 
\begin{pmatrix}r&s\\
 -s& r\end{pmatrix}
\begin{pmatrix}p_{k,2j+3}\\
 p_{k,2j+4}\end{pmatrix}
+
\begin{pmatrix}p_{k,2j+1}\\
p_{k,2j+2}\end{pmatrix}
\nonumber\\
&=&
 \Big(\sum\limits_{l_{1}=0}^kD_{n,n-(k-l_{1})}\begin{pmatrix}r&s\\
 -s& r\end{pmatrix}^{l_{1}}\Big)\begin{pmatrix}p_{1,2j+3}\\
 p_{1,2j+4}\end{pmatrix} 
\nonumber\\ 
&&+
\Big(\sum\limits_{l_{1}=0}^{k-1}\begin{pmatrix}r&s\\
 -s& r\end{pmatrix}^{l_{1}}\Big)\begin{pmatrix}p_{k-l_{1},2j+1}\\
 p_{k-l_{1},2j+2}\end{pmatrix} \;.
\eeqn
Further developing (\ref{recurrencerelations2}), leads to the final explicit expressions
\beqn \begin{pmatrix}p_{k+1,2j+3}\\
p_{k+1,2j+4}\end{pmatrix}&=& \sum\limits_{q=0}^{j+1}U(k,j,q) \begin{pmatrix}p_{1,2(j-q)+3}\\
 p_{1,2(j-q)+4}\end{pmatrix} \;, 
\eeqn
some of the $2\times 2$ matrices $U(k,j,q) $  have quite lengthy general expressions, some of which being zero (see cases $n=4,6,10 $ 
 below).
For example, one gets
\beqn
p_{2,1}&=& -r(n-1)p_{1, 1}+sp_{1, 2} \; , \; 
p_{2,2}= -r(n-1)p_{1, 2}-p_{1, 1}s, \nonumber\\
p_{3,1}&=&\frac{n-2}{2}\Big((s ^2+(n-1)r^2)\; p_{1, 1}-2rs\; p_{1, 2}\Big),
\nonumber\\
 p_{3,2} &=& \frac{n-2}{2}\Big((s ^2+(n-1)r^2)\; p_{1, 2}+2rs\; p_{1, 1}\Big) \; , \text{ and for } j\ge 1, \nonumber\\
p_{2, 2j+1}&=& -(n-1)rp_{1, 2j+1}+sp_{1, 2j+2}+p_{1, 2j-1}\;,\; \nonumber\\
p_{2, 2j+2}&=& -(n-1)rp_{1, 2j+2}-sp_{1, 2j+1}+p_{1, 2j} \;,\nonumber\\
&\vdots&\nonumber\\
p_{n,1} &=& (sp_{1,2}-rp_{1,1}) \vert z \vert^{n-2}  \;, \; 
p_{n,2}  = -(s p_{1,1}+rp_{1,2}) \vert z \vert^{n-2}\;,   \nonumber
\\
p_{n,3}&=& \Big((-r^2+s^2)p_{1,1}+2rsp_{1, 2}\Big)\vert z\vert^{n-4}+\Big(sp_{1, 4}-rp_{1, 3}\Big)\vert z \vert ^{n-2}\;,
\nonumber\\
p_{n,4}&=&\Big(2rsp_{1, 1}+(r^2-s^2)p_{1, 2}\Big)\vert z\vert ^{n-4}+\Big(rp_{1, 4}+sp_{1, 3}\Big)\vert z\vert ^{n-2}, 
\nonumber\\
p_{n, 2j+3} &=& (sp_{1, 2j+4}-rp_{1, 2j+3})\vert z \vert^{n-2} 
+ (rp_{n, 2j+1}-sp_{n, 2j+2})\vert z \vert^{-2},\nonumber\\
p_{n,2j+4}&= &-(r p_{1,2 j+4}+s p_{1,2 j+3}) \vert z \vert^{n-2}-(r p_{n,2 j+2}+s p_{n,2 j+1}) \vert z \vert^{-2}\;.
\eeqn

\subsection{Some particular examples in low dimensions}

In low dimensions, one gets simple expressions for $P$ by setting $p_{1,1}=1$ and $p_{1,j}=0$, if $j\ge 2$ or $p_{1,2}=1$ and $p_{1,j}=0$, if $j \neq  2$. Here are some examples.

\noindent
$\bullet$ For $n=4,$ we get 
$p_{1,1}= 1,$ 
$p_{1,2}= p_{1,3}= p_{1,4}= 0,$
 $p_{2,1}=   -3r,$
  $p_{2,2}= -s,$  
 $p_{2,3}= 1$, 
$p_{2,4}=0$, $p_{3,1}=3r^2+s^2$,
$p_{3,2}= 2rs$, 
$p_{3,3}= -2r,$
$p_{3,4}= -2s,$
$p_{4,1}= -r(r^2+s^2)$, 
$p_{4,2}= -r^2s-s^3,$
 $p_{4,3}= r^2-s^2$,
 $p_{4,4}= 2rs$ and  $\det(P)=4s^4.$


\noindent
$\bullet$ For $n=6$, one gets 
$
p_{1,1}= 1
,
$ $p_{1,2}= p_{1,3}= p_{1,4}=p_{1,5}= p_{1,6}=0,$

\noindent
 $ 
p_{2,1}= -5r
, 
$ $
p_{2,2} := -s
,
$
%
 $
p_{2,3}= 1
,
$ $p_{2,4}= p_{2,5}= p_{2,6}= 0,$
%
 $
p_{3,1}= 10r^2+2s^2
,
$

\noindent
 $
p_{3,2}= 4rs
,$
%
$
p_{3,3}= -4r
,
p_{3,4}= -2s
,
p_{3,5}= 1
,
$
%
$p_{4,1}= -2r(5r^2+3s^2)
,$ 

\noindent
 $
p_{4,2}= -2s(3r^2+s^2)
,
p_{4,3} = 6r^2
,$ 
  %
 $
p_{4,4} = 6rs
,
p_{4,5} = -3r
,
p_{4,6} = -3s
,$ 

\noindent
$
p_{5,1} := (5r^2+s^2)(r^2+s^2)
,$ $
p_{5,2} := 4rs(r^2+s^2)
,$ $
p_{5,3} = -4r^3
,$ 
%
 $
p_{5,4} := -2s(3r^2+s^2)
,$ 

\noindent 
$
p_{5,5} = 3r^2-3s^2
,$ $
p_{5,6} = 6rs
,$
%
$
p_{6,1} = -r(r^2+s^2)^2
,$ $
p_{6,2} = -s(r^2+s^2)^2
,$ $
p_{6,3} = r^4-s^4
,$ 
%
$
p_{6,4} := 2rs(r^2+s^2),$ $
p_{6,5} := -r(r^2-3s^2),$ $
p_{6,6}= -s(3r^2-s^2)$, $\det(P)=-64s^9.$

\noindent
$\bullet$ For $n=10,$  we choose $p_{1,1}=1$ and $p_{1,j}=0$, $\forall j\ge 2,$ so the nonzero coefficients of $P$ are:  
%
$p_{1,1}=1$, 
$p_{2,1}=-9 r,$ $ p_{2,2}=-s, $ $p_{2,3}=1,$ $ p_{3,1}=4(9r^2+s^2),$ $ p_{3,2}=8 s r,$ 

\noindent
$ p_{3,3}=-8 r, $ $p_{3,4}=-2 s,$ $ p_{3,5}=1, $ $p_{4,1}=-28 r (3 r^2+s^2), $
%
$p_{4,2}=-4 s (7 r^2+s^2),$ 

\noindent $p_{4,3}=28 r^2+2 s^(2),$ $p_{4,4}=14 s r, $  $p_{4,5}=-7 r, $
%
$p_{4,6}=-3 s,$ $ p_{4,7}=1,$

\noindent
 $p_{5,1}=6(21 r^4+14 r^2 s^2+s^4),$ $ p_{5,2}=8 sr(7 r^2+3 s^2), $ 
%
$p_{5,3}=-4 r(14 r^2+3 s^2), $

\noindent
 $p_{5,4}=-6 s (7 r^2+ s^2), $ $p_{5,5}=21 r^2-s^2, $ 
%
$p_{5,6}=18 s r, $ $p_{5,7}=-6 r,$ $ p_{5,8}=-4 s, $ $p_{5,9}=1,$

\noindent
$ p_{6,1}=-2 r(63 r^4+70 r^2 s^2+15 s^4),$ $ p_{6,2}=-2 s(35 r^4+30 r^2 s^2+3 s^4), $

\noindent
$ p_{6,3}=10 r^(2) (7 r^2+3 s^2), $ $p_{6,4}=10 s r(7 r^2+3 s^2),$ $  p_{6,5}=5 r (-7 r^2+s^2),$

\noindent
$ p_{6,6}=-5 s(9 r^2+s^2), $  $p_{6,7}=5(3 r^2-s^2), $  $p_{6,8}=20 s r, $  $p_{6,9}=-5 r,$ 
$ p_{6,10}=-5 s,$

\noindent
$ p_{7,1}=4(r^2+s^2) (21 r^4+14 r^2 s^2+s^4), $  $p_{7,2}=8s r (7 r^2+3 s^2) (r^2+s^2),$

\noindent
 $ p_{7,3}=-8 r^3 (7 r^2+5 s^2), $  $p_{7,4}=-2 s (35 r^4+30 r^2 s^2+3 s^4), $  
%
$p_{7,5}=5(7 r^4-2 r^2 s^2-s^4),$

\noindent
 $ p_{7,6}=20 s r (3 r^2+s^2), $ $p_{7,7}=-20 r (r^2-s^2),$ 
%
$ p_{7,8}=-40 s r^2,$ $p_{7,9}=10 (r^(2)-s^(2)),$ $ p_{7,10}=20 s r,$ 
%
$ p_{8,1}=-12 r (3 r^2+s^2) (r^2+s^2)^2, $ $p_{8,2}=-4 (7 r^2+s^2) (r^2+s^2)^2 s,$

\noindent
$ p_{8,3}=2 (r^2+s^2) (14 r^4+r^2 s^2-s^4),$ $ p_{8,4}=6 s r (7 r^2+3 s^2) (r^2+s^2), $

\noindent
$p_{8,5}=-r (21 r^4-10 r^2 s^2-15 s^4), $ $p_{8,8}=40 r^3 s,$ $ p_{8,9}=-10 r (r^2-3 s^2),$

\noindent
$ p_{8,10}=-10 s (3 r^2-s^2),$ $ p_{9,1}=(9 r^2+s^2) (r^2+s^2)^3, $  $p_{9,2}=8 s r (r^2+s^2)^3, $

\noindent
$ p_{9,3}=-4 r (2 r^2-s^2) (r^2+s^2)^2, $  $p_{9,4}=-2 (7 r^2+s^2) (r^2+s^2)^2 s, $

\noindent
$p_{9,5}=(r^2+s^2) (7 r^4-12 r^2 s^2-3 s^4),$  $ p_{9,6}=2 s r (9 r^2+s^2) (r^2+s^2),$

\noindent
$ p_{9,7}=-2 r (3 r^4-10 r^2 s^2-5 s^4),$   $ p_{9,8}=-4 s (5 r^4-s^4),$

\noindent
$ p_{9,9}=5 (r^2-2 s r-s^2) (r^2+2 s r-s^2),$  $ p_{9,10}=20 s r (r-s) (r+s), $

\noindent
$p_{10,1}=-r (r^2+s^2)^4,$  $ p_{10,2}=-s (r^2+s^2)^4,$  $ p_{10,3}=(r^2-s^2) (r^2+s^2)^3, $

\noindent
$p_{10,4}=2 s r (r^2+s^2)^3, $ $p_{10,5}=-r (r^2-3 s^2) (r^2+s^2)^2, $

\noindent
$p_{10,6}=-(3 r^2-s^2) (r^2+s^2)^2 s,$  $ p_{10,7}=(r^2+s^2) (r^2-2 s r-s^2) (r^2+2 s r-s^2),$

\noindent
$ p_{10,8}=4 s r (r-s) (r+s) (r^2+s^2), $  $p_{10,9}=-r (r^4-10 r^2 s^2+5 s^4), $

\noindent
$p_{10,10}=-s (5 r^4-10 r^2 s^2+s^4)$ and $\det(P)=-1048576s^{25}.$


\begin{thebibliography}{99}
\bibitem{alvarez-abelian-nilradical} Alvarez, M. A.; Rodr\'iguez-Vallarte, M. C. and Salgado, G.:
\newblock{\em Contact and Frobenius solvable Lie algebras with abelian nilradical.}
\newblock Commun. Algebra {\bf 46}, no. 10, 4344-4354 (2018). 

\bibitem{4d-para-hypercomplex} Blazi\'c, N. and Vukmirovi\'c, S.:
\newblock {\em Four-dimensional Lie algebras with a para-hypercomplex structure.} 
\newblock Rocky Mountain J. Math. {\bf 40}, no. 5, 1391-1439 (2010).

\bibitem{bordemann-medina} Bordemann, M.; Medina, A. and Ouadfel, A .: 
\newblock {\em Le groupe affine comme vari\'et\'e symplectique.}
\newblock  Tohoku Math. J. (2) {\bf 45}, no. 3, 423-436 (1993).


\bibitem{degraaf} De Graaf, W. A.: 
\newblock {\em Classification of nilpotent associative algebras of small dimension.} 
\newblock Internat. J. Algebra Comput. {\bf 28}, no. 1, 133-161 (2018).

\bibitem{de-Olmo-Rodriguez-Wintemitz-Zassenhaus} de Olmo, M. A.; Rodriguez, M. A. ; Wintemitz, P. and Zassenhaus, H.:  
\newblock {\em Maximal Abelian Subalgebras of Pseudounitary Lie Algebras.} 
\newblock Linear Algebra Appl. {\bf 135}, 79-151 (1990).


\bibitem{diatta-manga-mbaye-gerstenhaber} Diatta, A.; Manga, B. and Mbaye, A.: 
\newblock {\em On systems of commuting matrices, Frobenius Lie algebras and Gerstenhaber's Theorem.} 
\newblock Preprint.

\bibitem{diatta-manga} Diatta, A. and Manga, B.: 
\newblock {\em On properties of principal elements of Frobenius Lie algebras.}
\newblock  J. Lie Theory {\bf 24}, no. 3, 849-864 (2014).


\bibitem{Diatta-Contact}   Diatta, A.: 
\newblock {\em Left invariant contact structures on Lie groups.} 
\newblock Differential Geom. Appl. {\bf 26}, no. 5, 544-552  (2008).


\bibitem{diatta-medina}  Diatta, A. and Medina, A.: 
\newblock {\em Classical Yang-Baxter equation and left invariant affine geometry on Lie groups.} 
\newblock Manuscripta Math. {\bf 114}, no. 4, 477-486 (2004). 

\bibitem{dixmier} Dixmier, J.: 
\newblock {\em Sous-anneaux ab\'eliens maximaux dans les facteurs de type fini.}
\newblock  Ann. of Math. (2) {\bf 59}, 279-286  (1954).

\bibitem{g2-structures-fino} Fern\'andez, M.; Fino, A. and Manero, V.: 
\newblock {\em G2-structures on Einstein solvmanifolds.}
\newblock Asian J. Math. {\bf 19}, no. 2, 321-342 (2015).

\bibitem{harish-chandra} Harish-Chandra: 
\newblock {\em The characters of semisimple Lie groups.} 
\newblock Trans. Amer. Math. Soc., {\bf 83}, 98-163 (1956).

\bibitem{jacobson-schur} Jacobson, N.: 
\newblock {\em Schur's theorems on commutative matrices.}
\newblock Bull. Amer. Math. Soc. {\bf 50}, no. 6, 431-436 (1944).


\bibitem{kostant} Kostant B.: 
\newblock {\em On the conjugacy of real Cartan subalgebras. I.} 
\newblock Proc. Nat. Acad. Sci. U.S.A. {\bf 41}, 967-970 (1955). 

\bibitem{lichnerowicz-medina} Lichnerowicz, A. and Medina, A.; 
\newblock {\em On Lie groups with left-invariant symplectic or K\"ahlerian structures.}
\newblock  Lett. Math. Phys. {\bf 16}, no. 3, 225-235 (1988).


\bibitem{mbaye-these} Mbaye, A.: 
\newblock {\em 2-solvable Frobenius Lie algebras and  applications.} 
\newblock  Ph.D. Thesis. In preparation.

\bibitem{ndogmo-winternitz} Ndogmo, J. C. and Winternitz, P.: 
\newblock {\em Solvable Lie algebras with abelian nilradicals.}
 \newblock J. Phys. A {\bf 27}, no. 2, 405-423 (1994).


\bibitem{ooms} Ooms, A.I.: 
\newblock {\em On Frobenius Lie algebras.} 
\newblock Commun. Algebra {\bf 8}, no. 1, 13-52 (1980). 
\bibitem{ooms-primitve-ideals} Ooms, A.I.: 
\newblock {\em On Lie Algebras Having a Primitive
Universal Enveloping Algebra.}
\newblock  J. Algebra {\bf 32}, 488-500 (1974)

\bibitem{schur} Schur I.: 
\newblock {\em Zur Theorie vertauschbaren Matrizen.}
\newblock J. Reine Angew. Math. {\bf 130}, 66-76 (1905).

\bibitem{serre} Serre, J.P.: 
\newblock {\em Lie algebras and Lie groups.} 
\newblock Lectures notes in Math., {\bf 1500}, Springer-Verlag, Berlin, 2006.

\bibitem{suprunenko} Suprunenko,  D. A. and Tyshkevich, R.I.: 
\newblock {\em Commutative matrices.} 
\newblock Academic Press New York London, 1968.

\bibitem{snow} Snow, J. E.: 
\newblock {\em Invariant complex structures on four dimensional solvable real Lie groups.}
\newblock Manuscripta Math. {\bf 66},  397-412 (1990).


\bibitem{sugiura} Sugiura, M.: 
\newblock {\em Conjugate classes of Cartan subalgebras in real semi-simple Lie algebras.}
\newblock J. Math. Soc. Japan {\bf 11}, 374-434 (1959).

\bibitem{winternitz-zassenhaus}
Winternitz, P. and Zassenhaus, H.:  
\newblock {\em The structure of maximal Abelian subalgebras of classical Lie and Jordan algebras. }
\newblock XIIIth International colloquium on group theoretical methods in physics (College Park, Md., 1984), 115-113; World Sci. Publishing, Singapore, 1984. 

\bibitem{winternitz}
Winternitz, P.:  
\newblock {\em Subalgebras of Lie algebras. Example of sl(3,R). }
\newblock Symmetry in physics, 215-227,
CRM Proc. Lecture Notes, 34, Amer. Math. Soc., Providence, RI, 2004.


\end{thebibliography}
\end{document}